\documentclass[reqno]{amsart}
\usepackage[top=1.1in, bottom=1in, left=1in, right=1in]{geometry}
\usepackage{amsfonts}
\usepackage{amssymb}
\usepackage{amsthm}
\usepackage{amsmath}
\usepackage{mathrsfs}
\usepackage{color}
\usepackage{enumerate}
\usepackage[numbers,sort&compress]{natbib}
\usepackage{pdfsync}
\usepackage{esint}
\usepackage{graphicx}
\usepackage{float}
\usepackage{caption}
\usepackage{subfigure}
\usepackage{enumitem}
\usepackage{appendix}

\allowdisplaybreaks

\usepackage{tikz} 
\usetikzlibrary{calc}
\usetikzlibrary{intersections}
\usetikzlibrary{patterns}

\usepackage[colorlinks,
            linkcolor=black,       
            anchorcolor=blue,      
            citecolor=blue,        
            ]{hyperref}

\makeatother

\numberwithin{equation}{section}

\newtheorem{theorem}{Theorem}[section]
\newtheorem{definition}[theorem]{Definition}

\newtheorem{lemma}[theorem]{Lemma}

\newtheorem{proposition}[theorem]{Proposition}

\numberwithin{equation}{section}

\theoremstyle{remark}
\newtheorem{remark}[theorem]{Remark}

\newcommand{\R}{{\mathbb R}}
\newcommand{\lbb}{\overline{\lambda}}

\newcommand*{\supp}{\ensuremath{\mathrm{supp\,}}}

\def\vf{{\varphi}}
\def\lbb{\lambda}
\def\wt{\widetilde}
\def\wh{\widehat}
\def\9{{\infty}}
\def\ve{{\varepsilon}}
\def\na{{\nabla}}
\def\bbe{{\mathbb{E}}}
\def\bbr{{\mathbb{R}}}

\def\bbp{{\mathbb{P}}}
\def\({\left(}
\def\){\right)}
\def\<{\left<}
\def\>{\right>}
\def\ol{\overline}

\newcommand{\imu}{\mathrm{i}}
\newcommand{\dd}{\,\mathrm{d}}
\renewcommand{\epsilon}{\varepsilon}

\renewcommand{\Im}{\operatorname{Im}}
\renewcommand{\Re}{\operatorname{Re}}

\newcommand{\N}{{\mathbb{N}}}

\newcommand{\C}{{\mathbb{C}}}
\newcommand{\Z}{{\mathbb{Z}}}

\newcommand{\G}{{\mathbb{G}}}

\newcommand{\X}{{\mathbb{X}}}
\newcommand{\Y}{{\mathbb{Y}}}

\newcommand{\cL}{{\mathcal{L}}}
\newcommand{\cF}{{\mathcal{F}}}

\newcommand{\cE}{{\mathcal{E}}}
\newcommand{\cI}{{\mathcal{I}}}

\newcommand{\cP}{{\mathcal{P}}}
\newcommand{\cT}{{\mathcal{T}}}
\newcommand{\cU}{{\mathcal{U}}}
\newcommand{\cV}{{\mathcal{V}}}

\newcommand{\Schw}{\mathcal{S}}

\newcommand{\om}{\omega}
\newcommand{\on}{\overline{n}}

\newcommand{\ep}{\varepsilon}

\newcommand{\bdyop}{\Omega_{\mathrm{b}}}
\newcommand{\resf}{\om_{\mathrm{r}}}

\newcommand{\Sp}{{\mathbb S}}
\newcommand{\vece}{{\textnormal{\textbf{e}}}}

\begin{document}
	
\title[] {The three dimensional
stochastic Zakharov system}

\author{Sebastian Herr}
\address{Fakult\"at f\"ur Mathematik,
Universit\"at Bielefeld, D-33501 Bielefeld, Germany}
\email[Sebastian Herr]{herr@math.uni-bielefeld.de}
\thanks{}

\author{Michael R\"ockner}
\address{Fakult\"at f\"ur Mathematik,
Universit\"at Bielefeld, D-33501 Bielefeld, Germany; 
Academy of Mathematics and Systems Science, CAS, Beijing 100190, China}
\email{roeckner@math.uni-bielefeld.de}
\thanks{}

\author{Martin Spitz}
\address{Fakult\"at f\"ur Mathematik,
Universit\"at Bielefeld, D-33501 Bielefeld, Germany}
\email[Martin Spitz]{mspitz@math.uni-bielefeld.de}
\thanks{}

\author{Deng Zhang}
\address{School of Mathematical Sciences, MOE-LSC,
CMA-Shanghai, Shanghai Jiao Tong University, China.}
\email[Deng Zhang]{dzhang@sjtu.edu.cn}
\thanks{}



\begin{abstract}
   We study the three dimensional stochastic Zakharov system in the energy
space, where the Schr\"odinger equation is driven by linear multiplicative noise and the wave equation is
driven by additive noise. We prove the well-posedness of the system up to the maximal existence time and
provide a blow-up alternative. We further show that the solution exists at least as long as it remains below
the ground state. Two main ingredients of our proof are refined rescaling transformations and the normal
form method. Moreover, in contrast to the deterministic setting, our functional framework also incorporates the local smoothing
estimate for the Schr\"odinger equation in order to control
lower order perturbations arising from the noise. Finally, we prove a regularization by noise result which
states that finite time blowup before any given time can be prevented with high probability by adding sufficiently
large non-conservative noise. The key point of its proof is an estimate in Strichartz spaces for solutions of a Schrödinger type equation with a nonlocal potential involving the free wave.
\end{abstract}

\maketitle



\section{Introduction and main results}   \label{Sec-Intro}

\subsection{The Zakharov system} \label{Subsec-Intro}

We consider the three dimensional stochastic Zakharov system
\begin{equation}   \label{equa-stoZak}
	\left\{\aligned
	 \imu \dd X + \Delta X \dd t &= \Re(Y) X \dd t  - \imu \mu X \dd t + \imu X \dd W_1(t),  \\
	 \frac{1}{\alpha} \imu \dd Y + |\na|Y \dd t &= -|\na| |X|^2 \dd t + \dd W_2(t), \\
      (X(0), Y(0)) &= (X_0, Y_0) \in H^1 \times L^2.
	\endaligned
	\right.
\end{equation}
Here, the wave speed $\alpha>0$ is a fixed constant,
$\{W_j(t), t\geq 0\}$ are independent Wiener processes
white in time and colored in space,
\begin{align*}
   W_j(t,x) = \sum\limits_{k=1}^\infty \imu \phi^{(j)}_k(x) \beta^{(j)}_k(t),
   \ \ x\in \bbr^3,\ \ t\geq 0
\end{align*}
for $j=1,2$, $\{\phi^{(1)}_k\} \subseteq H^3(\bbr^3)$ and  $\{\phi^{(2)}_k\} \subseteq H^1(\bbr^3)$ are real-valued and
satisfy the summability conditions in~\eqref{phik-condition}, and
$\{\beta^{(j)}_k\}$ are real-valued independent Brownian motions
on a stochastic basis $(\Omega, \mathscr{F}, \{\mathscr{F}_t\}, \mathbb{P})$.
The stochastic term $X \dd W_1(t)$ is taken in the sense of It\^o,
and
\begin{align*}
    \mu = \frac 12 \sum\limits_{k=1}^\infty |\phi^{(1)}_k|^2 < \infty.
\end{align*}
In particular,
$-\imu \mu X \dd t + \imu X \dd W_1(t)$ is the usual Stratonovich differential $\imu X \circ  \dd W_1(t)$.

In 1972, Zakharov introduced the quadratically coupled system of a Schr\"odinger and a wave equation~\eqref{eq:ZakharovSystem}
in order to describe rapid oscillations of the electric field (Langmuir waves) in a non- or weakly magnetized plasma, see~\cite{Z72}. We provide more information on the model in Subsection~\ref{subsec:heur} below.

The deterministic Zakharov system has attracted a lot of attention in the literature as its coupling of a Schr\"odinger and a wave equation leads to interesting phenomena and presents various mathematical challenges.
Moreover, the Zakharov system is closely connected to the focusing cubic Schr\"odinger equation which arises in the subsonic limit $\alpha \to \infty$,
see \cite{SW86,OT92,MN08}. We refer to~\cite{BC96, GTV97, CHN20, Sa22} and the references therein for the local well-posedness theory and to~\cite{GNW13,G16,GN21,CHN21} and the references therein for the long-time behavior of the Zakharov system.

In contrast, very little is known about the stochastic Zakharov system. The available results~\cite{T22, GLY09,B22} are all restricted to the one dimensional case and consider the Zakharov system with additive noise.

In this article, we propose to consider the Zakharov system with linear multiplicative noise in the Schr\"odinger equation and additive noise in the wave equation. We provide heuristic arguments for this type of noise from a modeling perspective in Subsection~\ref{subsec:heur} below. Linear multiplicative noise in the Schr\"odinger equation also preserves the conservation of mass
of the Schr\"odinger component
\begin{align}  \label{mass-conserv}
   \|X(t)\|_{L^2}^2 = \|X_0\|_{L^2}^2
\end{align}
from the deterministic setting. Moreover, adding additive noise in the wave equation, at least formally we obtain the cubic focusing Schr\"odinger equation with linear multiplicative noise as the subsonic limit $\alpha \rightarrow \infty$ of the Zakharov system.

The aim of this work is to develop the well-posedness theory
for the three dimensional stochastic Zakharov system \eqref{equa-stoZak}
in the energy space $H^1 \times L^2$.
We prove the well-posedness up to the maximal existence time and
provide a blow-up alternative.
Using the variational characterization of the ground state,
we also show the well-posedness of the stochastic Zakharov system below the ground state.

One major difficulty
in solving the Zakharov system in the energy space is
that the Schr\"odinger component has to be controlled in $H^1$ while the wave component in the nonlinearity $\Re(Y) X$ only belongs to $L^2$.
Moreover, as pointed out in \cite{T22},
since the path of Brownian motions has merely $C^{1/2-}$ regularity, one cannot expect that in the stochastic case
solutions belong to the usual Bourgain space $X^{s,b}$ with $b>1/2$ which was
used in the deterministic setting.

Our proof utilizes a refined rescaling approach and a normal form method, which is new in the stochastic literature. The rescaling approach allows a pathwise treatment of the system so that we can apply a normal form transform which exploits the resonance structure of the nonlinearity in the Schr\"odinger equation and reveals that the loss of derivatives indicated above can be avoided. On the other hand, the rescaling leads to first order perturbation terms originating from the noise.
Our functional framework thus also incorporates local smoothing estimates
in order to control these first order perturbations,
which do not appear in the deterministic setting.
More precisely, in order to exploit the sharp local smoothing effect, we adapt an approach developed in~\cite{BIKT11} for the Schr{\"o}dinger map problem. Furthermore, compared with previous work on stochastic nonlinear Schr\"odinger equations (SNLS for short) \cite{BRZ14,BRZ16,Z17,Z18}, we remark that we can deal with noise of much lower regularity here.

Another delicate fact is
that the control of the first order perturbation requires a small amplitude
of the path of the Brownian motions, which, however, is only possible up to a small enough stopping time.
Consequently, we encounter another difficulty in extending the solution up to the maximal existence time,
as well as in proving a blow-up alternative,
which are the foundation for the long-time analysis.
Our strategy utilizes the refined rescaling transformations
which have been recently developed in \cite{Z18,Z19}
to tackle the global well-posedness problem for SNLS, particularly
in the critical regime.
Rather than performing one rescaling transformation,
we apply a series of rescaling transformations
which incorporate stopping times
constructed iteratively to keep track of the trajectory of the noise.
We can then construct solutions on every small random interval.
Finally, probabilistic arguments involving the H\"older continuity of the noise are used to glue these solutions together up to the maximal existence time.

Our method also permits to investigate the regularization by noise effect on well-posedness. To that purpose, we first note that finite time blowup might occur for the Zakharov system. In fact,
finite time blow-up was proved in~\cite{GM94} in dimension $d=2$ and conjectured in~\cite{M96} for $d=3$.

We prove that, with high probability,
non-conservative noise can prevent blow-up on bounded time intervals
for any $H^1 \times H^1$ initial data,
hence even for initial data above the ground state constraint
which may lead to finite time blow-up in the deterministic case.
Its proof uses a new type of rescaling transformation,
which reveals a geometrical Brownian motion
which decays exponentially fast at infinity and dampens the nonlinearity.
The key point is then to prove Strichartz estimates
for the Schr\"odinger operator with a potential solving the free wave equation.

\medskip

Before formulating our main results,
we first review the literature for
both the deterministic and stochastic Zakharov system.

\paragraph{\bf The deterministic Zakharov system}

The deterministic Zakharov system~\eqref{eq:ZakharovSystemFirstOrder} is a Hamiltonian system and conserves the mass
\begin{align}  \label{Mass-def}
   M(u) := \frac 12 \int |u|^2 \dd x ,
\end{align}
and the Hamiltonian
\begin{align} \label{EZ-def}
   E_Z(u,v) := \int \frac{|\na u|^2}{2} + \frac{|v|^2}{4} + \frac{{\rm Re \,} v |u|^2}{2} \dd x .
\end{align}

The conservation laws suggest the energy space $H^1\times L^2$
as a natural state space for the well-posedness of the Zakharov system \eqref{eq:ZakharovSystemFirstOrder}. In dimensions $d=2,3$ local well-posedness in the energy space and global well-posedness under a smallness condition
were proved by Bourgain-Colliander \cite{BC96}, see also Ginibre-Tsutsumi-Velo \cite{GTV97}, which covers the case $d = 1$ as well.
There has been extensive research on the question of determining the optimal range of regularity parameters $(s,l)$ such that the Zakharov system is locally well-posed in $H^s \times H^l$. A comprehensive answer has recently be given by Candy-Herr-Nakanishi~\cite{CHN20} in dimensions $d \geq 4$ and by Sanwal~\cite{Sa22} in dimensions $d \leq 3$. We refer to~\cite{CHN20, Sa22} for previous work on and the historic development of that problem.

As mentioned above,
the Zakharov system \eqref{eq:ZakharovSystemFirstOrder}
is closely related to the focusing cubic
nonlinear Schr\"odinger equation (NLS)
\begin{align}  \label{equa-NLS-cubic}
   \imu \partial_t u + \Delta u = - |u|^2 u,
\end{align}
since the latter arises as the subsonic limit $\alpha\to \infty$ of
the Zakharov system, see~\cite{SW86,OT92,MN08}.
The close connection between the Zakharov system and the cubic NLS also appears in the Hamiltonian
\begin{align}  \label{EZ-Hamil}
   E_Z(u,v) = E_S(u) + \frac 14 \| v +  |u|^2\|_{L^2}^2,
\end{align}
where $E_S(u)$ is the Hamiltonian for the focusing cubic NLS, i.e.,
\begin{align}   \label{ES-Hamil}
   E_S(u) := \int \frac{|\na u|^2}{2} - \frac{|u|^4}{4} \dd x.
\end{align}

This close relationship indicates the difficulty in studying the
global behavior
of solutions to the Zakharov system.
In fact,
in the 1D case
the cubic nonlinearity in the NLS is $L^2$-subcritical
with respect to scaling
and hence \eqref{equa-NLS-cubic} is globally well-posed in $L^2$.
But in the 2D case
the cubic nonlinearity becomes $L^2$-critical and
there exist solutions (e.g., the pseudo-conformal blow-up solutions)
which form a singularity in finite time.
For higher dimensions,
the cubic nonlinearity is even more singular
and becomes $H^1$-critical when $d=4$.

In the seminal work \cite{KM06},
Kenig-Merle proved that
the ground state is the sharp threshold for
global well-posedness, scattering and blowup
for the radial energy-critical focusing NLS in dimensions $3 \leq d \leq 5$.
In the 3D case,
the corresponding ground state $Q$
is the unique positive radial solution of the nonlinear elliptic equation
\begin{align}    \label{Ground-def}
     - \Delta Q + Q = Q^3.
\end{align}

The Kenig-Merle approach can also be used for the Zakharov system.
In dimension $d = 3$,
the scattering behavior for radial solutions with small energy
was proved by Guo-Nakanishi~\cite{GN14} and for radial solutions below the ground state by
Guo-Nakanishi-Wang \cite{GNW13}.
For the non-radial case with additional conditions
we refer to \cite{G16,GLNW14,HPS13, S22}.

In the energy-critical dimension $d = 4$,
global well-posedness and scattering for radial initial data
below the ground state
was proved by Guo-Nakanishi \cite{GN21}.
For the non-radial case,
the small-data global well-posedness and scattering
were proved in \cite{BGHN15}.
Recently, the global well-posedness below the ground state
was shown by Candy, Nakanishi
and the first author \cite{CHN21}.
The scattering for general non-radial data below the ground state
remains a challenging open problem.
We refer to the recent progress  by Candy \cite{C22} on the existence of a minimal energy
almost periodic solution if the scattering fails.
\medskip

\paragraph{\bf The stochastic case}
Much less is known in the stochastic case. The Zakharov system with additive noise
on a bounded interval with zero
Dirichlet boundary condition was studied in \cite{GLY09}
based on the Galerkin method.
Very recently,
Tsutsumi~\cite{T22} proved the global well-posedness for
the 1D Zakharov system driven by additive noise,
using the Fourier restriction method. Finally, the subsonic limit of the 1D Zakharov system with additive noise in the wave equation was studied by Barru{\'e}~\cite{B22}.

The stochastic Zakharov system with multiplicative noise, the well-posedness in the energy space and the stochastic Zakharov system in higher dimensions have not been studied before.

At least formally, in the subsonic limit $\alpha\to\infty$
the stochastic Zakharov system \eqref{equa-stoZak}
 reduces to the focusing SNLS
 driven by linear multiplicative noise
\begin{align} \label{equa-SNLS}
    \imu \dd X + \Delta X \dd t = -|X|^2 X \dd t  - \imu \mu X \dd t + \imu X \dd W(t),
\end{align}
where $W= W_1 - \imu |\na|^{-1} W_2$.
In the case where $W_2\equiv 0$,
the $H^1$ local well-posedness of \eqref{equa-SNLS} is well known,
see, e.g., \cite{BD03,BM14,BRZ16}.

The global existence and large time behavior of solutions to SNLS
are quite delicate.
In the case $d=2$,
the cubic nonlinearity is $L^2$-critical.
On the one hand,
the stochastic solutions exist globally below the ground state~\cite{M19}.
Furthermore,
critical mass blow-up solutions have recently been constructed in~\cite{SZ19} using the modulation method.
Hence, the ground state is still the sharp threshold for the global well-posedness and blowup
when $d=2$.

We also refer to \cite{FSZ20}
for the construction of stochastic log-log blow-up solutions,
and to \cite{SZ20,RSZ21,RSZ21.2} for the construction and conditional uniqueness
of multi-bubble Bourgain-Wang type blow-up solutions and (non-pure) multi-solitons,
which provide new examples for the mass quantization conjecture
and the soliton resolution conjecture.
For the refined uniqueness of pure multi-solitons
in the low asymptotical regime,
related to an open problem raised by Martel \cite{M18},
we refer to \cite{CSZ21}.

\subsection{Main results} \label{Subsec-Main}

We first present
the hypothesis for the spatial functions in the noise,
which also relates to the lateral spaces we introduce in Subsection~\ref{Subsec-Funct} in order to capture  the local smoothing estimates.

\paragraph{\bf Hypothesis (H)}
The spatial functions $\{\phi^{(j)}_k\}$, $j=1,2$, satisfy
the following summability conditions:
   \begin{align} \label{phik-condition}
   \sum\limits_{k=1}^\infty \|\phi^{(1)}_k\|_{H^3}^2 +
   \sum\limits_{j=1}^3  \sum\limits_{k=1}^\infty \int  \sup_{y\in \bbr^2} | \na \phi^{(1)}_k(r \vece_j+y)|\dd r  <\infty, \ \
   \sum\limits_{k=1}^\infty \|\phi^{(2)}_k\|_{H^1}^2 <\infty,
\end{align}
where $\vece_1, \vece_2, \vece_3$ denotes the natural orthonormal basis of $\bbr^3$.

It is standard that under the summability \eqref{phik-condition},
$W_1$ and $W_2$ are Wiener processes in $H^3$ and $H^1$, respectively.
See, e.g., \cite{PR07}.

We present the definition of solutions to \eqref{equa-stoZak},
which are taken in the probabilistic strong and
analytically weak sense. Since the wave speed $\alpha$ is a fixed constant in the well-posedness theory of~\eqref{equa-stoZak}, we set $\alpha = 1$ for convenience.
\begin{definition}   \label{Def-weak-sol}
Fix $T\in (0,\infty)$.
We say that $(X,Y)$ is a probabilistic strong solution to \eqref{equa-stoZak}
on $[0,\tau]$,
where $\tau\in(0,T]$ is an $\{\mathscr{F}_t\}$-stopping time,
if $(X,Y)$ is an $H^1\times L^2$-valued
$\{\mathscr{F}_t\}$-adapted process which belongs to $C([0,\tau],H^1 \times L^2)$
and satisfies $\mathbb{P}$-a.s.
for any $t\in [0,\tau]$,
\begin{equation}   \label{equa-stoZak-def}
	\left\{\aligned
	 X(t) &= \int_0^t \imu \Delta X \dd s - \int_0^t  \imu \Re(Y) X \dd s - \int_0^t \mu X \dd s + \int_0^t X \dd W_1(s),    \\
	Y(t) &= \int_0^t \imu |\na|Y \dd s + \int_0^t \imu |\na| |X|^2 \dd s - \imu W_2(t),
	\endaligned
	\right.
\end{equation}
as equations in $H^{-1} \times H^{-1}$.

Given an $\{\mathscr{F}_t\}$-stopping time $\tau^*$, we also call $(X,Y)$ a probabilistic strong solution to \eqref{equa-stoZak}
on $[0,\tau^*)$ if $(X,Y)$ is an $\{\mathscr{F}_t\}$-adapted process belonging to $C([0,\tau^*),H^1 \times L^2)$ such that for any $T \in (0,\infty)$ and any $\{\mathscr{F}_t\}$-stopping time $\tau < \tau^*$, $(X,Y)$ is a probabilistic strong solution to \eqref{equa-stoZak} on $[0,\tau \wedge T]$.
\end{definition}

The first main result of this paper provides the local well-posedness of~\eqref{equa-stoZak} and a blow-up alternative.

\begin{theorem} [Well-posedness up to maximal existence time]      \label{Thm-LWP}
	 Assume (H). Let $(X_0,Y_0)\in H^1\times L^2$.
	 Then, $\mathbb{P}$-a.s. there exist an $\{\mathscr{F}_t\}$-stopping time $\tau^*$
	 and a unique $\{\mathscr{F}_t\}$-adapted solution $(X,Y)\in C([0,\tau^*); H^1\times L^2)$
	 to the stochastic Zakharov system \eqref{equa-stoZak}
	 in the sense of Definition \ref{Def-weak-sol}.

	 Moreover,  $\bbp$-a.s.
we have either $\tau^*=\infty$, or
\begin{align} \label{blowup-alter}
		\lim_{t\to \tau^*} (\|X(t)\|_{H^1} + \|Y(t)\|_{L^2}) = \infty \ \ \text{if}\ \tau^*<\infty.
\end{align}
\end{theorem}

\begin{remark}
	In the statement of Theorem~\ref{Thm-LWP} uniqueness means that for any $T \in (0,\infty)$ and any $\{\mathscr{F}_t\}$-stopping time $\tau < \tau^*$, the process $(X,Y)$ is the unique solution of~\eqref{equa-stoZak} in
	\begin{align*}
		(C([0,\tau \wedge T];H^1) \cap \X(0,\tau \wedge T)) \times C([0,\tau \wedge T];L^2)
	\end{align*}
	in the sense of Definition~\ref{Def-weak-sol}, where the function space $\X$ is introduced in~\eqref{eq:DefXYG}.
\end{remark}

\begin{remark}
\begin{enumerate}[leftmargin=*]
	\item To the best of our knowledge, Theorem~\ref{Thm-LWP} provides the first well-posedness result for the stochastic Zakharov system with multiplicative noise, in the physical dimension three, and in the energy space.
	\item The proof makes use of the refined rescaling approach, recently developed for the critical stochastic nonlinear Schr\"odinger equations, and the normal form method, employing both Strichartz and local smoothing estimates. We explain the key ideas of the proof in Subsection~\ref{subsec:key} below. In particular, our strategy also applies in the context of stochastic nonlinear Schr\"odinger equations, allowing to treat noise with less regularity than in the previous papers \cite{BRZ14,BRZ16,Z18,Z19},
as well as the case of infinite modes of noise.
\end{enumerate}
\end{remark}

Next, we consider the well-posedness below the ground state,
which is known as the sharp threshold for global well-posedness
in the deterministic case.

We recall from \cite[(2.9)]{GNW13} that the ground state minimizes the action functional
\begin{align} \label{J-Q-mini}
   J(Q) = \inf  \{ J(\vf): \vf \not = 0, K(\vf) =0 \},
\end{align}
where
\begin{align}   \label{J-def}
   J(Q):= E_S(Q) + M(Q)
\end{align}
with $E_S$ being the Hamiltonian for the cubic focusing NLS in \eqref{ES-Hamil},
$M$ the mass functional given by \eqref{Mass-def}, and
$K$ stands for the scaling derivative of the action $J$, defined by
\begin{align}  \label{K-def}
   K(\vf)
    : = \partial_\lbb |_{\lbb= 1} J(\lbb^{\frac 32} \vf(\lbb x))
      = \int |\na \vf|^2 - \frac 3 4 |\vf|^4 \dd x.
\end{align}

\begin{theorem} [Well-posedness below the ground state]  \label{Thm-GWP-Ground}
Assume (H). Let $(X_0, Y_0) \in H^1 \times L^2$
satisfy $E_Z(X_0,Y_0) M(X_0) < E_S(Q) M(Q)$.
Let $(X,Y)$ be the corresponding unique solution to \eqref{equa-stoZak} on $[0,\tau^*)$ from Theorem~\ref{Thm-LWP},
where $\tau^*$ is the maximal existence time.
Define the $\{\mathscr{F}_t\}$-stopping time
\begin{align} \label{sigma*-def}
   \sigma_* := \inf\{t>0: E_Z(X(t),Y(t)) M(X(t)) \geq E_S(Q) M(Q)\}.
\end{align}
Then, if $K(X_0)\geq 0$,
we have  $\tau^*\geq \sigma_*$, $\bbp$-a.s.,
that is,
$(X,Y)$ exists at least up to the stopping time $\sigma_*$.
\end{theorem}

\begin{remark}
In the deterministic case the conservation of mass and energy shows that $\sigma_* = \infty$
if \linebreak $E_Z(u_0,v_0) M(u_0) < E_S(Q) M(Q)$, implying global existence below the ground state
(cf. \cite{GNW13}).
However, the conservation of the energy breaks down
in the stochastic case due to the presence of the It\^o type stochastic integrals.
In view of the possible blow-up phenomena above the ground state in the deterministic case,
we would expect $\sigma_*<\infty$.
This motivates the study of the regularization by noise effect on well-posedness below.
\end{remark}

It is widely expected in the probability community that
suitable noise may improve the well-posedness and long-time behavior of
deterministic systems, which is called the \emph{regularization by noise} effect.
We refer to \cite{FGP10} where regularization by noise was proved for
uniqueness of transport equations,
and to \cite{FL21} for the regularization effect of transport noise
on vorticity blow-up in 3D Navier-Stokes equations.

For stochastic Schr\"odinger equations,
the numerical experiments in \cite{DL02,DL02.2} show that
conservative smooth multiplicative noise can delay blowup,
while white noise can even prevent blowup.
In the case of non-conservative noise,
the noise is able to prevent blow-up with high probability \cite{BRZ14.3},
and it can also improve the scattering behavior of solutions \cite{HRZ18}.

Regarding the Zakharov system,
the existence of finite time blow-up solutions was proven in~\cite{GM94} in dimension $d = 2$
and conjectured in \cite{M96} in dimension $d = 3$.
Numerical evidence for finite time blow-up solutions was found in \cite{LPSSW92,PSSW91}.
In Theorem~\ref{Thm-Noise-Reg} below,
we show the regularization effect of non-conservative noise on the blow-up behavior of solutions
of the stochastic Zakharov system.

\begin{theorem} [Noise regularization on well-posedness]   \label{Thm-Noise-Reg}
Consider the stochastic Zakharov system \eqref{equa-stoZak},
where the driving noise $W_1$ is a one-dimensional Brownian motion
with non-zero imaginary part,
i.e., $\phi_1^{(1)}$ is a constant, $\Im \phi_1^{(1)} \not = 0$,
$\phi_k^{(1)} = 0$ for $2\leq k<\infty$,
and $W_2 \equiv 0$.
Then, for any deterministic initial data
$(X_0, Y_0) \in H^1 \times L^2$ and any $0<T<\infty$,
we have
\begin{align}
\bbp\(X(t) \text{ does not blow up on } [0,T]\) \longrightarrow 1,\ \ \text{as} \ \phi_1^{(1)}\to \infty.
\end{align}
\end{theorem}

\begin{remark}
\begin{enumerate}[leftmargin=*]
\item Unlike in Theorems \ref{Thm-LWP} and \ref{Thm-GWP-Ground},
the noise considered in Theorem \ref{Thm-Noise-Reg} is of {\it non-conservative type},
i.e., $\Im W^{(1)} \not = 0$.
The non-conservative noise permits to define
the so-called ``physical probability law''
which has important applications
in quantum measurements,
see, e.g., the monograph \cite{BG09} for the physical context.
Theorem \ref{Thm-Noise-Reg} shows that even for deterministic data
which leads to a singularity in finite time,
the one-dimensional non-conservative noise
can prevent blow-up on bounded time intervals with high probability,
as long as the noise is strong enough.

\item This regularization by noise effect
relies on the observation
that the non-conservative noise gives rise to a geometrical Brownian motion
in the nonlinearity
which has exponential decay at infinity and thus is able to weaken the nonlinearity.
The key step in order to explore this effect
is to prove estimates in Strichartz spaces for a Schr\"odinger type equation with a nonlocal potential involving the free wave $v_1:=e^{\imu t|\na|}Y_0$ and a boundary term operator that comes from the normal form transform, see~\eqref{equa-z-Omegav2z-mild} below for the precise formulation.

The global uniform Strichartz estimate for the Schrödinger operator with free wave potential \linebreak $\partial_t - \imu \Delta + \imu v_1$ is the key ingredient in \cite{CHN21} in proving the global well-posedness below the ground state for the deterministic problem in dimension $d = 4$.
Here, for solutions of the Schrödinger part of~\eqref{equa-z-Omegav2z-mild}, we prove estimates in Strichartz spaces on bounded time intervals using perturbation arguments,
which suffices for the finite time analysis.
\end{enumerate}
\end{remark}

\paragraph{\bf Organization of the paper}
In Section~\ref{sec:mot}, a heuristic derivation of the stochastic Zakharov system and the strategy of the proof is discussed.
Section~\ref{Sec-DS-transf} is mainly concerned with the refined rescaling transformations.
We prove the equivalence of solving the stochastic and random Zakharov systems,
and also show how solutions can be glued together.
Section~\ref{Sec-Normal-Form} is devoted to the normal form transform. In Section~\ref{Sec-Multi-Esti}, we provide the functional framework for the proof of Theorem~\ref{Thm-LWP} and provide corresponding multi-linear estimates.
In Section~\ref{Sec-WP-Max-Time} and Section~\ref{Sec-WP-Ground} we give the proofs of
Theorems \ref{Thm-LWP} and \ref{Thm-GWP-Ground},
respectively.
In Section~\ref{Sec-Noise-Reg}
we prove Theorem~\ref{Thm-Noise-Reg} concerning the noise regularization effect on well-posedness.
Finally, we prove the equivalence of different solution concepts we use in this article, i.e. weak solutions, mild solutions, and normal form solutions, in
Appendix~\ref{App-Weak-Mild} and Appendix~\ref{App-Mild-Normal}.

\section{Motivation}\label{sec:mot}

\subsection{Heuristic derivation of the stochastic Zakharov system}  \label{subsec:heur}
We recall that the Zakharov system was introduced in~\cite{Z72} as a model in plasma physics to describe rapid oscillations of the electric field in a non- or weakly magnetized plasma.
A more formal derivation using multiple-scale modulational analysis is presented in~\cite[Chapter 13]{SS99}, see also~\cite{T07,CEGT07}.

The starting point for the model is to consider the plasma as two interpenetrating fluids, the ion fluid and the electron fluid, each described by a set of hydrodynamic equations. Since the fluids consist of charged particles, they interact with the electromagnetic fields in the plasma, which evolve via Maxwell's equations. Consequently, the two-fluid model is described by an Euler-Maxwell system, where Euler's and Maxwell's equations are coupled via the current density (given as the sum of the products of charge, number density, and velocity field of the fluids) in the Maxwell equations, while the gradient of the pressure term in the Euler equations is complemented by the Lorentz force, see~\cite{SS99}. Considering long wavelength small-amplitude Langmuir waves, one can reduce the rather complex Euler-Maxwell system to the Zakharov system, which reads in its scalar version
 \begin{equation}   \label{eq:ZakharovSystem}
	\left\{\aligned
	 \imu \partial_t u + \Delta u &= V u, \\
     \frac{1}{\alpha^2}\partial_t^2 V -  \Delta V &=  \Delta |u|^2.
	\endaligned
	\right.
\end{equation}
We refer again to~\cite[Chapter~13]{SS99} for this derivation, see also~\cite{T07}. In~\eqref{eq:ZakharovSystem}, $u \colon \R \times \R^3 \rightarrow \C$ describes the complex envelope of the electric field,
$V \colon \R \times \R^3 \rightarrow \R$ is the ion density fluctuation,
and the fixed constant $\alpha > 0$ is called the ion sound speed.

To be more precise, making the modulational ansatz
\begin{align*}
	E &= \frac{\ep}{2} (\cE(\hat{t},\tilde{x}) e^{-\imu \omega_e t} + \textrm{c.c.}) + \ep^2 \overline{\cE}(\hat{t},\tilde{x}) + \ldots, \\
	n_i &= n_0 + \frac{\ep^2}{2} (\tilde{n}_i(\tilde{t},\tilde{x}) e^{-\imu \omega_e t} + \textrm{c.c.}) + \ep^2 \on_i(\tilde{t},\tilde{x}) + \ldots
\end{align*}
with $\tilde{x} = \ep x$, $\tilde{t} = \epsilon t$, $\hat{t} = \epsilon^2 t$, and $\omega_e$ being the plasma frequency as in~\cite{SS99}, $u$ represents (one component of) the amplitude $\cE$ and $V$ the mean ion density fluctuation $\on_i$. Note that in the above ansatz it is assumed that the unperturbed plasma density $n_0$ is known. This leads to the question how the model is affected if there is insufficient information on or internal random fluctuations in $n_0$. Taking into account this uncertainty and replacing $n_0$ by $n_0 + \epsilon^2 \dot{W}_1(\tilde{t},\tilde{x})$, it turns out that we have to replace $V$ by $V + \dot{W}_1$ on the right-hand side of the Schr{\"o}dinger equation in~\eqref{eq:ZakharovSystem}. We obtain
\begin{align*}
		\imu \partial_t u + \Delta u &= (V + \dot{W}_1) u = V u + u \dot{W}_1,
\end{align*}
suggesting to consider the Schr{\"o}dinger part of the Zakharov system with linear multiplicative noise.

Another source of uncertainty is the plasma temperature. In the long-wavelength limit it is assumed that the compression of the wave is adiabatic, so that we obtain for the pressure on the right-hand side of the Euler equation for the ion fluid $\nabla p_i = \gamma_i T_i \nabla n_i$, where $\gamma_i$ denotes the specific heat ratio of the ions and $T_i$ the ion temperature. Modeling thermal fluctuations by $T_i + \epsilon^2 c(\tilde{x}) \dot{\beta}(\tilde{t})$, we obtain for the right-hand side of the wave equation in~\eqref{eq:ZakharovSystem} $\Delta |u|^2 + \Delta c \dot{\beta}$. We refer to~\cite{SS99} for the details.

Summing up, we propose to model thermal fluctuations in the plasma by additive noise in the wave equation, leading to
\begin{align}
\label{eq:ZakSecOrderStoch}
	\left\{\aligned
	 \imu \partial_t u + \Delta u &= V u + \dot{W}_1 u, \\
     \frac{1}{\alpha^2}\partial_t^2 V -  \Delta V &=  \Delta |u|^2 + \dot{W}_2,
	\endaligned
	\right.
\end{align}
with space-time noise $\dot{W}_1, \dot{W}_2$.

Finally, we note that the Zakharov system has an equivalent first-order formulation obtained by setting $v:=V-\imu \frac 1 \alpha |\na|^{-1} \partial_t V$. In the deterministic case this transformation leads to
\begin{equation}  \label{eq:ZakharovSystemFirstOrder}
        \left\{\aligned
         \imu \partial_t u + \Delta u &= \Re(v) u, \\
         \frac 1 \alpha \imu \partial_t v + |\nabla| v &= - |\nabla| |u|^2,
\endaligned
\right.
 \end{equation}
while we get
 \begin{equation}  \label{eq:ZakharovSystemStochFirstOrder}
        \left\{\aligned
         \imu \partial_t u + \Delta u &= \Re(v) u + \dot{W}_1 u, \\
         \frac 1 \alpha \imu \partial_t v + |\nabla| v &= - |\nabla| |u|^2 + |\na|^{-1} \dot{W}_2,
\endaligned
\right.
 \end{equation}
 in the stochastic case. Writing $\dot{W}_2$ for the noise in the wave equation again and $(X,Y)$ instead of $(u,v)$, we arrive at~\eqref{equa-stoZak} as a stochastic model for the Zakharov system.

\subsection{Key ideas of the proof}\label{subsec:key}

In the following we present three main ingredients of the proofs of our main theorems.

\medskip

{\bf (i) Refined rescaling transformations.}
The rescaling transformation is a Doss-Sussman type transformation
first developed for
finite dimensional stochastic differential equations,
and now has been applied successfully to various infinite dimensional stochastic equations.
We refer to \cite{ABD21} for the stochastic Camassa-Holm equation
and to \cite{BRZ16.1,BRZ18,DF12, Z17,Z19} for the SNLS.
One main feature of the rescaling transformation
is
that it permits to
transform the original stochastic equation to a random equation,
and thus allows sharp analysis in a {\it pathwise} way,
which is difficult for the usual stochastic analysis.
In particular,
the resulting solutions depend continuously on the initial data pathwisely
and satisfy the strict cocycle property. This leads to stochastic dynamical systems which are favorable for the analysis
of large time dynamics, including the scattering and blow-up behavior and the study of solitons.
We refer to \cite{HRZ18,RSZ21,RSZ21.2,SZ19,SZ20,Z18} for corresponding references.

Unlike in \cite{BRZ14,BRZ16} for the Schr\"odinger equation,
we perform two different rescaling transforms in order to respect the different types of noise in the stochastic Zakharov system~\eqref{equa-stoZak}.
More precisely,
for the Schr\"odinger component,
we use the  transform
\begin{align}  \label{rescal}
     u = e^{-W_1}X,
\end{align}
and for the wave component we  use the transform
\begin{align}
	v(t):= Y(t) - \cT_{t}(W_2) \ \
	\text{with} \ \cT_t(W_2):= - \imu \int_0^t e^{\imu (t-s)|\na|} \dd W_2(s).
\end{align}

Then,
the original stochastic Zakharov system \eqref{equa-stoZak}
can be reduced to the random system
\begin{equation}   \label{equa-ranZak}
	\left\{\aligned
	  \imu \partial_t u + e^{-W_1}\Delta (e^{W_1} u) &= \Re(v) u + \Re(\cT_t(W_2)) u,  \\
	  \imu \partial_t v + |\na |v  &= - |\na||u|^2, \\
    (u(0), v(0)) &= (X_0, Y_0),
	\endaligned
	\right.
\end{equation}
or, equivalently, the system with random perturbations
\begin{equation}   \label{equa-ranZak-bc}
	\left\{\aligned
	  \imu \partial_t u + \Delta u
	   &= \Re(v) u - b\cdot \na u - cu + \Re(\cT_t(W_2)) u,  \\
	 \imu \partial_t v + |\na |v  &= - |\na||u|^2, \\
    (u(0), v(0)) &= (X_0, Y_0),
	\endaligned
	\right.
\end{equation}
where the random coefficients of the lower order perturbations are defined by
\begin{align}
   b &= 2 \na W_1 = 2 \imu \sum\limits_{k=1}^\infty \na \phi^{(1)}_k \beta^{(1)}_k,  \label{b-def}  \\
   c &= |\na W_1|^2 + \Delta W_1
     = - \sum\limits_{j=1}^{3} \(\sum\limits_{k=1}^\infty
	     \partial_j \phi_k^{(1)} \beta^{(1)}_k \)^2
        + \imu \sum\limits_{k=1}^\infty \Delta \phi^{(1)}_k \beta^{(1)}_k.   \label{c-def}
\end{align}
The solutions to the random Zakharov system \eqref{equa-ranZak}
(or, \eqref{equa-ranZak-bc})
are taken in the analytically weak sense,
analogous to the stochastic case \eqref{equa-stoZak-def} in Definition \ref{Def-weak-sol} above.

The relation between the two systems~\eqref{equa-stoZak} and~\eqref{equa-ranZak-bc}
can be seen by a formal application of It\^o's formula.
However, the rigorous justification of the application of It\^o's formula
is non-trivial in the infinite dimensional case,
due to the failure of the $C_b^2$-regularity of the nonlinearity
(see also \cite{BRZ14} for detailed explanations in the context of SNLS).
The rigorous proof for the equivalence of~\eqref{equa-stoZak} and~\eqref{equa-ranZak-bc}
is given in Section~\ref{Sec-DS-transf}.

In order to deal with the first order perturbation, we will employ local smoothing estimates for the Schr\"odinger equation (see (ii) below). In the resulting functional framework the deterministic estimates for the first order perturbation do not gain a power of the length of the time interval. In order to extend solutions up to the maximal existence time, we instead exploit that the variation of the noise is small on small time intervals.
This is done by means of
the refined rescaling transformations, 
recently developed in the context of critical stochastic nonlinear Schr\"odinger equations \cite{Z18}, where the deterministic estimates also do not gain a power of the length of the time interval.

More precisely,
for any  stopping time $\sigma$
we define the increments of the noise by
\begin{align*}
    W_{1,\sigma}(t) &:= W_1(\sigma+t) - W_1(\sigma),    \\
    \cT_{\sigma+t, \sigma} (W_2)
     &:= - \imu \int_\sigma^{\sigma+t} e^{\imu(\sigma+t-s)|\na|} \dd W_2(s).
\end{align*}
Using the transformations
\begin{align*}
	 u _\sigma(t) &:= e^{W_1(\sigma)} u(\sigma+t), \\
     v_\sigma(t) &:= v(\sigma+t) + e^{\imu t|\na|}\cT_\sigma(W_2),
\end{align*}
where $0\leq t\leq \tau$ with $\tau$ being an $\{\mathscr{F}_{\sigma+t}\}$-stopping time,
we transform the random Zakharov system \eqref{equa-ranZak-bc} to
the new system
\begin{equation}   \label{equa-Zakha-sigma-intro}
	\left\{\aligned
	  \partial_t u_\sigma(t) &=  \imu e^{-W_{1,\sigma}(t)} \Delta (e^{W_{1,\sigma}(t)} u_\sigma(t))
	- \imu \Re(v_\sigma(t)) u_\sigma(t) 
	- \imu \Re( \cT_{\sigma+t, \sigma}(W_2)) u_\sigma(t),   \\
	  \partial_t v_\sigma(t) &=   \imu  |\na| v_\sigma(t) + \imu |\na| |u_\sigma(t)|^2.
	\endaligned
	\right.
\end{equation}
Note that~\eqref{equa-Zakha-sigma-intro} reduces to~\eqref{equa-ranZak-bc} if $\sigma \equiv 0$.

The functional framework we set up below will require the $L^\infty$-norm in time of the noise to be small on small time intervals. This is ensured by the refined rescaling transform, since $W_{1,\sigma}(0) = 0$ and $W_{1,\sigma}$ is continuous.
In the proof we will iteratively construct sequences $(\sigma_n)_{n \in \N}$ and $(\tau_n)_{n \in \N}$ of $\{\mathscr{F}_{t}\}$ and $\{\mathscr{F}_{\sigma_n+t}\}$ stopping times, respectively,
in order to split the trajectory of the noise.
On each small random interval $[\sigma_n, \sigma_n+\tau_{n+1}]$ we can then construct 
the local solution.

Probabilistic arguments involving the H\"older continuity of the noise
will then be applied to glue together all these local solutions, providing the solution
up to the maximal existence time
and the blow-up alternative.
The detailed arguments are contained in Section~\ref{Sec-DS-transf}.

\medskip

{\bf (ii) Normal form method and local smoothing.}
Solving the random system~\eqref{equa-ranZak-bc} in the energy space $H^1 \times L^2$, one of the main difficulties is the regularity of the nonlinearity. At first sight, there seems to be a derivative loss as we have to control the Schr{\"o}dinger component in $H^1$, but in the nonlinearity $\Re(v) u$ the wave component only belongs to $L^2$. This problem already appears in the deterministic system~\eqref{eq:ZakharovSystemFirstOrder}, and one tool to overcome it is the so called normal form transform.

The normal form method was originally introduced by Shatah~\cite{S85} to analyze the long time behavior of quadratic nonlinear Klein-Gordon equations. The idea was to transform the equation in such a way that the order of the nonlinearity increases, as higher order nonlinearities are favorable for the long time analysis using perturbation techniques. Nowadays, the normal form method is widely used to study both the long time behavior and the low regularity well-posedness of dispersive PDEs.

For the Zakharov system~\eqref{eq:ZakharovSystemFirstOrder}, the normal form transform was introduced by Guo and Nakanishi in~\cite{GN14} in order to show scattering for small, radial initial data in the energy space. It was subsequently used to advance the understanding of both the low regularity well-posedness and the long-time behavior, see~\cite{BGHN15,GNW13,GLNW14,G16,S22}.

The underlying idea is to use the resonance structure of the system to reveal subtle features of the nonlinear interactions. To be more precise, an integration by parts argument on the Fourier side transforms the equation in such a way that only low-high and high-high interactions of the quadratic nonlinearity remain. Consequently, derivatives only fall on the Schr{\"o}dinger component of the quadratic nonlinearity, showing that a loss of derivatives in the nonlinearity $\Re(v)u$ can be avoided.
We refer to Section~\ref{Sec-Normal-Form} for the details, where we perform the normal form transform for the random system~\eqref{equa-ranZak-bc}.

We point out that the previously discussed rescaling transforms, allowing us to analyze the stochastic Zakharov system~\eqref{equa-stoZak} pathwisely via the random system~\eqref{equa-ranZak-bc}, are the reason that we can employ analytical tools such as the normal form method, ultimately giving us access to the rather low regularity setting of the energy space.

For later reference we note that~\eqref{equa-ranZak-bc} after the normal form transform reads
\begin{equation*}
	\left\{\aligned
	(\partial_t-\imu \Delta) (u+\bdyop(v,u))&=
	   -\imu (vu)_{R} + \imu (b\cdot \na u + cu - \cT_t(W_2)u)
	   + \imu \bdyop(|\na||u|^2, u)
	   - \imu \bdyop(v,vu) \\
	   &\qquad + \imu \bdyop(v,b\cdot \na u + cu),  \\
	(\partial_t - \imu |\na|) v &= \imu |\na||u|^2,
	\endaligned
	\right.
\end{equation*}
where $\bdyop$ is a bilinear Fourier multiplier operator related to quadratic high-low interactions and $(vu)_R$ contains the possibly resonant interactions of $vu$,
see~\eqref{eq:DefR} and~\eqref{Omegab-def} below for the precise definitions. When working with the normal form transform, we also drop the real parts in~\eqref{equa-ranZak-bc} for convenience, see Remark~\ref{rem:DroppingRealPart}.

Another regularity issue is posed by the first order perturbation $b \cdot \na u$
on the right-hand side of~\eqref{equa-ranZak-bc} as $\nabla u$ only belongs to $L^2$
but $u$ must be controlled in $H^1$. This problem is new since these lower order perturbations,
generated by the noise, do not appear in the deterministic setting.
We deal with the $b \cdot \na u$ term by employing local smoothing estimates for the Schr{\"o}dinger equation.
To make optimal use of the local smoothing effect,
we adapt an approach developed in~\cite{BIKT11} for the Schr{\"o}dinger maps problem,
setting up a functional framework combining Strichartz and local smoothing estimates,
see Subsection~\ref{Subsec-Funct} below.

\medskip

{\bf (iii) Further reduction for noise regularization.}
The structural difference between the stochastic Zakharov systems driven by conservative
or non-conservative noise is revealed by the rescaling transform.

For the non-conservative noise in Theorem \ref{Thm-Noise-Reg},
we introduce a new rescaling transformation by
\begin{equation}
	\left\{\aligned
	 z &:= e^{\wh \mu t - W_1(t)} X,  \\
	 v &:= Y,
	\endaligned
	\right.
\end{equation}
where
\begin{align}
	\wh \mu := \frac 12 (|\phi_1^{(1)}|^2 - (\phi_1^{(1)})^2).
\end{align}
Note that
in the current non-conservative case we have
\begin{align}  \label{wtmu-Imphi2}
	 \Re \wh \mu
	 =  (\Im \phi_1^{(1)})^2 >0,
\end{align}
while $\Re \wh \mu = 0$ in the conservative case as in Theorem~\ref{Thm-LWP} and Theorem~\ref{Thm-GWP-Ground}.

The stochastic Zakharov system \eqref{equa-stoZak} can then be reduced to
\begin{equation}   \label{equa-ranZak-noncons}
	\left\{\aligned
	  \imu \partial_t z + \Delta z &= \Re(v)z,  \\
	 \imu \partial_t v + |\na |v  &= - h|\na||z|^2, \\
    (u(0), v(0)) &= (X_0, Y_0) \in H^1 \times H^1
	\endaligned
	\right.
\end{equation}
with $h$ being the geometric Brownian motion
\begin{align}  \label{h-W1-def}
	h(t):= e^{2 \Re (W_1(t)- \wh \mu t)}
    = e^{-2 \Im \phi_1^{(1)} \beta^{(1)}_1(t) - 2 (\Im \phi_1^{(1)})^2 t},
\end{align}
or equivalently in the normal form formulation
\begin{equation}   \label{equa-Zak-norm-noncons}
	\left\{\aligned
	  (\partial_t - \imu \Delta) ( z + \bdyop(v,z))
	&= - \imu (vz)_R  + \imu \bdyop(h|\na||z|^2, z)
	- \imu \bdyop (v,vz),  \\
	 (\partial_t - \imu |\na|) v &= \imu h|\na||z|^2.
	\endaligned
	\right.
\end{equation}

The key fact here is that,
because of the iterated logarithmic law of Brownian motion
\begin{align} \label{log-BM}
   \limsup\limits_{t\to \infty} \frac{\beta^{(1)}_1(t)}{\sqrt{2t \log\log t}} =1, \ \
    \liminf\limits_{t\to \infty} \frac{\beta^{(1)}_1(t)}{\sqrt{2t \log\log t}} = -1, \ \ \bbp\text{-a.s.},
\end{align}
the geometrical Brownian motion $h$ decays exponentially fast at infinity.
Heuristically,
we expect that it weakens the nonlinearity in the wave equation
and thus stabilizes the system.

In order to explore this damping effect,
it is important to
extract the free wave from
the Schr\"odinger equation.
More precisely, we
set
\begin{align} \label{eq:Decv1v2}
	v_1(t):= e^{\imu t|\na|} Y_0, \ \
	v_2:= v-v_1,
\end{align}
write~\eqref{equa-Zak-norm-noncons} in its mild formulation, and separate all free wave interactions which do not involve $v_2$. Writing $\cU$ and $\cI$ for the homogeneous and inhomogeneous flow operators of the free Schrödinger equation, respectively,  we obtain
\begin{equation}   \label{equa-z-Omegav2z-mild}
	\left\{\aligned
	 z(t) + \bdyop(v_1, z)(t) + \imu \cI_t(v_1 z)_R + \imu \cI_t \bdyop(v_1, v_1 z) &= \cU_t f - \bdyop(v_2, z)(t) - \imu \cI_t(v_2 z)_R + \imu \cI_t \bdyop(h |\nabla| |z|^2, z)\\
	&\qquad  - \imu \cI_t \bdyop(v_2, (v_1 + v_2) z) - \imu \cI_t \bdyop(v_1, v_2 z),  \\
	 v_2(t) &= \imu \int_0^t e^{\imu (t -s) |\nabla|} (h|\na||z|^2)(s) \dd s,
	\endaligned
	\right.
\end{equation}
where $f = X_0 + \bdyop(Y_0, X_0)$.
The key point is then to control solutions of the Schr\"odinger type equation with potential $v_1$ appearing on the left-hand side of the first equation in~\eqref{equa-z-Omegav2z-mild} in Strichartz spaces,
which is done in Lemma~\ref{Lem-Est-v1Pot}.

\section{Refined rescaling transformations}  \label{Sec-DS-transf}

In this section we perform the refined rescaling transformations,
which convert the original stochastic system into a random system,
allowing us to apply the normal form method in order to overcome the analytical difficulties inherent to the Zakharov system in the next section.

We start with the equivalence of the solvability of the two
systems~\eqref{equa-stoZak} and~\eqref{equa-ranZak}.

\subsection{Equivalence via rescaling transformations}

\begin{theorem} [Equivalence of the solvability via rescaling transformations]  \label{Thm-stocha-weak}

\begin{enumerate}
\item[]
\item \label{it:EquivRescStochToRandom} Let $(X,Y)$ be a solution to \eqref{equa-stoZak} on $[0,\tau]$
in the sense of Definition~\ref{Def-weak-sol},
where $\tau$ is an $\{\mathscr{F}_t\}$-stopping time
and $(X,Y) \in C([0,\tau]; H^1 \times L^2)$ $\bbp$-a.s.
Set $u: =e^{-W_1}X$ and $v:= Y - \cT_t(W_2)$.
Then,
$(u,v)$ is an analytically weak solution to~\eqref{equa-ranZak} on $[0,\tau]$
as equations in $H^{-1} \times H^{-1}$.

\item \label{it:EquivRescRandomToStoch} Let $(u,v)$ be an analytically weak solution to \eqref{equa-ranZak} on $[0,\tau]$
as equations in $H^{-1} \times H^{-1}$,
where $\tau$ is an $\{\mathscr{F}_t\}$-stopping time,
and $(u,v)$ is $\{\mathscr{F}_t\}$-adapted and continuous in $H^1\times L^2$.
Set $(X, Y): =(e^{W_1}u, v+\cT_t(W_2))$.
Then,
$(X, Y)$ is a solution of~\eqref{equa-stoZak} on $[0,\tau]$
in the sense of Definition~\ref{Def-weak-sol}.
\end{enumerate}
\end{theorem}

\begin{proof}
It is not difficult to prove the statements for the wave component,
as $\cT_t(W_2)$ satisfies the equation
\begin{align}  \label{equa-cTW2}
	\dd \cT_t(W_2) = \imu |\na| \cT_t(W_2) \dd t - \imu \dd W_2(t).
\end{align}
Hence, we focus on the Schr\"odinger component $u$ and
prove the assertion in~\ref{it:EquivRescStochToRandom} as the one in~\ref{it:EquivRescRandomToStoch} can be proved similarly.

Let $\{e_j\}_{j=1}^\infty \subseteq \mathcal{S}$ be an orthonormal basis of $L^2$.
Let $J_\ve:= (I - \ve \Delta)^{-1}$,
$f_\ve := J_\ve (f)$ for any $f\in \mathcal{S}'$,
and set $\tau_M:= \inf\{t\in[0,T]: \|X(t)\|_{L^2}^2 > M\} \wedge \tau$,
$M\in \mathbb{N}$, with $T > 0$ from Definition~\ref{Def-weak-sol}.
Note that,
for $\bbp$-a.e. $\omega \in \Omega$, there exists $M$ large enough such that
$\tau(\omega) = \tau_M(\omega)$ since $X(\omega)\in C([0,\tau(\omega)]; H^1)$.
This yields that
\begin{align} \label{tauW-tau}
   \bigcup_{M\in \mathbb{N}} \{t\leq \tau_M(\omega)\}
   = \{t\leq \tau(\omega)\}
\end{align}
for $\bbp$-a.e. $\omega \in \Omega$.
For the sake of brevity, we omit the dependence on $\omega$ below.

Fix any $M\in \mathbb{N}$.
We apply the operator $J_\ve$ to both sides of the Schr\"odinger equation in  \eqref{equa-stoZak}
to get
\begin{align*}
   & \dd X_\ve = \imu \Delta X_\ve \dd t - \imu (X \Re Y)_\ve \dd t - (\mu X)_\ve \dd t + \imu \sum\limits_{k=1}^\infty (X\phi^{(1)}_k)_\ve \dd \beta^{(1)}_k(t),
   \ \ t\in [0,\tau_M],
\end{align*}
with $ X_\ve(0) = (X_0)_\ve$.
Pairing the above equation in $L^2$ with $e_j$,
we then obtain that
\begin{align} \label{equa-Xve-ej}
    \<X_\ve(t), e_j\>
    = \<(X_0)_\ve, e_j\>
      + \int_0^t \<\imu \Delta X_\ve - \imu (X \Re Y)_\ve - (\mu X)_\ve, e_j\> \dd s
      + \sum\limits_{k=1}^\infty \int_0^t \<\imu (X\phi^{(1)}_k)_\ve, e_j\>  \dd \beta^{(1)}_k(s)
\end{align}
for every $j \geq 1$ and $t \in [0, \tau_M]$.
Moreover, by It\^o's calculus we have
\begin{align*}
   e^{W_1(t)} = 1  - \int_0^t \mu e^{W_1(s)} \dd s
   + \sum\limits_{k=1}^\infty
     \int_0^t \imu \phi^{(1)}_k e^{W_1(s)} \dd \beta^{(1)}_k(s),
   \ \ t\in[0,\tau_M].
\end{align*}
This yields that for any $\vf \in \mathcal{S}$ and any $j\geq 1$
\begin{align} \label{equa-eW-ej}
   \<e_j, e^{W_1(t)} \vf\>
   = \<e_j, \vf\>
     - \int_0^t \<e_j, \mu e^{W_1} \vf\> \dd s
      +  \sum\limits_{k=1}^\infty \int_0^t \<e_j, \imu \phi^{(1)}_k e^{W_1} \vf\> \dd \beta^{(1)}_k(s),
       \ \ t\in[0,\tau_M],
\end{align}
where we also interchanged the summation over $k$ and the inner product,
which is justified due to
\begin{align*}
   \mathbb{E} \sup\limits_{t\in[0,\tau_M]}
    \bigg|\sum\limits_{k=1}^\infty \int_0^{t} \<e_j, \phi^{(1)}_ke^{W_1}\> \dd \beta^{(1)}_k \bigg|^2
  \lesssim   \mathbb{E} \sum\limits_{k=1}^\infty  \int_0^{\tau_M}
            \left| \<e_j, \phi^{(1)}_k e^{W_1}\>\right|^2 \dd s
  \lesssim \tau_M \|e_j\|_{L^2}^2
            \sum\limits_{k=1}^\infty \|\phi^{(1)}_k\|_{L^2}^2 <\infty.
\end{align*}

Hence, applying the product rule for scalar valued processes
and using \eqref{equa-Xve-ej} and \eqref{equa-eW-ej},
we compute for $t\in [0,\tau_M]$,
\begin{align}   \label{equa-Xve-ej-eW}
      \<X_\ve(t), e_j\> \<e_j, e^{W_1(t)} \vf\>
   =& \<(X_0)_\ve, e_j\> \<e_j, \vf\>
      + \int_0^t \<\imu\Delta X_\ve - \imu (X \Re Y)_\ve - (\mu X)_\ve, e_j\> \<e_j, e^{W_1} \vf\> \dd s \notag \\
    & - \int_0^t \<X_\ve, e_j\> \<e_j, \mu e^{W_1} \vf\> \dd s
    + \sum\limits_{k=1}^\infty \int_0^t \<\imu(X\phi^{(1)}_k)_\ve, e_j\> \<e_j, \imu\phi^{(1)}_k  e^{W_1} \vf\> \dd s   \notag  \\
    & + \sum\limits_{k=1}^\infty \int_0^t \<\imu(X\phi^{(1)}_k)_\ve, e_j\> \<e_j, e^{W_1} \vf\> \dd \beta^{(1)}_k(s)  \notag \\
    & +  \sum\limits_{k=1}^\infty \int_0^t \<X_\ve, e_j\> \<e_j, \imu \phi^{(1)}_k  e^{W_1} \vf\> \dd \beta^{(1)}_k(s).
\end{align}

Next, we sum both sides over $j$ in order
to derive the evolution formula of
the process  $\<X_\ve(t), e^{W_1(t)} \vf\>$.
Note that
the regularity
$(X,Y)\in C([0,\tau_M]; H^1) \times C([0,\tau_M]; L^2)$
and the summability $\{\|\phi^{(1)}_k\|_{L^\infty}\}\in \ell^2$,
implied by \eqref{phik-condition} and the Sobolev embedding $H^3\hookrightarrow L^\infty$ in three dimensions,
suffice to justify the application of Fubini's theorem
to interchange the sum with integration.
We take one stochastic integration as an example
to illustrate the arguments.
Actually,
by the Burkholder-Davis-Gundy inequality
and the Cauchy inequality,
\begin{align*}
   &\mathbb{E} \sup\limits_{s\in [0,\tau_M\wedge t]}
   \bigg|\sum\limits_{j=1}^\infty  \sum\limits_{k=1}^\infty  \int_0^{s}
    \<\imu (X\phi^{(1)}_k)_\ve, e_j\> \<e_j, e^{W_1} \vf\> \dd \beta^{(1)}_k(r) \bigg|^2 \notag \\
   &\lesssim \mathbb{E} \sum\limits_{j=1}^\infty  \sum\limits_{k=1}^\infty  \int_0^{\tau_M\wedge t}
   \bigg|\<\imu (X\phi^{(1)}_k)_\ve, e_j\> \<e_j, e^{W_1} \vf\>\bigg|^2 \dd s  \notag\\
   &\lesssim \sum\limits_{k=1}^\infty
   \mathbb{E} \int_0^{\tau_M \wedge t}  \|(X\phi^{(1)}_k)_\ve\|_{L^2}^2
             \|e^{W_1} \vf\|_{L^2}^2 \dd s  \notag \\
   &\lesssim t M \|\vf\|_{L^2}^2
               \sum\limits_{k=1}^\infty \|\phi^{(1)}_k\|_{L^\infty}^2
  <\infty,
\end{align*}
where we also used $|e^{W_1}|=1$
and
\begin{align*}
   \sup_{s\in[0,\tau_M\wedge t]}\|(X(s) \phi^{(1)}_k)_\ve\|_{L^2}^2
   \leq \sup_{s\in [0,\tau_M \wedge t]} \|X(s)\|_{L^2}^2 \|\phi^{(1)}_k\|_{L^\infty}^2
   \leq M  \|\phi^{(1)}_k\|_{L^\infty}^2
\end{align*}
in the last step.
We thus obtain
\begin{align*}
    \sum\limits_{j=1}^\infty \sum\limits_{k=1}^\infty \int_0^t \<\imu (X\phi^{(1)}_k)_\ve, e_j\> \<e_j, e^{W_1} \vf\> \dd \beta^{(1)}_k(s)
    =& \sum\limits_{k=1}^\infty \int_0^t  \sum\limits_{j=1}^\infty \<\imu (X\phi^{(1)}_k)_\ve, e_j\> \<e_j, e^{W_1} \vf\> \dd \beta^{(1)}_k(s)  \notag \\
    =&  \sum\limits_{k=1}^\infty \int_0^t  \<\imu (X\phi^{(1)}_k)_\ve,  e^{W_1} \vf\> \dd \beta^{(1)}_k(s)
\end{align*}
for all $t\in [0,\tau_M]$. The other terms in~\eqref{equa-Xve-ej-eW}
can be treated in a similar manner.

Hence, interchanging the summation over $j$
with the integration in time in~\eqref{equa-Xve-ej-eW}, we infer that
\begin{align}  \label{equa-Xve-eW}
   \<X_\ve(t), e^{W_1(t)} \vf\>
   =& \sum\limits_{j=1}^\infty
   \<X_\ve(t), e_j\> \<e_j, e^{W_1(t)} \vf\> \notag \\
   =& \<(X_0)_\ve, \vf\>
      + \int_0^t \<\imu \Delta X_\ve - \imu (X \Re Y)_\ve - (\mu X)_\ve, e^{W_1} \vf\> \dd s  \notag \\
    & - \int_0^t \<X_\ve, \mu e^{W_1} \vf\> \dd s
      + \sum\limits_{k=1}^\infty  \int_0^t \<(X\phi^{(1)}_k)_\ve, \phi^{(1)}_k e^{W_1}  \vf\> \dd s \notag \\
     & + \sum\limits_{k=1}^\infty \int_0^t \<\imu (X\phi^{(1)}_k)_\ve, e^{W_1} \vf\>  \dd \beta^{(1)}_k(s)
      -  \sum\limits_{k=1}^\infty  \int_0^t \<\imu \phi^{(1)}_k X_\ve, e^{W_1} \vf\>  \dd \beta^{(1)}_k(s)
\end{align}
for all $t\in [0,\tau_M]$.
Next we pass to the limit $\ve\to 0$ in \eqref{equa-Xve-eW} above.
Note that
the operator $J_\ve$ is contractive on the spaces $H^k$
and
$f_\ve \to f$ in $H^k$ as $\epsilon \rightarrow 0$ for every $k=-1,0,1$.
Taking into account the regularity of $(X, Y)$
and the $\ell^2$-summability of $\{\|\phi^{(1)}_k\|_{L^\infty}\}$,
we can interchange the limit with the sum and the integral.
For instance, with the Burkholder-Davis-Gundy inequality we get
for the stochastic integration
\begin{align*}
  & \mathbb{E} \sup\limits_{s\in[0,\tau_M\wedge t]} \bigg| \sum\limits_{k=1}^\infty \int_0^{\tau_M \wedge t}
    \<\imu (X\phi^{(1)}_k)_\ve - \imu X\phi^{(1)}_k, e^{W_1} \vf\> \dd \beta^{(1)}_k(r) \bigg|^2  \\
  &\lesssim \mathbb{E} \int_0^{\tau_M \wedge t}  \sum\limits_{k=1}^\infty
  \bigg|\<\imu (X\phi^{(1)}_k)_\ve - \imu X\phi^{(1)}_k, e^{W_1} \vf\> \bigg|^2 \dd s \\
  &\lesssim \| \vf\|_{L^2}^2
  \mathbb{E} \int_0^{\tau_M \wedge t}  \sum\limits_{k=1}^\infty
  \|(X\phi^{(1)}_k)_\ve - X\phi^{(1)}_k\|_{L^2}^2  \dd s.
\end{align*}
Since
$ \|(X\phi^{(1)}_k)_\ve - X\phi^{(1)}_k\|_{L^2}^2 \to 0$, $\dd \bbp \otimes \dd t$-a.e.
and $\sup_{s\in[0,\tau_M]} \|(X(s)\phi^{(1)}_k)_\ve\|_{L^2}^2 \leq  M \|\phi^{(1)}_k\|_{L^\infty}^2
\in L^1_\Omega L^1_t\ell_k^1$,
the dominated convergence theorem yields that
the above right-hand side converges to zero as $\ve\to 0$,
and thus there is a null sequence $(\ve_n)_{n \in \N}$ such that
\begin{align*}
   \sum\limits_{k=1}^\infty \int_0^t \<\imu (X\phi^{(1)}_k)_{\ve_n} , e^{W_1} \vf\> \dd \beta^{(1)}_k(s)
   \to  \sum\limits_{k=1}^\infty \int_0^t \<\imu X\phi^{(1)}_k , e^{W_1} \vf\> \dd \beta^{(1)}_k(s),
   \ \ {\rm as}\ n \to \infty,\ t\in [0,\tau_M],\ \bbp-a.s.
\end{align*}
Analogous arguments allow us to pass to the limit for the other terms in \eqref{equa-Xve-eW}
and thus we arrive at
\begin{align*}
   _{H^{-1}}\< X(t), e^{W_1(t)} \vf \>_{H^1}
   = {}_{H^{-1}}\<X_0, \vf\>_{H^{1}}
      + \int_0^t {}_{H^{-1}}\<\imu \Delta X - \imu X \Re Y, e^{W_1} \vf\>_{H^1} \dd s,
      \ \ t\in [0,\tau_M].
\end{align*}
Taking into account $\ol{e^{W_1}} = e^{-W_1}$
and interchanging the integration with the duality pairing,
we conclude
\begin{align*}
   {}_{H^{-1}}\<u(t), \vf\>_{H^1}
   =& {}_{H^{-1}}\<X_0 + \int_0^t  \imu e^{-W_1} \Delta (e^{W_1} u) - \imu u \Re Y \dd s , \vf\>_{H^1}  \notag \\
   =& {}_{H^{-1}}\<X_0
   + \int_0^t  \imu e^{-W_1}\Delta (e^{W_1} u) - \imu u \Re v - \imu u \Re \cT_s(W_2) \dd s, \vf\>_{H^1}
\end{align*}
for any $\vf\in \mathcal{S}$.
But
since $\mathcal{S}$ is dense in $H^{1}$,
it follows that that $u$ solves
the Schr\"odinger equation in~\eqref{equa-ranZak} in the space $H^{-1}$
on $[0,\tau_M]$, $\bbp$-a.s.

Finally, we deduce from~\eqref{tauW-tau} that
$u$ solves the Schr\"odinger equation \eqref{equa-ranZak} on the whole interval $[0,\tau]$, $\bbp$-a.s.,
completing the proof.
\end{proof}

\subsection{Refined rescaling transformations}

In the proof of the wellposedness theorem we will extend a solution of~\eqref{equa-ranZak} up to its maximal existence time. This requires to solve the system~\eqref{equa-ranZak} also on intervals $[\sigma, \sigma+\tau]$ away from zero.
In Proposition~\ref{Prop-usigma-vsigma-equa} below we use refined rescaling transforms in order to show that this problem can be reduced to solving an appropriate system on $[0,\tau]$.

\begin{proposition} [Refined rescaling transformations] \label{Prop-usigma-vsigma-equa}
Let $\sigma, \tau \colon \Omega \rightarrow [0,T]$ such that $\sigma+\tau\leq T$.
\begin{enumerate}
\item \label{it:RefResc0ToSig}
Let $(u_\sigma, v_\sigma) \in C([0,\tau]; H^1 \times L^2)$
be an analytically weak solution to the system
\begin{equation}   \label{equa-Zakha-sigma}
	\left\{\aligned
	  \partial_t u_\sigma(t) &=  \imu e^{-W_{1,\sigma}(t)} \Delta (e^{W_{1,\sigma}(t)} u_\sigma(t))
	- \imu u_\sigma(t) \Re v_\sigma(t)
	- \imu u_\sigma(t) \Re \cT_{\sigma+t, \sigma}(W_2),   \\
	  \partial_t v_\sigma(t) &=   \imu  |\na| v_\sigma(t) + \imu |\na| |u_\sigma(t)|^2,
	\endaligned
	\right.
\end{equation}
as equations in $H^{-1} \times H^{-1}$,
where $W_{1,\sigma}$ and $ \cT_{\sigma+t, \sigma} (W_2)$ are increments of the noise defined by
\begin{align}
   & W_{1,\sigma}(t):= W_1(\sigma+t) - W_1(\sigma),   \label{W-sigma-rescal}  \\
   & \cT_{\sigma+t, \sigma} (W_2)
     := -\imu \int_\sigma^{\sigma+t} e^{\imu (\sigma+t-s)|\na|} \dd W_2(s)  \label{Tsigmat-W2}
\end{align}
for all $t\in [0,\tau]$.
For any $t\in [\sigma, \sigma + \tau]$, we set
\begin{align}
	 u(t)
	&:= e^{-W_1(\sigma)} u _\sigma(t-\sigma),    \label{u-sigma-rescal} \\
     v(t) &:= v_\sigma(t-\sigma)- e^{\imu (t-\sigma) |\na|}\cT_\sigma(W_2).    \label{v-sigma-rescal}
\end{align}
Then, $(u, v)$ is an analytically weak solution of the system \eqref{equa-ranZak}
on $[\sigma, \sigma+\tau]$
with
\begin{align}  \label{usigma-vsigma-initial}
   (u(\sigma), v(\sigma)) = (e^{-W_1(\sigma)} u_\sigma(0), v_\sigma(0)-\cT_\sigma(W_2)).
\end{align}

\item \label{it:RefRescSigTo0} Conversely,
if $(u,v) \in C([\sigma, \sigma+\tau]; H^1\times L^2)$
is an analytically weak solution of the system \eqref{equa-ranZak} on $[\sigma, \sigma+\tau]$
as equations in $H^{-1}\times H^{-1}$,
then
\begin{align}
	\label{eq:u-sigma-v-sigma-rescal}
    ( u_\sigma(t), v_\sigma(t))
    :=( e^{W_1(\sigma)} u(\sigma+t), v(\sigma+t) + e^{\imu t|\na|}\cT_\sigma(W_2)),
    \ \ t\in [0,\tau],
\end{align}
is an analytically weak solution of the system \eqref{equa-Zakha-sigma} on $[0,\tau]$.
\end{enumerate}
\end{proposition}

\begin{proof}
\ref{it:RefResc0ToSig} Let us first consider the wave component.
By \eqref{v-sigma-rescal} and the equation for $v_\sigma$ in \eqref{equa-Zakha-sigma},
we get for any $t\in [\sigma, \sigma+\tau]$
\begin{align}  \label{v-vsigma-esti.1}
	v(t)
	&= v_\sigma(t-\sigma) - e^{\imu (t-\sigma)|\na|} \cT_\sigma(W_2)  \notag \\
    &= v_\sigma(0)
	   + \int_0^{t-\sigma} \imu |\na| v_\sigma(s) \dd s
	   + \int_0^{t-\sigma} \imu |\na| |u_\sigma(s)|^2 \dd s
	   - e^{\imu (t-\sigma)|\na|} \cT_\sigma(W_2).
\end{align}
Using \eqref{v-sigma-rescal} and
a change of variables, we derive
\begin{align}  \label{int-vsigma-esti}
	\int_0^{t-\sigma} \imu |\na| v_\sigma(s) \dd s
	&= \int_0^{t-\sigma} \imu |\na| (v(s+\sigma)+e^{\imu s|\na|}\cT_\sigma(W_2)) \dd s \notag \\
	&= \int_\sigma^t \imu |\na| (v(s) + e^{\imu (s-\sigma)|\na|}\cT_\sigma(W_2)) \dd s \notag \\
	&= \int_\sigma^t \imu |\na| v(s) \dd s
	   + (e^{\imu (t-\sigma)|\na|}-1) \cT_\sigma(W_2).
\end{align}
Moreover, we obtain from \eqref{u-sigma-rescal}
and $|e^{W_1}|=1$ that
\begin{align} \label{int-usigma-esti}
	\int_0^{t-\sigma} \imu |\na| |u_\sigma(s)|^2 \dd s
    =\int_0^{t-\sigma} \imu |\na| |u(\sigma+s)|^2 \dd s
	= \int_\sigma^t \imu |\na| |u(s)|^2 \dd s.
\end{align}
Hence, plugging \eqref{int-vsigma-esti} and \eqref{int-usigma-esti} into \eqref{v-vsigma-esti.1},
we conclude
\begin{align*}
	v(t)
    &= v_\sigma(0) - \cT_\sigma(W_2)
      +  \int_\sigma^t i|\na| v(s) \dd s
	+ \int_\sigma^t i|\na| |u(s)|^2 \dd s \notag \\
    &= v(\sigma) +  \int_\sigma^t i|\na| v(s)
	    +   i|\na| |u(s)|^2 \dd s,
\end{align*}
which yields that $v$ satisfies the wave equation in~\eqref{equa-ranZak} on $[\sigma, \sigma + \tau]$.

Regarding the Schr\"odinger component,
we infer from~\eqref{u-sigma-rescal} and the equation for $u_\sigma$ in~\eqref{equa-Zakha-sigma}
that
\begin{align*}
      u(\sigma+t)
   &= e^{-W_1(\sigma)} u_\sigma(t)  \\
   &=  e^{-W_1(\sigma)} u_\sigma(0)
      + e^{-W_1(\sigma)} \int_0^t \imu e^{-W_{1,\sigma}(s)}
	  \Delta ( e^{W_{1,\sigma}(s)} u_\sigma(s))
           - \imu u_\sigma(s)
		   \Re (v_\sigma(s) + \cT_{\sigma+s,\sigma}(W_2)) \dd s.
\end{align*}
 for any $t\in [0,\tau]$. The definition of $W_{1,\sigma}$ in~\eqref{W-sigma-rescal} then yields
\begin{align*}
   u(\sigma+t)
  &= e^{-W_1(\sigma)} u_\sigma(0)
      +\int_0^t \imu e^{-W_1(\sigma+s)} \Delta ( e^{W_1(\sigma+s)} e^{-W_1(\sigma)}u_\sigma(s)) \\
          &\hspace{14em}  - \imu e^{-W_1(\sigma)} u_\sigma(s)
		   \Re(v_\sigma(s) + \cT_{\sigma+s,\sigma}(W_2)) \dd s.
\end{align*}
Employing~\eqref{u-sigma-rescal} again,
the identity
\begin{align*}
	v_\sigma(s)+ \cT_{\sigma+s,\sigma}(W_2)
	&= v(\sigma+s) + e^{is|\na|} \cT_\sigma(W_2)
	+ \cT_{\sigma+s, \sigma}(W_2)  \notag \\
	&= v(\sigma+s) + \cT_{\sigma+s}(W_2),
\end{align*}
and a change of variables,
we infer
\begin{align*}
    u(\sigma+t)
  &= e^{-W_1(\sigma)} u_\sigma(0)
      +\int_0^t \imu e^{-W_1(\sigma+s)} \Delta ( e^{W_1(\sigma+s)}  u(\sigma+s))
           - \imu  u(\sigma+s)
		   \Re(v(\sigma+s) + \cT_{\sigma+s}(W_2)) \dd s \\
  &=  e^{-W_1(\sigma)} u_\sigma(0)
      +\int_\sigma^{\sigma+t} \imu e^{-W_1(s)} \Delta ( e^{W_1(s)}  u(s))
           - \imu   u(s) \Re v(s)
		   - \imu u(s)\Re \cT_s(W_2)\dd s
\end{align*}
as an equation in $H^{-1}$ for all $t\in [0,\tau]$.
We conclude that $u$ is an analytically weak solution of the Schr\"odinger equation in~\eqref{equa-ranZak}
on $[\sigma, \sigma+\tau]$.

\ref{it:RefRescSigTo0}
Using~\eqref{eq:u-sigma-v-sigma-rescal} and
the equation for the wave component $v$ in \eqref{equa-ranZak} on $[\sigma, \sigma+\tau]$,
we obtain
\begin{align*}
	v_\sigma(t)
	&= v(\sigma) + \int_\sigma^{\sigma+t} \imu |\na|v(s) \dd s
	 + \int_\sigma^{\sigma+t} \imu |\na||u(s)|^2 \dd s
	 + e^{\imu t|\na|} \cT_\sigma (W_2)
\end{align*}
for all $t\in [0,\tau]$.
Employing~\eqref{eq:u-sigma-v-sigma-rescal} once more, we derive with a change of variables
\begin{align*}
	\int_\sigma^{\sigma+t} \imu |\na|v(s) \dd s
	&= \int_\sigma^{\sigma+t} \imu |\na|(v_\sigma(s-\sigma)
	   - e^{\imu (s-\sigma)|\na|}\cT_\sigma(W_2)) \dd s  \\
	&= \int_0^{t} \imu |\na| v_\sigma(s) \dd s
	- \int_0^{t} \imu |\na| e^{\imu s|\na|}\cT_\sigma(W_2)) \dd s  \\
	&=  \int_0^{t} \imu |\na| v_\sigma(s) \dd s
	  - e^{\imu t|\na|} \cT_\sigma(W_2) + \cT_\sigma(W_2).
\end{align*}
Moreover,
since $|e^{W_1(\sigma)}|=1$,
\eqref{eq:u-sigma-v-sigma-rescal} and a change of variables also yield
\begin{align*}
	\int_\sigma^{\sigma+t} \imu |\na||u(s)|^2 \dd s
	= \int_0^t \imu |\na| |u(\sigma+s)|^2 \dd s
	= \int_0^t \imu |\na| |u_\sigma(s)|^2 \dd s.
\end{align*}
Combining the last three identities, we thus arrive at
\begin{align*}
	v_\sigma(t)
	&= v(\sigma) + \cT_\sigma(W_2)
	   +  \int_0^{t} \imu |\na| v_\sigma(s) \dd s
	   + \int_0^t \imu |\na| |u_\sigma(s)|^2 \dd s \\
	&= v_\sigma(0)
	   +  \int_0^{t} \imu |\na| v_\sigma(s)  \dd s
	   + \imu |\na| |u_\sigma(s)|^2 \dd s
\end{align*}
for all $t\in [0,\tau]$.
Consequently, $v_\sigma$ satisfies the wave equation in \eqref{equa-Zakha-sigma}
on $[0,\tau]$.

Concerning the Schr\"odinger component,
we first note that the equation for $u$ in~\eqref{equa-ranZak} on $[\sigma, \sigma+\tau]$
implies
\begin{align*}
  u_\sigma(t)
   &= e^{W_1(\sigma)} u(\sigma+t) \\
   &= e^{W_1(\sigma)} u(\sigma)
     +  e^{W_1(\sigma)} \int_\sigma^{\sigma+t} \imu e^{-W_1(s)}
	 \Delta(e^{W_1(s)} u(s)) - \imu u(s) \Re v(s)
	 - \imu u(s) \Re \cT_s(W_2) \dd s \\
   &= u_\sigma(0)
     +  e^{W_1(\sigma)} \int_0^{t} \imu e^{-W_1(\sigma+s)} \Delta(e^{W_1(\sigma+s)} u(\sigma+s))
            - \imu  u(\sigma+s) \Re v(\sigma+s) \\
		&\hspace{28em}	- \imu u(\sigma+s) \Re \cT_{\sigma+s}(W_2) \dd s
\end{align*}
for all $t\in [0,\tau]$.
Employing~\eqref{W-sigma-rescal} and~\eqref{eq:u-sigma-v-sigma-rescal}, we thus obtain
\begin{align*}
  u_\sigma(t)
   &= u_\sigma(0)
     +   \int_0^{t} \imu e^{-(W_1(\sigma+s) - W_1(\sigma))}
          \Delta(e^{W_1(\sigma+s)- W_1(\sigma)} e^{W_1(\sigma)} u(\sigma+s)) \dd s  \notag \\
    & \qquad \  - \int_0^{t} \imu  e^{W_1(\sigma)}u(\sigma+s)
			\Re(v(\sigma+s) + \cT_{\sigma+s}(W_2)) \dd s \\
   &=  u_\sigma(0)
     + \int_0^t \imu e^{-W_{1,\sigma}(s)} \Delta (e^{W_{1,\sigma}} u_\sigma(s))
            - \imu u_\sigma(s)
			\Re(v_\sigma(s) - e^{\imu s|\na|}\cT_\sigma(W_2) + \cT_{\sigma+s}(W_2)) \dd s
\end{align*}
for any $t\in [0,\tau]$,
where the equation is taken in the space $H^{-1}$.
Taking into account the identity
\begin{align*}
	- e^{\imu s|\na|}\cT_\sigma(W_2) + \cT_{\sigma+s}(W_2)
	= \cT_{\sigma+s, \sigma}(W_2),
\end{align*}
we arrive at
\begin{align}
	u_\sigma(t)
	= u_\sigma(0)
	+ \int_0^t \imu e^{-W_{1,\sigma}(s)} \Delta (e^{W_{1,\sigma}(s)} u_\sigma(s))
		   - \imu u_\sigma(s) \Re v_\sigma(s)
		   - \imu u_\sigma(s) \Re \cT_{\sigma+s, \sigma}(W_2) \dd s
\end{align}
 for all  $t\in [0,\tau]$. We conclude that $u_\sigma$ solves the Schr\"odinger equation in \eqref{equa-Zakha-sigma}
on $[0,\tau]$, finishing the proof.
\end{proof}

\subsection{Gluing solutions}

When we extend a solution of~\eqref{equa-ranZak} up to its maximal existence time, we have to concatenate the solutions provided by Proposition~\ref{Prop-usigma-vsigma-equa}. The next result explains how we can glue together these solutions.

\begin{proposition} [Gluing solutions] \label{Prop-u1-usigma-glue}
Let $(u_1,v_1)\in C([0,\sigma]; H^1) \times C([0,\sigma]; L^2)$
be an analytically weak solution to \eqref{equa-ranZak} on $[0,\sigma]$,
and let $(u_\sigma,v_\sigma)\in C([0,\tau]; H^1) \times C([0,\tau]; L^2)$
be an analytically weak solution to \eqref{equa-Zakha-sigma} on $[0,\tau]$
with the initial condition
\begin{align*}
   (u_\sigma(0), v_\sigma(0))
   := (e^{W_1(\sigma)} u_1(\sigma), v_1(\sigma) + \cT_\sigma(W_2)).
\end{align*}
For any $t\in [0,\sigma+\tau]$,
we set
\begin{align}
\label{u-u1-usigma}
   u(t):= \begin{cases}
   				u_1(t), \quad &\text{if } t \in [0,\sigma), \\
   				 e^{-W_1(\sigma)}u_\sigma(t-\sigma), &\text{if } t \in [\sigma, \sigma + \tau], 						\end{cases} \qquad
   v(t):= \begin{cases}
   				 v_1(t), \quad &\text{if } t \in [0,\sigma), \\
   				v_\sigma(t-\sigma) - e^{\imu (t-\sigma)|\na|}\cT_\sigma(W_2)), &\text{if } t \in [\sigma, \sigma + \tau].
   				\end{cases}
\end{align}
Then, $(u,v) \in C([0,\sigma+\tau]; H^1\times L^2)$
is an analytically weak solution to the system \eqref{equa-ranZak} on the time
interval $[0,\sigma+\tau]$.
\end{proposition}

\begin{proof}
By \eqref{u-u1-usigma}, $(u,v)$ solves~\eqref{equa-ranZak}
on $[0,\sigma]$.
We thus focus on the case $t\in [\sigma, \sigma+\tau]$ in the following.

Proposition~\ref{Prop-usigma-vsigma-equa}~\ref{it:RefResc0ToSig}
yields that for every $t\in [\sigma, \sigma+\tau]$,
\begin{align*}
   u(t)  = u(\sigma)
    + \int_\sigma^t \imu e^{-W_1(s)} \Delta (e^{W_1(s)} u(s)) - \imu u(s) \Re v(s) - \imu u(s) \Re \mathcal{T}_s(W_2) \dd s.
\end{align*}
Since $u(\sigma) = e^{-W_1(\sigma)} u_\sigma(0) = u_1(\sigma)$
and $u_1$ solves the Schr\"odinger equation in~\eqref{equa-ranZak} on $[0,\sigma]$,
we also have
\begin{align*}
   u(\sigma)
   =  u_1(\sigma)
   = X_0 + \int_0^\sigma \imu e^{-W_1(s)} \Delta (e^{W_1(s)} u(s)) - \imu u(s) \Re v(s) - \imu u(s) \Re\mathcal{T}_s(W_2) \dd s.
\end{align*}
Therefore, combining the previous two equations,
we arrive at
\begin{align*}
    u(t)
   = X_0
    + \int_0^t \imu e^{-W_1(s)} \Delta (e^{W_1(s)} u(s)) - \imu u(s) \Re v(s)  - \imu u(s) \Re \mathcal{T}_s(W_2) \dd s
\end{align*}
for all $t\in [\sigma, \sigma+\tau]$, which yields that $u$ solves the Schr\"odinger equation in~\eqref{equa-ranZak} on $[\sigma, \sigma+\tau]$.

Turning to the wave component,
we apply Proposition~\ref{Prop-usigma-vsigma-equa}~\ref{it:RefResc0ToSig} again
in order to derive that for every $t\in [\sigma, \sigma+\tau]$,
\begin{align*}
   v(t)
   &= v(\sigma) +  \int_\sigma^{t} \imu |\na| v(s) + \imu |\na| |u(s)|^2 \dd s\\
   &= v_\sigma(0) -\cT_\sigma(W_2)
      + \int_\sigma^{t} \imu |\na| v(s) + \imu |\na| |u(s)|^2 \dd s.
\end{align*}
Since $v_1$ solves the wave equation in~\eqref{equa-ranZak} on $[0,\sigma]$, we moreover get
\begin{align*}
  v_\sigma(0) - \cT_\sigma(W_2) = v_1(\sigma)
  = Y_0 + \int_0^\sigma \imu |\na| v(s) + \imu |\na| |u(s)|^2 \dd s,
\end{align*}
which implies
\begin{align*}
   v(t) = Y_0 + \int_0^t \imu |\na||u(s)|^2
       +  \imu |\na| v(s) \dd s
\end{align*}
for all $t\in [\sigma, \sigma+\tau]$. Consequently, $v$ also satisfies the wave equation in~\eqref{equa-ranZak} on $[\sigma, \sigma+\tau]$, completing the proof.
\end{proof}


\section{Normal form reduction}   \label{Sec-Normal-Form}

In this section we perform the normal form transform for the stochastic Zakharov system in the formulation~\eqref{equa-ranZak-bc}. It is based on an integration by parts on the Fourier side exploiting the resonance structure of the Zakharov system. It transforms the system into a form which can be solved by a fixed point argument based on Strichartz and local smoothing estimates.
 	
We closely follow~\cite{GN14}, where the normal form transform for the Zakharov system was introduced. However, we note that in~\cite{GN14} and the subsequent works~\cite{GNW13, GLNW14, G16} both the Schr{\"o}dinger and the wave part of the Zakharov system were transformed. For our purposes here, it is sufficient to only transform the Sch{\"o}dinger part as in~\cite{S22} so that we content ourselves with this simpler form.

We first fix some notation.
Take an even function $\eta_0 \in C_c^\infty(\R)$ such that $0 \leq \eta_0 \leq 1$, $\eta_0(x) = 1$ for $|x| \leq \frac{5}{4}$, and $\eta_0(x) = 0$ for $|x| \geq \frac{8}{5}$. For every dyadic number $N \in 2^\Z$ we set
\begin{align*}
	\chi_N(\xi) = \eta_0(|\xi|/N) - \eta_0(2|\xi|/N), \qquad \chi_{\leq N}(\xi) = \eta_0(|\xi|/N)
\end{align*}
for all $\xi \in \R^3$.
To each of these symbols we associate the Littlewood-Paley projectors
\begin{align*}
	P_N f = \cF^{-1}(\chi_N \hat{f}), \qquad P_{\leq N} f = \cF^{-1}(\chi_{\leq N} \hat{f}),
\end{align*}
where $\hat{f} = \cF f$ denotes the Fourier transform of $f$.

We next define different frequency interactions by
 	\begin{align}
    \label{eq:DefParaproduct}
     (f g)_{LH} &:= \sum_{N \in 2^\Z} P_{\leq K^{-1}N} f \, P_N g, \hspace{1.5em} (f g)_{HL} := (g f)_{LH}, \nonumber\\
     (f g)_{HH} &:=  \sum_{\substack{N_1/N_2, N_2/N_1 < K, \\ N_1, N_2 \in 2^\Z}} P_{N_1} f \, P_{N_2} g,
    \end{align}
    where $K \in 2^\Z$ with $K \geq 2^5$ is a dyadic integer which will be fixed in the contraction argument. To distinguish resonant interactions, we also introduce
    \begin{align}
    \label{eq:DefXLalphaL}
     &(f g)_{1 L} := \sum_{\substack{2^{-1} \leq N \leq 2, \\ N \in 2^\Z}} P_N f \, P_{\leq K^{-1}N} g, \qquad (f g)_{XL} := \sum_{\substack{|\log_2 N| > 1, \\ N \in 2^\Z}} P_N f \, P_{\leq K^{-1} N} g.
    \end{align}
    Note that
    \begin{align*}
    	f g = (f g)_{HL} + (f g)_{HH} + (f g)_{LH} = (f g)_{X L} + (f g)_{1L} + (f g)_{HH} + (f g)_{LH}.
    \end{align*}
    For later use we set
    \begin{equation}
    	\label{eq:DefR}
    	(f g)_R := (f g)_{1L + HH + LH}.
\end{equation}
    Finally, we write $\cP_*$ for the symbol of the bilinear operator with index $*$, i.e.,
    \begin{align*}
    	\cF(fg)_*(\xi) := \int_{\R^3} \cP_*(\xi-\eta,\eta) \hat{f}(\xi - \eta) \hat{g}(\eta) \dd \eta
    \end{align*}
    with $* \in \{1L, XL, HH, LH, R\}$.

\begin{remark}
	\label{rem:DroppingRealPart}
	Since $\Re(v) = \frac{1}{2}( v + \overline{v})$, the nonlinearity $\Re(v) u$ in~\eqref{equa-ranZak-bc} can be written as $\frac{1}{2} v u + \frac{1}{2} \overline{v} u$. The term $\overline{v} u$ can be treated in the same way as $v u$ so that we replace $\Re(v) u$ by $v u$ in~\eqref{equa-ranZak-bc} for simplicity. Moreover, our arguments work for $\Re(\cT_{\cdot}(W_2))$ in the same way as for $\cT_{\cdot}(W_2)$, so that we also drop the real part here for notational convenience.
\end{remark}

    We write~\eqref{equa-ranZak-bc} with initial data $(u_0, v_0) \in H^1 \times L^2$ in the Duhamel formulation
    \begin{equation} \label{eq:RZakharovDuhamel}
	\left\{\aligned
	 u(t) &= e^{\imu t \Delta} u_0
                  - \imu \int_0^t e^{\imu (t-s)\Delta} (vu  - b \cdot \nabla u - c u + \mathcal{T}_\cdot(W_2)u)(s) \dd s, \\
     v(t) &= e^{\imu t |\nabla|} v_0 + \imu \int_0^t e^{\imu (t-s)|\nabla|} (|\nabla| |u|^2)(s) \dd s .
	\endaligned
	\right.
\end{equation}
	The normal form transform (for the Schr{\"o}dinger part) builds on the fact that the resonance function
	\begin{align*}
 		\resf(\xi-\eta,\eta) := |\xi|^2 + |\xi - \eta| - |\eta|^2
 	\end{align*} 	
 	does not vanish on the support of $\cP_{XL}$, since $1 \nsim |\xi - \eta| \sim |\xi| \gg |\eta|$ for $(\xi - \eta, \eta) \in \supp \cP_{XL}$. Taking the Fourier transform on both sides of \eqref{eq:RZakharovDuhamel}, we obtain
    	\begin{align*}
         \hat{u}(t) &= e^{-\imu t |\xi|^2} \hat{u}_0 - \imu \int_0^t e^{-\imu (t-s)|\xi|^2} \cF(vu)(s) \dd s
         + \imu \int_0^t e^{-\imu (t-s)|\xi|^2} \cF(b \cdot \nabla u + c u - \mathcal{T}_{\cdot}(W_2) u)(s) \dd s, \\
         \hat{v}(t) &= e^{\imu t |\xi|} \hat{v}_0 + \imu \int_0^t e^{\imu (t-s)|\xi|} |\xi|| \cF(|u|^2)(s) \dd s,
        \end{align*}
        and thus
        \begin{equation} \label{pt-wtu-wtv}
	\left\{\aligned
	& \partial_t(e^{\imu t |\xi|^2} \hat{u}(t,\xi)) = -\imu e^{\imu t |\xi|^2}\cF(v u)(t,\xi)
           + \imu e^{\imu t |\xi|^2}\cF(b \cdot \nabla u + c u - \mathcal{T}_{\cdot}(W_2)u)(t,\xi),  \\
	& \partial_t(e^{-\imu t |\xi|} \hat{v}(t,\xi)) = \imu e^{-\imu t |\xi|} |\xi| \cF(|u|^2)(t,\xi).
	\endaligned
	\right.
\end{equation}
 	Writing
 	\begin{align}
 	\label{eq:FTSchroedinger}
 		\hat{u}(t) &= e^{-\imu t |\xi|^2} \hat{u}_0 - \imu \int_0^t e^{-\imu (t-s)|\xi|^2} \cF(vu)_{XL}(s) \dd s - \imu \int_0^t e^{-\imu (t-s)|\xi|^2} \cF(vu)_{R}(s) \dd s \nonumber\\
 		&\qquad + \imu \int_0^t e^{-\imu (t-s)|\xi|^2}\cF(b \cdot \nabla u + c u - \mathcal{T}_{\cdot}(W_2)u)(s) \dd s =: I + II + III + IV,
 	\end{align}
 	we obtain
 	\begin{align*}
 		II &= - \imu \int_0^t \int_{\R^3} e^{-\imu (t-s)|\xi|^2} \cP_{XL}(\xi-\eta, \eta) \hat{v}(s,\xi-\eta) \hat{u}(s,\eta) \dd \eta \dd s \\
 		&= - \imu  e^{-\imu t|\xi|^2} \int_0^t \int_{\R^3} e^{\imu s (|\xi|^2 + |\xi - \eta| - |\eta|^2)} \cP_{XL}(\xi-\eta, \eta) [e^{-\imu s |\xi - \eta|}\hat{v}(s,\xi-\eta)] [e^{\imu s |\eta|^2} \hat{u}(s,\eta)] \dd \eta \dd s \\
 		&= - e^{-\imu t|\xi|^2} \int_0^t \int_{\R^3} \frac{\cP_{XL}(\xi-\eta, \eta)}{\resf(\xi-\eta,\eta)} \partial_s e^{\imu s \resf(\xi-\eta,\eta)} [e^{-\imu s |\xi - \eta|}\hat{v}(s,\xi-\eta)] [e^{\imu s |\eta|^2} \hat{u}(s,\eta)] \dd \eta \dd s.
 	\end{align*}
 	Integrating by parts, using \eqref{eq:RZakharovDuhamel} and writing $\bdyop$ for the bilinear operator
 	\begin{align}   \label{Omegab-def}
 		\bdyop(f,g) := \cF^{-1} \int_{\R^3} \frac{\cP_{XL}(\xi-\eta, \eta)}{\resf(\xi-\eta,\eta)} \hat{f}(\xi - \eta) \hat{g}(\eta) \dd \eta,
 	\end{align}
 	we get
 	\begin{align}  \label{II-normform}
 		II =& - \cF \bdyop(v,u)(t,\xi) + e^{-\imu t |\xi|^2} \cF\bdyop(v_0,u_0)(\xi)  \notag  \\
 			&+ e^{-\imu t|\xi|^2} \int_0^t \int_{\R^3} \frac{\cP_{XL}(\xi-\eta, \eta)}{\resf(\xi-\eta,\eta)} e^{\imu s \resf(\xi-\eta,\eta)} [e^{-\imu s |\xi - \eta|}\imu |\xi - \eta|\cF(|u|^2)(s,\xi - \eta)] [e^{\imu s |\eta|^2} \hat{u}(s,\eta)] \dd \eta \dd s  \notag  \\
 			&+ e^{-\imu t|\xi|^2} \int_0^t \int_{\R^3} \frac{\cP_{XL}(\xi-\eta, \eta)}{\resf(\xi-\eta,\eta)} e^{\imu s \resf(\xi-\eta,\eta)} [e^{-\imu s |\xi - \eta|}\hat{v}(s,\xi-\eta)]  \notag  \\
 			&\hspace{18em} \cdot [e^{\imu s |\eta|^2} (-\imu)\cF(v u - b \cdot \nabla u - cu + \mathcal{T}_{\cdot}(W_2)u)(s,\eta)] \dd \eta \dd s  \notag  \\
 			= & - \cF \bdyop(v,u)(t,\xi) + e^{-\imu t |\xi|^2} \cF\bdyop(v_0,u_0)(\xi)
 					+ \imu \cF \Big(\int_0^t e^{\imu (t-s)\Delta} \bdyop(|\nabla| |u|^2, u)(s) \dd s\Big)(\xi)  \notag  \\
 			&+ \imu \cF\Big(\int_0^t e^{\imu(t-s)\Delta} \bdyop(v, -v u + b \cdot \nabla u + cu - \mathcal{T}_{\cdot}(W_2)u)(s) \dd s \Big)(\xi).
 	\end{align}
 	Inserting this formula into~\eqref{eq:FTSchroedinger}, we have transformed system~\eqref{eq:RZakharovDuhamel} into
 	 \begin{equation}
    	\label{eq:RZakharovNormalForm}
    	\begin{aligned}
         u(t) &= e^{\imu t \Delta} u_0 + e^{\imu t \Delta} \bdyop(v_0,u_0) -\bdyop(v,u)(t)  - \imu \int_0^t e^{\imu (t-s)\Delta} (vu)_{R}(s) \dd s \\
         &\qquad + \imu \int_0^t e^{\imu (t-s)\Delta}(b \cdot \nabla u + c u - \mathcal{T}_{\cdot}(W_2)u)(s) \dd s
          + \imu \int_0^t e^{\imu (t-s)\Delta} \bdyop(|\nabla| |u|^2, u)(s) \dd s\\
         &\qquad
         + \imu \int_0^t e^{\imu(t-s)\Delta} \bdyop(v, -vu + b \cdot \nabla u + cu -\mathcal{T}_{\cdot}(W_2)u)(s) \dd s, \\
         v(t) &= e^{\imu t |\nabla|} v_0 + \imu \int_0^t e^{\imu (t-s)|\nabla|} (|\nabla| |u|^2)(s) \dd s .
        \end{aligned}
    \end{equation}

Correspondingly we also have
\begin{equation*}
	\left\{\aligned
	&  (\partial_t- \imu \Delta) (u+\bdyop(v,u))
	  = -\imu (vu)_{R} + \imu (b\cdot \na u + cu - \mathcal{T}_{\cdot}(W_2)u)
	   + \imu \bdyop(|\na||u|^2, u)  \\
	  &\qquad  \qquad \qquad \qquad \qquad \ \ \ \ \ + \imu \bdyop(v, - vu + b\cdot \na u + cu - \mathcal{T}_{\cdot}(W_2)u),  \\
	& (\partial_t - \imu |\na|) v = \imu |\na||u|^2.
	\endaligned
	\right.
\end{equation*}

\begin{remark}
\begin{enumerate}[leftmargin=*]
\item The above computations in \eqref{II-normform} also show that
\begin{align} \label{equa-Omega-vu}
	(\partial_t - \imu \Delta) \bdyop(v,u)
	= \imu (vu)_{XL} + \bdyop((\partial_t - \imu|\na|)v, u)
	  + \bdyop (v, (\partial_t -\imu \Delta)u).
\end{align}
\item Note that also~\eqref{equa-Zakha-sigma} is of the form~\eqref{equa-ranZak-bc} with $\cT_t(W_2)$ replaced by $\cT_{\sigma + t, \sigma}(W_2)$. In particular, the normal form for~\eqref{equa-Zakha-sigma} is given by~\eqref{eq:RZakharovNormalForm} with  $\cT_t(W_2)$ replaced by $\cT_{\sigma + t, \sigma}(W_2)$.
\item While we did not specify the regularity of the involved functions in the derivation above and one might think of sufficiently smooth solutions at first, we rigorously prove in Proposition~\ref{Prop-Equiv-Mild-Norm} in the appendix that the normal form~\eqref{eq:RZakharovNormalForm} is equivalent to the Duhamel formulation~\eqref{equa-ranZak-bc} in the energy space.
\end{enumerate}
\end{remark}

\section{Multilinear estimates}   \label{Sec-Multi-Esti}
 	
In this section we prove the key multilinear estimates
which we will use to solve the Zakharov system in its normal form formulation~\eqref{eq:RZakharovNormalForm}. We first introduce our functional framework and then estimate the boundary terms, the quadratic interactions, and the cubic interactions on the right-hand side of~\eqref{eq:RZakharovNormalForm}.

Throughout this section, $I$ denotes a bounded subinterval of $[0,\infty)$ and $t_0:=\inf I$.

\subsection{Functional setting}   \label{Subsec-Funct}
Our functional framework combines Strichartz and local smoothing norms. We start by introducing the lateral spaces we will use to capture the local smoothing effect of the Schr{\"o}dinger flow.
 	
 	Let  $\textbf{e} \in \Sp^{2}$ and $\mathcal{P}_{\vece}=\{\xi \in \R^{3} \, |\, \xi \cdot \textbf{e}=0 \}$ with the induced Euclidean measure.
For $p, q \in [1,\infty]$, we define
\begin{equation} \label{Lepq-def}
\| f \|_{L_{\textbf{e}}^{p,q}(I \times \R^3)}
: = \Bigl( \int_{\R} \Bigl( \int_{I \times \mathcal{P}_{\textbf{e}}} |f(t, r \textbf{e} + y)|^q \dd t \dd y \Bigr)^{p/q} \dd r \Bigr)^{1/p}
\end{equation}
with the usual adaptions if $p = \infty$ or $q = \infty$.

We fix a nonnegative and symmetric function $\phi \in C_0^\infty(\R)$ such that $\phi(r)=0$ if $|r|\leq \frac{1}{8}$ or $|r|>4$ and $\phi(r)=1$ if $\frac{1}{4} \leq |r| \leq 2$, and set $\phi_N(r)=\phi(r/N)$. Consequently, we have
\begin{equation} \label{eq:DecompositionIdentity}
\prod_{j=1}^{3}(1-\phi_N(\xi_j))=0
\end{equation}
for all $\xi \in \R^3$ with $N/2 < |\xi| < 2 N$.
We finally define  $P_{N,\vece} :=\cF_{x}^{-1} \phi_N(\xi \cdot \vece) \cF_{x}$.
Note that~\eqref{eq:DecompositionIdentity} implies
 \begin{equation}\label{eq:dec-pe}
		P_N f = \sum_{j=1}^{3} P_{N,\vece_j} \Big[\prod_{l=1}^{j-1} (1-P_{N,\vece_l})\Big]P_N f,
	\end{equation}
	where $\vece_1, \vece_2, \vece_3$ denotes the standard basis of $\R^3$.

	We introduce the following norms. For every $N \in 2^\Z$ we set
\begin{align} \label{X-def}
		\|f\|_{\X(I)} := \|f\|_{L^\infty(I; H^1_x)} + \Big(\sum_{N \in 2^\Z} \|P_N f\|_{\X_N(I)}^2 \Big)^{\frac{1}{2}}
	\end{align}
with
	\begin{align*}
		\|f\|_{\X_N(I)} := \langle N \rangle \|f\|_{L^2(I;L^6_x)}
        +  \sum_{j = 1}^3 N^{\frac{3}{2}} \|P_{N, \vece_j} f\|_{L^{\infty,2}_{\vece_j}(I\times \bbr^3)},
	\end{align*}
	and
	\begin{align*}
       \|f\|_{\Y(I)} := \|f\|_{L^\infty(I; L^2)}.
\end{align*}
	
In order to estimate the nonlinear terms, we introduce
	\begin{align}  \label{G-def}
		\|g\|_{\G(I)} := \Big(\sum_{N \in 2^Z} \|P_N g\|_{\G_N(I)}^2\Big)^{\frac{1}{2}}
	\end{align}
with 	
\begin{align*}
		\|g\|_{\G_N(I)} := \inf_{g = g_1 + g_2 + g_3}\Big( \langle N \rangle \|g_1\|_{L^1(I;L^2_x)}
     + \langle N \rangle \|g_2\|_{L^{\frac{8}{5}}(I;L^{\frac{4}{3}}_x)}
      + \sum_{j = 1}^3 N^{\frac{1}{2}} \|g_3\|_{L^{1,2}_{\vece_j}(I\times \bbr^3)} \Big), \ \ N \in 2^{\Z}.
	\end{align*}
We further set	
\begin{align} \label{eq:DefXYG}
	\X(I) &:= \{f \in C(I; H^1_x) \colon \|f\|_{\X(I)} < \infty\}, \qquad \Y(I) := L^\infty(I; L^2_x), \qquad
	\G(I) := \{g \in L^1(I; L^2_x) \colon \|g\|_{\G(I)} < \infty\}.
\end{align}
	
We next provide the estimates for the linear flow of the Schr{\"o}dinger equation which we will use in the following. Part~\ref{it:LinStrichartz} contains the classical Strichartz estimate, while the local smoothing estimate in part~\ref{it:LinLocalSmoothing} follows from~(4.18) in~\cite{IK07}.

\begin{lemma} [Strichartz and local smoothing estimates] \label{lem:LinearEstimates}
Let $N \in 2^\Z$, $\vece \in \Sp^{2}$, and $(q,p)$ be Schr{\"o}dinger admissible,
i.e., $p,q\in [2,\infty]$ with $\frac{2}{q} + \frac{3}{p}= \frac{3}{2}$.
Then, for any $f \in L^2$ we have
\begin{enumerate}
  \item \label{it:LinStrichartz} Strichartz estimate:
				\begin{align*}
					\|e^{\imu (\cdot-t_0) \Delta} P_N f \|_{L^q(I;L^p_x)} \lesssim \|P_N f\|_{L^2}.
				\end{align*}
  \item \label{it:LinLocalSmoothing} Local smoothing estimate:
				\begin{align*}
					\|e^{\imu (\cdot-t_0) \Delta} P_{N, \vece} f\|_{L^{\infty,2}_{\vece}(I\times \bbr^d)} \lesssim N^{-\frac{1}{2}} \|f\|_{L^2}.
				\end{align*}
\end{enumerate}
\end{lemma}
Here and in the following we write $A \lesssim B$ if there is a constant $C > 0$ such that $A \leq C B$.
Note that Lemma~\ref{lem:LinearEstimates} immediately implies
\begin{equation}
	\label{eq:HomEstimate}
	\|e^{\imu (\cdot-t_0) \Delta} f \|_{\X(I)} \lesssim \|f\|_{H^1_x}
\end{equation}
for all $f \in H^1$. We also need to estimate the inhomogeneous terms in our functions spaces.
\begin{lemma}[Inhomogeneous estimate]
	\label{lem:InhomEstimate}
For any $g\in \G(I)$ we have
  \begin{align}
	\label{eq:EstDuhamel}
		\Big\| \int_{t_0}^\cdot e^{\imu (\cdot-s) \Delta} g(s) \dd s \Big\|_{\X(I)} \lesssim \|g\|_{\G(I)}.
	\end{align}
\end{lemma}

\begin{proof}
	Note that it is enough to prove for any $N \in 2^\Z$
	\begin{align*}
		\Big\| \int_{t_0}^\cdot e^{\imu (\cdot-s) \Delta} P_N g(s) \dd s \Big\|_{\X_N(I)} &\lesssim \langle N \rangle \|P_N g\|_{L^{q'}(I; L^{p'}_x)}, \\
		\Big\| \int_{t_0}^\cdot e^{\imu (\cdot-s) \Delta} P_N g(s) \dd s \Big\|_{\X_N(I)} &\lesssim \sum_{j = 1}^3 N^{\frac{1}{2}} \|P_N g\|_{L^{1,2}_{\vece_j}(I\times \bbr^3)}
	\end{align*}
	for $(q',p') = (1,2)$ and for $(q',p') = (\frac{8}{5}, \frac{4}{3})$. The first estimate (both for $(q',p') = (1,2)$ and for $(q',p') = (\frac{8}{5}, \frac{4}{3})$) follows from Lemma~\ref{lem:LinearEstimates}, the dual Strichartz estimate, and the Christ-Kiselev lemma in the form of Lemma~B.3 in~\cite{WHH09}. The second estimate follows as Proposition~3.8 in~\cite{BIKT11}.
\end{proof}

\subsection{Boundary estimates}
We next derive the multilinear estimates for the terms appearing on the right-hand side of~\eqref{eq:RZakharovNormalForm}. We start with the boundary terms from the integration by parts argument. To that purpose, we first introduce the bilinear operators
\begin{equation}
	\label{eq:DefTm}
	T_m(f,g)(x) = \int_{\R^6} m(\xi,\eta) \hat{f}(\xi) \hat{g}(\eta) e^{\imu x (\xi + \eta)} \dd \xi \dd \eta,\ \ x \in \R^3,
\end{equation}
for $m \in L^{\infty}(\R^6)$ and $f,g \in \Schw(\R^3)$.

A crucial tool in estimating the boundary terms is the following
Coifman-Meyer-type bilinear multiplier estimate, which was proven in~\cite[Lemma~3.5]{GN14}.

\begin{lemma}[Bilinear multiplier estimate]
	\label{lem:BilinearMultiplierEstimate}
	Let  $m \in C^{\infty}(\R^{6})$ be bounded and satisfy
	\begin{equation*}
		|\partial^\alpha_\xi \partial^\beta_\eta m(\xi,\eta)| \lesssim_{\alpha,\beta} |\xi|^{-|\alpha|} |\eta|^{-|\beta|}
	\end{equation*}
	for all $\xi,\eta \in \R^3$ and $\alpha,\beta \in \N_0^3$. Take $p,q,r \in [1,\infty]$ with $\frac{1}{r} = \frac{1}{p} + \frac{1}{q}$.
	Then, for any  $f \in L^p(\R^3)$, $g \in L^q(\R^3)$, and $N_1, N_2 \in 2^\Z$, we have
	\begin{equation*}
		\|T_m(P_{N_1} f, P_{N_2} g) \|_{L^r} \lesssim \|P_{N_1} f\|_{L^p} \|P_{N_2} g\|_{L^q},
	\end{equation*}
	where $T_m$ is the operator defined in~\eqref{eq:DefTm}.
\end{lemma}

We will use the above lemma to estimate the bilinear operator $\bdyop$ given by \eqref{Omegab-def}.
Roughly speaking, Lemma~\ref{lem:BilinearMultiplierEstimate} yields that $\bdyop$ acts like
\begin{align} \label{Omegab-na}
	 \bdyop(f,g) \sim |\nabla|^{-1} \langle \nabla \rangle^{-1} (fg)_{XL}.
\end{align}

In fact, following \cite{GN14, S22}, we will always apply Lemma~\ref{lem:BilinearMultiplierEstimate} to $|\nabla| \langle \nabla \rangle \bdyop(\cdot, \cdot)$, which is a bilinear Fourier multiplier with symbol
\begin{align*}
			m(\xi,\eta) = \frac{|\xi + \eta| \langle \xi + \eta \rangle \sum_{|\log_2 N| > 1} \chi_N(\xi) \chi_{\leq K^{-1} N}(\eta)}{|\xi + \eta|^2 + |\xi| - |\eta|^2}.
		\end{align*}
		We note that within the support of $m$, the numerator and the denominator of this symbol are both of size $N$ in the case $N < \frac{1}{2}$ and both of size $N^2$ in the case $N > 2$. A short computation then shows that $m$ satisfies the assumptions of Lemma~\ref{lem:BilinearMultiplierEstimate}.

We obtain the following estimates for the boundary terms.
\begin{lemma} [Boundary estimates]   \label{lem:Bdy}
We have
	\begin{align}
		& \|e^{\imu (\cdot-t_0) \Delta} \bdyop(v_0,u_0)\|_{\X(I)} \lesssim K^{-1} \|v_0\|_{L^2} \|u_0\|_{H^1}, \label{eq:EstBdyInitial} \\
		& \|\bdyop(v,u)\|_{\X(I)} \lesssim (K^{-1} + |I|^{\frac{1}{8}})\|v\|_{\Y(I)} \|u\|_{\X(I)}.    \label{eq:EstBdyTime}
	\end{align}
\end{lemma}

\begin{proof}
By~\eqref{eq:HomEstimate}, it is enough to show
\begin{equation}
\label{eq:ReductionBdyEstInitial}
	\| \bdyop(v_0, u_0)\|_{H^1} \lesssim K^{-1} \|v_0\|_{L^2} \|u_0\|_{H^1}
\end{equation}
in order to obtain~\eqref{eq:EstBdyInitial}.
To this end, employing Lemma~\ref{lem:BilinearMultiplierEstimate}, we infer
\begin{align}
\label{eq:EstBdyOpInitial}
	\| \bdyop(v_0, u_0)\|_{H^1} &\lesssim \Big(\sum_{N \in 2^\Z} \|\langle \nabla \rangle \bdyop(P_N v_0, P_{\leq K^{-1} N} u_0)\|_{L^2}^2 \Big)^{\frac{1}{2}} \nonumber \\
	&\lesssim \Big(\sum_{N \in 2^\Z} \Big(\sum_{N_1 \leq K^{-1} N} \||\nabla| \langle \nabla \rangle \bdyop(N^{-1} P_N v_0, P_{N_1} u_0)\|_{L^2}\Big)^2 \Big)^{\frac{1}{2}} \nonumber\\
	&\lesssim \Big(\sum_{N \in 2^\Z} \Big(\sum_{N_1 \leq K^{-1} N} N^{-1} \| P_N v_0 \|_{L^2} \|P_{N_1} u_0\|_{L^\infty}\Big)^2 \Big)^{\frac{1}{2}} \nonumber\\
	&\lesssim \|v_0\|_{L^2} \Big(\sum_{N \in 2^\Z} \Big(\sum_{N_1 \leq K^{-1} N} N^{-1}N_1 \|N_1^{\frac{1}{2}} P_{N_1} u_0\|_{L^2}\Big)^2 \Big)^{\frac{1}{2}} \notag \\
    & \lesssim K^{-1} \|v_0\|_{L^2} \|u_0\|_{H^1},
\end{align}
where we used Young's inequality for series in the last step.

In order to prove~\eqref{eq:EstBdyTime}, we first note that arguing as in~\eqref{eq:EstBdyOpInitial} we get
\begin{equation}
	\label{eq:EstBdyOpTimeH1}
	\| \bdyop(v,u)\|_{L^\infty(I;H^1_x)}
     \lesssim K^{-1} \|v\|_{L^\infty(I;L^2_x)} \|u\|_{L^\infty(I;H^1_x)} \lesssim K^{-1} \|v\|_{\Y(I)} \|u\|_{\X(I)}.
\end{equation}
To control the second component of the $\X(I)$-norm, we use Lemma~\ref{lem:BilinearMultiplierEstimate} again to derive
\begin{align}
	\| \langle \nabla \rangle \bdyop(v,u)\|_{L^2(I;\dot{B}^0_{6,2})}
	&\lesssim \Big\|\Big(\sum_{N \in 2^\Z} \Big(\sum_{N_1 \leq K^{-1} N} \||\nabla| \langle \nabla \rangle \bdyop(N^{-1} P_N v, P_{N_1} u)\|_{L^6_x}\Big)^2  \Big)^{\frac{1}{2}} \Big\|_{L^2(I)} \nonumber\\
	&\lesssim \Big\|\Big(\sum_{N \in 2^\Z} \Big(\sum_{N_1 \leq K^{-1} N} \|N^{-1} P_N v\|_{L^6_x} \|P_{N_1} u\|_{L^\infty_x}\Big)^2  \Big)^{\frac{1}{2}} \Big\|_{L^2(I)} \nonumber\\
	&\lesssim \Big\|\Big(\sum_{N \in 2^\Z} \| P_N v\|_{L^2_x}^2 \Big(\sum_{N_1 \leq K^{-1} N} N_1^{\frac{3}{4}} \|P_{N_1} u\|_{L^4_x}\Big)^2  \Big)^{\frac{1}{2}} \Big\|_{L^2(I)}  \nonumber\\
	&\lesssim \|v\|_{\Y(I)} \|\langle \nabla \rangle u\|_{L^2(I;L^4_x)} \nonumber\\
	 &\lesssim |I|^{\frac{1}{8}} \|v\|_{\Y(I)} \|u\|_{\X(I)}, \label{eq:EstBdyOpTimeL2L6}
\end{align}
where we exploited that $\sum_{N_1 \in 2^\Z} N_1^{\frac{3}{4}} \|P_{N_1} u \|_{L^4_x} \lesssim \|\langle \nabla \rangle u\|_{L^4_x}$
and $\dot{B}^0_{4,2} \hookrightarrow L^4$.
It remains to control the local smoothing component of the $\X(I)$-norm.
For this purpose, we first note that $\mathcal{F}_{x_j} ((P_{N,\vece_j} g)(x'))$ is supported in an interval of length $4N$ for any $g \in L^2(\R^3)$, where $x' = (x_2, x_3)$, $x' = (x_1, x_3)$, respectively $x' = (x_1, x_2)$ if $j = 1, 2,3$. Fix $j \in \{1,2,3\}$. Using Minkowski's inequality and Bernstein's inequality in one dimension, we obtain
\begin{align*}
	&\Big(\sum_{N \in 2^\Z} (N^{\frac{3}{2}} \|P_N P_{N,\vece_j} \bdyop(v,u)\|_{L^{\infty,2}_{\vece_j}(I\times \bbr^3)})^2 \Big)^{\frac{1}{2}} \\
	&\lesssim \Big(\sum_{N \in 2^\Z}( N^{\frac{3}{2}} \|P_{N,\vece_j} \bdyop(P_N v,P_{\leq K^{-1}N} u)\|_{L^2_{t,x'} L^\infty_{x_j}})^2 \Big)^{\frac{1}{2}} \\
	&\lesssim \Big(\sum_{N \in 2^\Z} ( N^{2} \| \bdyop(P_N v,P_{\leq K^{-1}N} u)\|_{L^2(I\times \bbr^3)}\Big)^2 \Big)^{\frac{1}{2}} \\
	&\lesssim \Big\|\Big(\sum_{N \in 2^\Z} \Big(\sum_{N_1 \leq K^{-1}N} N^{2} \| \bdyop(P_N v,P_{N_1} u)\|_{L^2}\Big)^2 \Big)^{\frac{1}{2}}\Big\|_{L^2(I)}.
\end{align*}
Applying Lemma~\ref{lem:BilinearMultiplierEstimate} once more and arguing as in~\eqref{eq:EstBdyOpTimeL2L6},
we infer that the right-hand side is bounded by
\begin{align}
	 & \Big\|\Big(\sum_{N \in 2^\Z} \Big(\sum_{N_1 \leq K^{-1} N} \|P_N v\|_{L^2_x} \|P_{N_1} u\|_{L^\infty}\Big)^2 \Big)^{\frac{1}{2}}\Big\|_{L^2(I)} \nonumber\\
	&\lesssim \Big\|\Big(\sum_{N \in 2^\Z} \|P_N v\|_{L^2}^2
       \Big(\sum_{N_1 \leq K^{-1} N} N_1^{\frac{3}{4}} \|P_{N_1} u\|_{L^4}\Big)^2 \Big)^{\frac{1}{2}}\Big\|_{L^2(I)} \notag\\
     &\lesssim |I|^{\frac{1}{8}} \|v\|_{\Y(I)} \|u\|_{\X(I)}. \label{eq:EstBdyOpLocSmooth}
\end{align}
Combining~\eqref{eq:EstBdyOpTimeH1}, \eqref{eq:EstBdyOpTimeL2L6}, and~\eqref{eq:EstBdyOpLocSmooth}, we conclude~\eqref{eq:EstBdyTime}.
\end{proof}
 	
 	\subsection{Bilinear estimates}

We first derive the bilinear estimates for the Schr\"odinger operator.

 	\begin{lemma}
 		\label{lem:Quadratic}
   We have
 		\begin{align}
 			&\Big\|\int_{t_0}^\cdot e^{\imu(\cdot-s)\Delta} (v u)_R(s) \dd s \Big\|_{\X(I)}
               \lesssim  |I|^{\frac{1}{4}} K \log_2 K \|v\|_{\Y(I)} \|u\|_{\X(I)}, \label{eq:EstQuadraticR}\\
 			&\Big\|\int_{t_0}^\cdot e^{\imu(\cdot-s) \Delta} (b \cdot \nabla u)(s) \dd s \Big\|_{\X(I)}
              \lesssim \Big(|I|^{\frac{1}{4}} \|b\|_{L^\infty(I;H^1_x)}
               + \sum_{j = 1}^3 \|b\|_{L^{1,\infty}_{\vece_j}(I\times \bbr^3)} \Big)\|u\|_{\X(I)}, \label{eq:EstQuadraticbNabu}\\
 			&\Big\|\int_{t_0}^\cdot e^{\imu(\cdot-s) \Delta} (c u)(s) \dd s \Big\|_{\X(I)}
             \lesssim |I|^{\frac{1}{4}} \|c\|_{L^\infty(I; H^1_x)} \|u\|_{\X(I)}, \label{eq:EstQuadraticcu} \\
            & \Big\|\int_{t_0}^\cdot e^{\imu(\cdot-s) \Delta} (\mathcal{T}_{\cdot}(W_2)u)(s) \dd s \Big\|_{\X(I)}
            \lesssim |I|^{\frac{1}{4}} \|\mathcal{T}_{\cdot}(W_2) \|_{L^\infty(I; H^1_x)} \|u\|_{\X(I)}.   \label{eq:EstQuadraticNaTWu}
 		\end{align}
 	\end{lemma}

\begin{proof}
 		We start by proving~\eqref{eq:EstQuadraticR}. By~\eqref{eq:EstDuhamel} and the definition of $\|\cdot\|_{\G(I)}$ we obtain
 		\begin{align*}
 			\Big\|\int_{t_0}^\cdot e^{\imu(\cdot-s)\Delta} (v u)_R(s) \dd s \Big\|_{\X(I)}
            \lesssim \|\|\langle \nabla \rangle P_N(vu)_R\|_{L^\frac 85(I;L^\frac 43_x)}\|_{l^2_N}
            \lesssim \|\langle \nabla \rangle (vu)_{R}\|_{L^{\frac{8}{5}}(I;\dot{B}^0_{\frac{4}{3},2})},
 		\end{align*}
 		where we applied Minkowski's inequality in the last step.

In order to estimate the right-hand side above,
we start with the $LH$-component of $(vu)_R$. Employing H{\"o}lder's inequality, we derive
 		\begin{align} \label{vu-LH-esti}
 			\|\langle \nabla \rangle (vu)_{LH}\|_{L^{\frac{8}{5}}(I;\dot{B}^0_{\frac{4}{3},2})}
 			&\lesssim \Big\|\Big(\sum_{N \in 2^{\Z}} \langle N \rangle^2
                \|P_{\leq K^{-1} N} v P_N u\|_{L^{\frac{4}{3}}_x}^2 \Big)^{\frac{1}{2}} \Big\|_{L^{\frac{8}{5}}(I)} \notag \\
 			&\lesssim \Big\|\Big(\sum_{N \in 2^{\Z}} \langle N \rangle^2 \|P_{\leq K^{-1} N} v\|_{L^2_x}^2
            \|P_N u\|_{L^{4}_x}^2 \Big)^{\frac{1}{2}} \Big\|_{L^{\frac{8}{5}}(I)}  \notag  \\
            &\lesssim |I|^{\frac{1}{4}} \|v\|_{\Y(I)} \|u\|_{\X(I)}.
 		\end{align}
Similarly, we obtain for the $HH$- and $1L$-components
 		\begin{align} \label{vu-HH-esti}
 			\|\langle \nabla \rangle (vu)_{HH}\|_{L^{\frac{8}{5}}(I;\dot{B}^0_{\frac{4}{3},2})}
 			&\lesssim \|\langle \nabla \rangle (vu)_{HH}\|_{L^{\frac{8}{5}}(I;L^{\frac{4}{3}}_x)} \notag \\
 			&\lesssim \Big\|\sum_{N \in 2^{\Z}} \sum_{\substack{N_1 \in 2^{\Z} \\
           K^{-1}N < N_1 < K N}} \langle \nabla \rangle (P_{N_1} v P_N u) \Big\|_{L^{\frac{8}{5}}(I; L^{\frac{4}{3}}_x)} \nonumber\\
 			&\lesssim \Big\|\sum_{\substack{M \in 2^{\Z} \\
              K^{-1} < M < K }} \sum_{N \in 2^{\Z}}  \langle M N + N \rangle \|P_{MN} v\|_{L^2_x} \| P_N u\|_{L^4_x} \Big\|_{L^{\frac{8}{5}}(I)}  \notag \\
             &\lesssim |I|^{\frac{1}{4}} K \log_2 K \|v\|_{\Y(I)} \|u\|_{\X(I)},
        \end{align}
and
        \begin{align} \label{vu-1L-esti}
 			\|\langle \nabla \rangle (vu)_{1L}\|_{L^{\frac{8}{5}}(I;\dot{B}^0_{\frac{4}{3},2})}
 			&\lesssim \Big\| \sum_{\substack{2^{-1} \leq N \leq 2  \\ N \in 2^{\Z}}} \langle
            \nabla \rangle (P_{N} v P_{\leq K^{-1}N} u) \Big\|_{L^{\frac{8}{5}}(I;L^{\frac{4}{3}}_x)} \notag \\
            &\lesssim |I|^{\frac{1}{4}} \|v\|_{\Y(I)} \|u\|_{L^{\frac{8}{3}}(I; L^{4}_x)} \nonumber\\
 			&\lesssim |I|^{\frac{1}{4}} \|v\|_{\Y(I)} \|u\|_{\X(I)}.
 		\end{align}
Estimate \eqref{eq:EstQuadraticR} then follows immediately from \eqref{vu-LH-esti}, \eqref{vu-HH-esti}, and \eqref{vu-1L-esti}.
 		
		We next prove~\eqref{eq:EstQuadraticbNabu}. To that purpose, we decompose
		\begin{align*}
			b \cdot \nabla u = (b \cdot \nabla u)_{HL} + (b \cdot \nabla u)_{HH} + (b \cdot \nabla u)_{LH}
		\end{align*}
		with $K = 2^5$. 		
		Estimate~\eqref{eq:EstDuhamel} and the definition of $\|\cdot\|_{\G(I)}$, again combined with Minkowski's inequality, yield
		\begin{align} \label{eq:EstQuadraticbNablauDuhamel}
			 \Big\|\int_0^t e^{\imu(t-s) \Delta} (b \cdot \nabla u)(s) \dd s \Big\|_{\X(I)}
			&\lesssim \| \langle \nabla \rangle (b \cdot \nabla u)_{HL + HH}\|_{L^{\frac{8}{5}}(I;\dot{B}^0_{\frac{4}{3},2})}   \nonumber\\
			&\qquad + \sum_{j = 1}^3 \Big(\sum_{N \in 2^\Z} N \|P_N (b \cdot \nabla u)_{LH}\|_{L^{1,2}_{\vece_j}(I\times \bbr^d}^2 \Big)^{\frac{1}{2}}.
		\end{align}
		We start with the first summand on the right-hand side. We infer
		\begin{align}
			\| \langle \nabla \rangle (b \cdot \nabla u)_{HL}\|_{L^{\frac{8}{5}}(I;\dot{B}^0_{\frac{4}{3},2})}
            &\lesssim \Big\| \Big(\sum_{N \in 2^\Z} \langle N \rangle^2 \|P_N b \|_{L^2_x}^2 \|P_{\leq K^{-1}N} (\nabla u)\|_{L^{4}_x}^2 \Big)^{\frac{1}{2}} \Big\|_{L^{\frac{8}{5}}(I)}  \nonumber \\
			&\lesssim \|b\|_{L^\infty(I;H^1_x)} \|\nabla u\|_{L^{\frac{8}{5}}(I;L^4_x)} \nonumber\\
			&\lesssim |I|^{\frac{1}{4}} \|b\|_{L^\infty(I;H^1_x)} \|\langle \nabla \rangle u\|_{L^{\frac{8}{3}}(I;\dot{B}^0_{4,2})}   \nonumber \\
            &\lesssim |I|^{\frac{1}{4}} \|b\|_{L^\infty(I;H^1_x)} \|u\|_{\X(I)}. \label{eq:EstQuadraticbNablauHL}
		\end{align}
		With the same adaptions as in~\eqref{vu-HH-esti}, we also obtain
		\begin{align}
			\| \langle \nabla \rangle (b \cdot \nabla u)_{HH}\|_{L^{\frac{8}{5}}(I;\dot{B}^0_{\frac{4}{3},2})}
           \lesssim |I|^{\frac{1}{4}} \|b\|_{L^\infty(I;H^1_x)} \|u\|_{\X(I)}. \label{eq:EstQuadraticbNablauHH}
		\end{align}
		It remains to estimate the sum on the right-hand side of~\eqref{eq:EstQuadraticbNablauDuhamel}. Fix $j \in \{1,2,3\}$. Using the decomposition~\eqref{eq:dec-pe}, we then derive
		\begin{align}
			 &\Big(\sum_{N \in 2^\Z} N \|P_N (b \cdot \nabla u)_{LH}\|_{L^{1,2}_{\vece_j}(I\times \bbr^3)}^2 \Big)^{\frac{1}{2}}
             \lesssim \Big(\sum_{N \in 2^\Z} N \|P_{\leq K^{-1} N} b P_N(\nabla u)\|_{L^{1,2}_{\vece_j}(I\times \bbr^3)}^2 \Big)^{\frac{1}{2}} \nonumber\\
			&\lesssim \Big(\sum_{N \in 2^\Z} N \| |P_{\leq K^{-1} N} b|^{\frac{1}{2}} \|_{L^{2,\infty}_{\vece_j}(I\times \bbr^3)}^2
               \Big\||P_{\leq K^{-1} N} b|^{\frac{1}{2}} \Big|\sum_{m=1}^{3} P_{N,\vece_m} \Big[\prod_{l=1}^{m-1} (1-P_{N,\vece_l})\Big]P_N(\nabla u)\Big|\Big\|_{L^{2}_{t,x}(I\times \bbr^3)}^2 \Big)^{\frac{1}{2}}   \nonumber\\
			&\lesssim \|b\|_{L^{1,\infty}_{\vece_j}(I\times \bbr^3)}^{\frac{1}{2}}
             \sum_{m = 1}^3 \Big(\sum_{N \in 2^\Z} N  \Big\||P_{\leq K^{-1} N} b|^{\frac{1}{2}} \Big|P_{N,\vece_m} \Big[\prod_{l=1}^{m-1} (1-P_{N,\vece_l})\Big]P_N(\nabla u)\Big|\Big\|_{L^{2}_{t,x}(I\times \bbr^3)}^2 \Big)^{\frac{1}{2}} \nonumber\\
             &\lesssim \|b\|_{L^{1,\infty}_{\vece_j}(I\times \bbr^3)}^{\frac{1}{2}}
             \sum_{m = 1}^3 \Big(\sum_{N \in 2^\Z} N \| |P_{\leq K^{-1} N} b|^{\frac{1}{2}} \|_{L^{2,\infty}_{\vece_m}}^2 \Big\| P_{N,\vece_m} \Big[\prod_{l=1}^{m-1} (1-P_{N,\vece_l})\Big]P_N(\nabla u)\Big\|_{L^{\infty,2}_{\vece_m}(I\times \bbr^3)}^2 \Big)^{\frac{1}{2}} \nonumber\\
             &\lesssim \|b\|_{L^{1,\infty}_{\vece_j}(I\times \bbr^3)}^{\frac{1}{2}}
             \sum_{m = 1}^3 \|b\|_{L^{1,\infty}_{\vece_m}(I\times \bbr^3)}^{\frac{1}{2}}\Big(\sum_{N \in 2^\Z} N^3  \| P_{N,\vece_m}P_N u \|_{L^{\infty,2}_{\vece_m}(I\times \bbr^3)}^2 \Big)^{\frac{1}{2}} \lesssim \Big( \sum_{j = 1}^3 \|b\|_{L^{1,\infty}_{\vece_j}(I\times \bbr^3)} \Big) \|u\|_{\X(I)}. \label{eq:EstQuadraticbNablauLH}
		\end{align}
		
 		Inserting~\eqref{eq:EstQuadraticbNablauHL}, \eqref{eq:EstQuadraticbNablauHH}, and~\eqref{eq:EstQuadraticbNablauLH} into~\eqref{eq:EstQuadraticbNablauDuhamel}, we conclude~\eqref{eq:EstQuadraticbNabu}. 		
 		
 		Finally, to get~\eqref{eq:EstQuadraticcu} we estimate as above
 		\begin{align*}
 			\Big\|\int_{t_0}^\cdot e^{\imu(\cdot-s) \Delta} (c u)(s) \dd s \Big\|_{\X(I)}
        &\lesssim \| \langle \nabla \rangle (cu) \|_{L^\frac{8}{5}(I;\dot{B}^0_{\frac 43,2})}
        \lesssim |I|^{\frac{1}{4}} \|\langle \nabla \rangle c\|_{L^\infty(I; L^2_x)} \| \langle \nabla \rangle u \|_{L^{\frac{8}{3}}(I;L^4_x)} \\
 		&\lesssim |I|^{\frac{1}{4}} \|c\|_{L^\infty(I; H^1_x)} \| u \|_{\X(I)}.
 		\end{align*}
Estimate \eqref{eq:EstQuadraticNaTWu} is shown in the same way as~\eqref{eq:EstQuadraticcu} by replacing $c$ with $\mathcal{T}_{\cdot}(W_2)u$.
Therefore, the proof is complete.
\end{proof}
 	
Regarding the wave operator,
we have the following bilinear estimate.
\begin{lemma}  \label{lem:QuadraticWave}
We have
 		\begin{align}
 			&\Big\|\int_{t_0}^\cdot e^{\imu(\cdot-s)|\nabla|} |\nabla| (u_1 u_2)(s) \dd s \Big\|_{\Y(I)}
             \lesssim |I|^{\frac 14} \|u_1\|_{\X(I)} \|u_2\|_{\X(I)}.     \label{eq:EstQuadraticWave}
 		\end{align}
 	\end{lemma}
 	
 	\begin{proof}
 		The energy-estimate, dyadic decomposition and interpolation yield
 		\begin{align*}
 			\Big\|\int_{t_0}^\cdot e^{\imu(\cdot-s)|\nabla|} |\nabla| (u_1 u_2)(s) \dd s \Big\|_{\Y(I)}
            &\lesssim \|  |\nabla| (u_1 u_2)\|_{L^1(I;L^2_x)}  \\
 			 &\lesssim  |I|^{\frac 14} \|\langle \nabla \rangle u_1\|_{L^2(I;\dot{B}^0_{6,2})}
              \|\langle \nabla \rangle u_2\|_{L^4(I;\dot{B}^0_{3,2})} \\
 			&\lesssim   |I|^{\frac 14}  \|u_1\|_{\X(I)} \|u_2\|_{\X(I)}.
 		\end{align*}
 	\end{proof}
 	
 	\subsection{Trilinear estimates}
 	
\begin{lemma} [Trilinear estimates]  \label{lem:Cubic}
We have
 \begin{align}
 		 &\Big\|\int_{t_0}^\cdot e^{\imu (\cdot-s)\Delta}
 \bdyop( |\nabla| (u_1 u_2), u_3)(s) \dd s \Big\|_{\X(I)}
       \lesssim |I|^\frac 14 \|u_1\|_{\X(I)} \|u_2\|_{\X(I)} \|u_3\|_{\X(I)}, \label{eq:EstCubicuuu}\\
         &\Big\|\int_{t_0}^\cdot e^{\imu(\cdot-s)\Delta} \bdyop(v_1, v_2 u)(s) \dd s \Big\|_{\X(I)}
          \lesssim |I|^{\frac{1}{8}} \|v_1\|_{\Y(I)} \|v_2\|_{\Y(I)} \|u\|_{\X(I)}, \label{eq:EstCubicvvu}\\
         &\Big\|\int_{t_0}^\cdot e^{\imu(\cdot-s)\Delta} \bdyop(v, b \cdot \nabla u + cu - \mathcal{T}_{\cdot}(W_2)u)(s) \dd s \Big\|_{\X(I)}
         \lesssim |I|^{\frac{1}{8}} \bigg(\|b\|_{L^\infty(I;H^1_x)} + \|c\|_{\Y(I)} \notag \\
             &\qquad \qquad \qquad \qquad \qquad \qquad \qquad \qquad \qquad \qquad \qquad \qquad \
             \ \ + \|\mathcal{T}_{\cdot}(W_2)\|_{\Y(I)} \bigg)\|v\|_{\Y(I)} \|u\|_{\X(I)}.    \label{eq:EstCubicLowerOrderTerms}
 \end{align}
\end{lemma}
 	
 	\begin{proof}
 		We start by proving~\eqref{eq:EstCubicuuu}. Using estimate~\eqref{eq:EstDuhamel} and Lemma~\ref{lem:BilinearMultiplierEstimate}, we derive
 		\begin{align*}
 			 \Big\|\int_{t_0}^\cdot e^{\imu (\cdot-s)\Delta} \bdyop(|\nabla| (u_1 u_2), u_3)(s) \dd s \Big\|_{\X(I)}
        &\lesssim \| \bdyop( |\nabla| (u_1 u_2), u_3)\|_{L^1(I;H^1_x)} \\
 		&\lesssim \Big\| \Big(\sum_{N \in 2^\Z}
        \Big( \sum_{N_1 \leq K^{-1} N}\| P_N (u_1 u_2)\|_{L^2_x} \| P_{N_1} u_3 \|_{L^\infty_x}\Big)^2 \Big)^{\frac{1}{2}} \Big\|_{L^1(I)}   \\
 		&\lesssim |I|^\frac 14 \| \|u_1 u_2\|_{L^2_x} \|\langle \nabla \rangle u_3\|_{L^4_x} \|_{L^\frac 43(I)}  \\
        &\lesssim |I|^\frac 14 \|u_1\|_{\Y(I)} \|\langle \nabla \rangle u_2\|_{L^\frac{8}{3}(I;L^4_x)}
          \|\langle \nabla \rangle u_3\|_{L^\frac{8}{3}(I; L^4_x)} \\
        &\lesssim |I|^\frac 14 \|u_1\|_{\X(I)} \|u_2\|_{\X(I)} \|u_3\|_{\X(I)},
 		\end{align*}
 		where we once more used that, via Bernstein's inequality,
 		\begin{align*}
 			\sum_{N_1 \leq K^{-1} N} \| P_{N_1} u \|_{L^\infty_x} \lesssim \sum_{N_1 \in 2^{\Z}} N_1^{\frac{3}{4}} \| P_{N_1} u \|_{L^4_x} \lesssim \|\langle \nabla \rangle u\|_{L^4_x}.
 		\end{align*}

 		We proceed in the same way to show~\eqref{eq:EstCubicvvu}. Estimate~\eqref{eq:EstDuhamel} and Lemma~\ref{lem:BilinearMultiplierEstimate} imply
 		\begin{align*}
 			 \Big\|\int_{t_0}^\cdot e^{\imu(\cdot-s)\Delta} \bdyop(v_1, v_2 u)(s) \dd s \Big\|_{\X(I)}
         &\lesssim \| \langle \nabla \rangle \bdyop(v_1,v_2u)\|_{L^{\frac{8}{5}}(I; \dot{B}^{0}_{\frac{4}{3},2})} \\
 			&\lesssim \Big\|\Big(\sum_{N \in 2^\Z}
        \Big(\sum_{N_1 \leq K^{-1} N} N^{-1} \|P_N v_1 \|_{L^2_x} \|P_{N_1}(v_2 u)\|_{L^{4}_x} \Big)^2 \Big)^{\frac{1}{2}} \Big\|_{L^{\frac{8}{5}}(I)} \\
 			 &\lesssim \|v_1 \|_{\Y(I)}\Big\|
         \Big(\sum_{N \in 2^\Z} \Big(\sum_{N_1 \leq K^{-1} N} N^{-1} N_1  \|P_{N_1}(v_2u)\|_{L^{\frac{12}{7}}_x} \Big)^2 \Big)^{\frac{1}{2}} \Big\|_{L^{\frac{8}{5}}(I)}.
 		\end{align*}
 		By Young's inequality for series and the embedding $L^{\frac{12}{7}} \hookrightarrow \dot{B}^0_{\frac{12}{7},2}$ we thus conclude
 		\begin{align}
 			\Big\|\int_{t_0}^\cdot e^{\imu(\cdot-s)\Delta} \bdyop(v_1, v_2 u)(s) \dd s \Big\|_{\X(I)}
 			\lesssim& \|v_1\|_{\Y(I)} \|v_2 u\|_{L^{\frac{8}{5}}(I; L^{\frac{12}{7}}_x)} \notag \\
            \lesssim& \|v_1\|_{\Y(I)} \|v_2\|_{L^\infty(I; L^2_x)} \|u\|_{L^{\frac{8}{5}}(I; L^{12}_x)}  \nonumber\\
 			  \lesssim& |I|^{\frac{1}{8}} \|v_1\|_{\Y(I)} \|v_2\|_{\Y(I)} \|u\|_{\X(I)}. \label{eq:EstDuhamelBdyvvu}
 		\end{align}

 		Replacing $v_2$ in the second component of $\bdyop$ by $c-\mathcal{T}_{\cdot}(W_2)$, we obtain from~\eqref{eq:EstDuhamelBdyvvu}
 		\begin{align}
 			\label{eq:EstDuhamelBdyvcu}
 			&\Big\|\int_{t_0}^\cdot e^{\imu(\cdot-s)\Delta} \bdyop(v, (c-\cT_\cdot(W_2)) u)(s) \dd s \Big\|_{\X(I)}
 			\lesssim  |I|^{\frac{1}{8}} \|v\|_{\Y(I)}
                     (\|c\|_{\Y(I)} + \|\mathcal{T}_{\cdot}(W_2)\|_{\Y(I)})  \|u\|_{\X(I)}.
 		\end{align}
 		Similarly, replacing $v_2 u$ by $b \cdot \nabla u$,
        using Sobolev's embedding $H^1 \hookrightarrow L^{\frac{12}{5}}$
        and arguing as in the proof of~\eqref{eq:EstDuhamelBdyvvu}, we obtain
 		\begin{align}
 			\Big\|\int_{t_0}^\cdot e^{\imu(\cdot-s)\Delta} \bdyop(v, b \cdot \nabla u)(s) \dd s \Big\|_{\X(I)}
 			\lesssim& \|v\|_{L^\infty(I;L^2_x)} \|b \cdot \nabla u\|_{L^{\frac{8}{5}}(I;L^{\frac{12}{7}}_x)} \nonumber\\
 			\lesssim& \|v\|_{L^\infty(I;L^2_x)} \|b\|_{L^{8}(I;L^{\frac{12}{5}}_x)} \|\langle \nabla \rangle u\|_{L^2(I;L^6_x)}   \notag \\
 			\lesssim& |I|^{\frac{1}{8}} \|b\|_{L^{\infty}(I;H^1_x)} \|v\|_{\Y(I)} \|u\|_{\X(I)}. \label{eq:EstDuhamelBdyvbnablau}
 		\end{align}
Finally, estimate \eqref{eq:EstCubicLowerOrderTerms} follows from estimates~\eqref{eq:EstDuhamelBdyvcu} and~\eqref{eq:EstDuhamelBdyvbnablau}.
 	\end{proof}

\section{Well-posedness up to the maximal existence time}  \label{Sec-WP-Max-Time}

We prepare the proof of the well-posedness theorem by showing several properties of the noise which we use in the construction of the solutions.

\begin{lemma}  \label{Lem-bc-HolderCont}
Let $T\in (0,\infty)$ and $\kappa\in (0,1/2)$.
Then, $W_1$ is $C^\kappa$-H\"older continuous in $H^3$ and
$W_2$ and the process $t \mapsto \int_0^t e^{\imu(t -s)|\na|} \dd W_2(s)$ are $C^\kappa$-H\"older
continuous in $H^1$.
Moreover, for every $j=1,2,3$ and
for $\bbp$-a.e. $\omega\in \Omega$,
there exists a sequence $(n_l(\omega))_{l \in \N}$ in $\N$ with $n_l(\omega) \rightarrow \infty$ as $l \rightarrow \infty$ such that
\begin{align}  \label{phikbetak-limit-0}
   \sum\limits_{k = n_l}^\infty \int  \sup_{y\in \bbr^2} |\na \phi^{(1)}_k(r \vece_j+y)| \dd r
   \sup\limits_{t\in [0,T]} |\beta^{(1)}_k(t,\omega)|
  \longrightarrow 0,\ \ \text{as}\ l \to \infty.
\end{align}
\end{lemma}

\begin{proof}
We first note that under condition \eqref{phik-condition},
$W_1$ is a continuous and $\{\mathscr{F}_t\}$-adapted Wiener process in $H^3$
(see, e.g., \cite{PR07}).
Moreover, for any $p\geq 1$,
using the Gaussianity, we obtain for $0\leq s<t\leq T$
\begin{align*}
   \bbe \|W_1(t)- W_1(s)\|_{H^3}^p
   \lesssim (\bbe \|W_1(t)- W_1(s)\|_{H^3}^2)^{\frac p2}
   \lesssim \Big(\sum\limits_{k=1}^\infty \|\phi^{(1)}_k\|_{H^3}^2\Big)^\frac p2 (t-s)^{\frac p2}.
\end{align*}
The Kolmogorov continuity criterion thus implies the $C^\kappa$-H\"older continuity of $W_1$ in $H^3$.

The assertion for $W_2$
and $\int_0^\cdot e^{\imu (\cdot -s)|\na|} \dd W_2(s)$ can be proved similarly using condition~\eqref{phik-condition}
and the fact that $e^{\imu t |\na|}$ is unitary in $H^1$.

In order to prove \eqref{phikbetak-limit-0},
we note that
\begin{align*}
   & \bbe  \sum\limits_{k=n}^\infty \int  \sup_{y\in \bbr^2} |\na \phi^{(1)}_k(r \vece_j+y)| \dd r \sup\limits_{t\in [0,T]} |\beta^{(1)}_k(t)|  \notag \\
   &=    \sum\limits_{k=n}^\infty \int  \sup_{y\in \bbr^2} |\na \phi^{(1)}_k(r\vece_j+y)|\dd r \, \bbe \sup\limits_{t\in [0,T]}|\beta^{(1)}_k(t) |
   \lesssim   T^\frac 12  \sum\limits_{k=n}^\infty \int  \sup_{y\in \bbr^2} |\phi^{(1)}_k(r \vece_j+y)| \dd r,
\end{align*}
which tends to zero as $n \to \infty$ due to the summability in~\eqref{phik-condition}.
We infer that
\begin{align*}
   \sum\limits_{k=n}^\infty \int  \sup_{y\in \bbr^2} |\na \phi^{(1)}_k(r \vece_j+y)| \dd r \sup\limits_{t\in [0,T]} |\beta^{(1)}_k(t)|
   \longrightarrow 0 \ \ \text{in\ probability as}\ n \to \infty,
\end{align*}
which yields the almost sure convergence up to some subsequence $(n_l)_{l \in \N}$.
\end{proof}

We are now ready to prove the local well-posedness result in Theorem \ref{Thm-LWP}.
\medskip

\paragraph{\bf Proof of Theorem \ref{Thm-LWP}.}
Fix any $T\in (0,\infty)$.
Write $\X(0,\tau):= \X([0,\tau])$ and $\Y(0,\tau):= \Y([0,\tau])$
for $\tau \in (0,\infty)$.
The proof proceeds in four steps.

{\bf Step $1$.}
We define the fixed point operator $\Phi = (\Phi_1, \Phi_2)$
by the right-hand side of \eqref{eq:RZakharovNormalForm}.
More precisely,
let $\cU_t$ and $\cI_t$ denote the homogeneous and inhomogeneous flow operators,
respectively, for the linear Schr\"odinger equation.
In view of \eqref{eq:RZakharovNormalForm}, we define
\begin{align}  \label{Phi1-def}
  \Phi_1(u,v)(t) &:= \cU_t(u_0 + \bdyop(v_0,u_0))  -\bdyop(v,u) - \imu  \cI_t(v u)_{R}
         + \imu  \cI_t(b \cdot \nabla u + c u - \mathcal{T}_{\cdot}(W_2)u) \notag \\
           &\qquad + \imu \cI_t \bdyop(|\nabla| |u|^2, u)
          - \imu \cI_t \bdyop(v, vu)
          + \imu \cI_t \bdyop(v,  b \cdot \nabla u + cu -\mathcal{T}_{\cdot}(W_2)u)  \notag \\
          &=: \cU_t(u_0 + \bdyop(v_0,u_0)) + \sum\limits_{j=1}^6 \Phi_{1,j}(v,u)(t),
\end{align}
and
\begin{align}    \label{v(u)-def}
    \Phi_2(u,v)(t) :=  \cV_t v_0 + \imu  \int_0^t e^{\imu (t-s) |\nabla|} (|\nabla| |u|^2)(s) \dd s
\end{align}
for $t \in [0,T]$, with $\cV_t = e^{\imu t |\na|}$ and initial data
\begin{align}
    (u_0,v_0):= (X_0, Y_0).
\end{align}

First, we note that Lemma~\ref{lem:QuadraticWave} yields
\begin{align}  \label{v(u)-esti-fp}
   \| \Phi_2(u,v)\|_{\Y(0,\tau)}
   \leq \|Y_0\|_{L^2} + C \tau^\frac 14 \|u\|_{\X(0,\tau)}^2,
\end{align}
where $C$ is the implicit constant from~\eqref{eq:EstQuadraticWave}.

For the first component of the fixed point operator, the multilinear estimates in Section~\ref{Sec-Multi-Esti} imply
\begin{align}  \label{Phi11-esti.0}
  \|\Phi_1(u,v)\|_{\X(0,\tau)}
  &\lesssim  (1+ K^{-1} \|Y_0\|_{L^2}) \|X_0\|_{H^1}
             + (K^{-1} + \tau^{\frac 18})\|v\|_{\Y(0,\tau)}\|u\|_{\X(0,\tau)} \notag \\
            &\qquad  + \tau^\frac 14 K \log_2K \|v\|_{\Y(0,\tau)}\|u\|_{\X(0,\tau)}
            + \Big(\tau^\frac 14\|b\|_{L^\infty((0,\tau),H^1_x)}
             +  \sum\limits_{j=1}^3\|b\|_{L^{1,\infty}_{\vece_j}((0,\tau)\times \bbr^3)}  \notag \\
            &\qquad \qquad + \tau^\frac 14 \|c\|_{L^\infty((0,\tau),H^1_x)}
            + \tau^{\frac 14}  \|\mathcal{T}_{\cdot}(W_2)\|_{L^\infty((0,\tau),H^1_x)} \Big) \|u\|_{\X(0,\tau)}  \notag \\
          &\qquad + \tau^\frac 14 \|u\|_{\X(0,\tau)}^3
            + \tau^\frac 18 \|v\|^2_{\Y(0,\tau)}\|u\|_{\X(0,\tau)} \notag \\
          &\qquad + \tau^\frac 18  (\|b\|_{L^\infty((0,\tau);H^1_x)}
           + \|c\|_{\Y(0,\tau)} +  \|\mathcal{T}_{\cdot}(W_2)\|_{\Y(0,\tau)} ) \|v\|_{\Y(0,\tau)}\|u\|_{\X(0,\tau)}.
\end{align}
We note that by \eqref{b-def} and \eqref{c-def}, we have
\begin{align}
   & \|b(t)\|_{H^1}
      +  \sum\limits_{j=1}^3\|b\|_{L^{1,\infty}_{\vece_j}((0,\tau)\times \bbr^3)} 
               + \|c(t)\|_{H^1} +  \|\mathcal{T}_{t}(W_2)\|_{H^1}  \notag \\
  &\lesssim \|\na W_1(t)\|_{H^1}
            + \sum\limits_{j=1}^3 \sum\limits_{k=1}^\infty \int  \sup_{y\in \bbr^2}
           |\na \phi^{(1)}_k(r \vece_j+y)| \dd r  \sup\limits_{s\in[0,t]}  |\beta^{(1)}_k(s)|
            +  \|W_1(t)\|_{H^3}^2 + \|W_1(t)\|_{H^3} \notag  \\
  &\qquad +   \|\mathcal{T}_{t}(W_2)\|_{H^1}   =: W^*(t).
\end{align}
Recall that $K$ is a dyadic integer with $K \geq 2^5$. We thus obtain for $\tau\leq 1$
\begin{align} \label{Phi1-esti}
  \|\Phi_1(u,v)\|_{\X(0,\tau)}
  &\lesssim  (1+ K^{-1} \|Y_0\|_{L^2}) \|X_0\|_{H^1}
             + (K^{-1} + \tau^{\frac 18} K \log_2K  + \|W^*\|_{C([0,\tau])} ) \|v\|_{\Y(0,\tau)}\|u\|_{\X(0,\tau)}  \notag \\
          &\qquad  +  \tau^\frac 18 \|v\|^2_{\Y(0,\tau)}\|u\|_{\X(0,\tau)}
              + \tau^\frac 18 \|u\|^3_{\X(0,\tau)}
\end{align}
where the implicit constant is independent of $\tau, K, \|X_0\|_{H^1}$, and $\|Y_0\|_{L^2}$.

We next define
\begin{align*}
	M := \max\{10 C_0 \|X_0\|_{H^1}, 10 \|Y_0\|_{L^2}, 1\},
\end{align*}
where $C_0$ is the maximum of the implicit constant in \eqref{Phi1-esti} and the constant in~\eqref{v(u)-esti-fp}, as well as
\begin{align*}
	B_{\X(0,\tau) \times \Y(0,\tau)}(M) := \{(u,v) \in \X(0,\tau) \times \Y(0,\tau) \colon \|u\|_{\X(0,\tau)} +  \|v\|_{\Y(0,\tau)} \leq M\}.
\end{align*}
We then fix a dyadic integer $K = K(\|X_0\|_{H^1}, \|Y_0\|_{L^2}) \geq 2^5$ large enough such that $K^{-1}\|Y_0\|_{L^2} <\frac{1}{10}$ and $K^{-1} C_0 M < \frac{1}{10}$,
and subsequently define the $\{\mathscr{F}_t\}$-stopping time
\begin{align*}
   \tau:=\inf \left\{t\in [0,T] \colon t^\frac 18 C_0 (K\log_2(K)M
           +  M^2) \geq \frac{1}{10},  \
            W^*(t)  C_0 M \geq \frac{1}{10} \right\} \wedge T.
\end{align*}
Employing Lemma~\ref{phikbetak-limit-0}, we note that $W^*(t) \rightarrow 0$ as $t \rightarrow 0$ $\bbp$-a.s., so that $\bbp$-a.s. $\tau > 0$.
The estimates in~\eqref{Phi1-esti} and~\eqref{v(u)-esti-fp} thus yield
\begin{align*}
   \Phi(B_{\X(0,\tau) \times \Y(0,\tau)}(M)) \subseteq B_{\X(0,\tau) \times \Y(0,\tau)}(M).
\end{align*}

\medskip
{\bf Step $2$.}
The contraction of the operator $\Phi$ can be proved similarly.
Actually, take any $(u_1,v_1),(u_2,v_2)\in B_{\X(0,\tau) \times \Y(0,\tau)}(M)$.
We then have
\begin{align*}
   \|u_j\|_{\X(0,\tau)} + \|v_j\|_{\Y(0,\tau)} \leq M = C(\|X_0\|_{H^1}, \|Y_0\|_{L^2}) \quad \text{for } j = 1,2.
\end{align*}
Lemma~\ref{lem:QuadraticWave} yields
\begin{align}  \label{vu1-vu2-diff}
   \|\Phi_2(u_1,v_1) - \Phi_2(u_2,v_2)\|_{\Y(0,\tau)}
   &\leq \left\|\int_0^t e^{\imu (t-s) |\nabla|} (|\na||u_1|^2 - |\na||u_2|^2)(s) \dd s \right\|_{\Y(0,\tau)} \notag \\
   &\leq C_0 \tau^\frac 14 (\|u_1\|_{\X(0,\tau)}+\|u_2\|_{\X(0,\tau)}) \|u_1-u_2\|_{\X(0,\tau)} \notag \\
   &\leq \tau^\frac 18 \|u_1-u_2\|_{\X(0,\tau)},
\end{align}
where we used that $C_0 \tau^\frac 18 M \leq 1/10$ in the last step,
due to the definition of $\tau$.

In the following we will frequently use the multilinear estimates from Section~\ref{Sec-Multi-Esti} and that
$\|u_j\|_{\X(0,\tau)} + \|v_j\|_{\Y(0,\tau)} \leq C(\|X_0\|_{H^1}, \|Y_0\|_{L^2})$.
Hence, the implicit constants below may depend on $\|X_0\|_{H^1}$ and $\|Y_0\|_{L^2}$.

Applying Lemma \ref{lem:Bdy}, we get
\begin{align*}
   \|\Phi_{1,1}(u_1,v_1) - \Phi_{1,1}(u_2,v_2)\|_{\X(0,\tau)}
   &= \|\bdyop(v_1-v_2, u_1) + \bdyop(v_2, u_1-u_2)\|_{\X(0,\tau)} \notag \\
   &\lesssim (K^{-1} + \tau^\frac 18) (\|v_1-v_2\|_{\Y(0,\tau)}\|u_1\|_{\X(0,\tau)} + \|v_2\|_{\Y(0,\tau)}\|u_1 - u_2\|_{\X(0,\tau)}) \notag \\
   &\lesssim (K^{-1} + \tau^\frac 18) \|(u_1,v_1) - (u_2,v_2)\|_{\X(0,\tau) \times \Y(0,\tau)}.
\end{align*}
In the same way, Lemma \ref{lem:Quadratic} yields
\begin{align*}
   \|\Phi_{1,2}(u_1,v_1) -\Phi_{1,2}(u_2,v_2)\|_{\X(0,\tau)}
   &\leq \|\cI_t((v_1 - v_2)u_1)_{R}\|_{\X(0,\tau)}
            +  \|\cI_t(v_2 (u_1-u_2))_{R}  \|_{\X(0,\tau)}  \notag \\
   &\lesssim \tau^\frac 14 K\log_2K \|(u_1,v_1) - (u_2,v_2)\|_{\X(0,\tau) \times \Y(0,\tau)},
\end{align*}
and
\begin{align*}
   \|\Phi_{1,3}(u_1,v_1) -\Phi_{1,3}(u_2,v_2)\|_{\X(0,\tau)}
    \lesssim  \|W^*\|_{C([0,\tau])}  \|u_1-u_2\|_{\X(0,\tau)} .
\end{align*}
We also derive via Lemma \ref{lem:Cubic} that
\begin{align*}
   \|\Phi_{1,4}(u_1,v_1) -\Phi_{1,4}(u_2,v_2)\|_{\X(0,\tau)}
   \lesssim \tau^\frac 14 (\|u_1\|^2_{\X(0,\tau)} + \|u_2\|^2_{\X(0,\tau)} ) \|u_1-u_2\|_{\X(0,\tau)}
   \lesssim \tau^\frac 14 \|u_1-u_2\|_{\X(0,\tau)}
\end{align*}
and
\begin{align*}
   \|\Phi_{1,5}(u_1,v_1) -\Phi_{1,5}(u_2,v_2)\|_{\X(0,\tau)}
   &\lesssim \|\cI_\cdot \bdyop(v_1-v_2, v_1 u_1)\|_{\X(0,\tau)}
            +  \|\cI_\cdot\bdyop(v_2, (v_1-v_2)u_1)\|_{\X(0,\tau)}   \notag \\
           &\qquad +   \|\cI_\cdot\bdyop(v_2, v_2(u_1-u_2))\|_{\X(0,\tau)}  \notag \\
   &\lesssim \tau^\frac 18 \big( \|v_1-v_2\|_{\Y(0,\tau)} \|v_1\|_{\Y(0,\tau)} \|u_1\|_{\X(0,\tau)}  \notag \\
           &\qquad \qquad +  \|v_2\|_{\Y(0,\tau)} \|v_1-v_2\|_{\Y(0,\tau)} \|u_1\|_{\X(0,\tau)}   \notag \\
           &\qquad \qquad +   \|v_2\|_{\Y(0,\tau)}  \|v_2\|_{\Y(0,\tau)} \|u_1-u_2\|_{\X(0,\tau)} \big)   \notag \\
   &\lesssim \tau^\frac 18 \|(u_1,v_1) - (u_2,v_2)\|_{\X(0,\tau) \times \Y(0,\tau)}.
\end{align*}
Regarding the last term $\Phi_{1,6}$, we obtain
\begin{align*}
   \|\Phi_{1,6}(u_1,v_1) -\Phi_{1,6}(u_2,v_2)\|_{\X(0,\tau)}
   &\lesssim \|\cI_\cdot \bdyop(v_1-v_2,  b \cdot \nabla u_1 + c u_1 -\mathcal{T}_{\cdot}(W_2)u_1 ) \|_{\X(0,\tau)}  \notag \\
           &\qquad +  \|\cI_\cdot \bdyop(v_2,  b \cdot \nabla (u_1-u_2) + c(u_1-u_2) -\mathcal{T}_{\cdot}(W_2) (u_1-u_2) ) \|_{\X(0,\tau)} \notag \\
   &\lesssim \tau^\frac 18 \|v_1 - v_2\|_{\Y(0,\tau)}\|W^*\|_{C([0,\tau])} \|u_1\|_{\X(0,\tau)}  \notag \\
           &\qquad  + \tau^\frac 18 \|v_2\|_{\Y(0,\tau)}\|W^*\|_{C([0,\tau])} \|u_1 - u_2\|_{\X(0,\tau)} \notag \\
   &\lesssim \tau^\frac 18 \|(u_1,v_1) - (u_2,v_2)\|_{\X(0,\tau) \times \Y(0,\tau)},
\end{align*}
where we also used the fact that $\|W^*\|_{C([0,\tau])} \lesssim 1$ in the last step.

We thus showed that
\begin{align*}
   &\|\Phi(u_1,v_1) -\Phi(u_2,v_2)\|_{\X(0,\tau) \times \Y(0,\tau)} \\
   &\leq C'(\|X_0\|_{H^1}, \|Y_0\|_{L^2})
         (K^{-1} + \tau^\frac 18 + \tau^\frac 14 K \log_2K
           + \|W^*\|_{C([0,\tau])}) \|(u_1,v_1) - (u_2,v_2)\|_{\X(0,\tau) \times \Y(0,\tau)}.
\end{align*}
Taking $K=K(\|X_0\|_{H^1}, \|Y_0\|_{L^2})$ possibly larger such that
$K^{-1} C'(\|X_0\|_{H^1}, \|Y_0\|_{L^2}) < \frac{1}{10}$, updating the definition of $\tau$,
and defining the $\{\mathscr{F}_t\}$-stopping time
\begin{align}
   \tau_1:= \inf\{t\in [0,T] \colon
    C'(\|X_0\|_{H^1}, \|Y_0\|_{H^1}) (t^\frac 18 K \log_2K + W^*(t)) \geq \frac{1}{10} \} \wedge \tau,
\end{align}
we infer that
$\Phi$ is a contractive selfmapping on $B_{\X(0,\tau_1) \times \Y(0,\tau_1)}(M)$. Note that $\tau_1 > 0$ $\bbp$-a.s. as $W^*(t) \rightarrow 0$ for $t \rightarrow 0$ $\bbp$-a.s.

In conclusion, we obtain that there exists a
small constant $\ve_*(\|X_0\|_{H^1}, \|Y_0\|_{L^2}) > 0$
which is decreasing with respect to its arguments
such that
$\Phi$ is a contractive selfmapping on a closed ball in $\X(0,\tau_1) \times \Y(0,\tau_1)$,
where
\begin{align}
   \tau_1 = \inf\{t\in [0,T]: t\geq \ve_*(\|X_0\|_{H^1}, \|Y_0\|_{L^2}),
   W^*(t) \geq \ve_*(\|X_0\|_{H^1}, \|Y_0\|_{L^2}) \}
    \wedge T
\end{align}
is an $\{\mathscr{F}_t\}$-stopping time.
Hence, there exists
$(\wt u_1, \wt v_1) \in C([0,\tau_1]; H^1 \times L^2) \cap (\X(0,\tau_1) \times \Y(0,\tau_1))$
such that $\Phi_1(\wt u_1, \wt v_1) = (\wt u_1, \wt v_1)$.
Letting $(u_1(t), v_1(t)):= (\wt u_1(t\wedge \tau_1), \wt v_1(t\wedge \tau_1))$ for all $t\in [0,T]$,
we infer that $(u_1, v_1)$ is an $\{\mathscr{F}_t\}$-adapted continuous process in $H^1 \times L^2$ (see e.g.~\cite{BRZ14} for the relevant arguments) and solves \eqref{eq:RZakharovNormalForm} on $[0,\tau_1]$,
thus also \eqref{equa-ranZak}
due to the equivalence results in Propositions \ref{Prop-Equiv-Weak-Mild} and \ref{Prop-Equiv-Mild-Norm}.

\medskip
{\bf Step $3$.}
In order to extend the solution to the maximal existence time,
we use the refined rescaling transforms
and the gluing procedure combined
with induction arguments.

Suppose that for some $n\geq 1$, $(u_n, v_n)$ is an $\{\mathscr{F}_t\}$-adapted continuous process in $H^1\times L^2$
which solves~\eqref{equa-stoZak} on $[0,\sigma_n]$
and satisfies $(u_n, v_n) \equiv (u_n(\sigma_n), v_n(\sigma_n))$ on $[\sigma_n,T]$,
where $\sigma_n (\leq T)$ is an $\{\mathscr{F}_t\}$-stopping time.

Analogous to \eqref{Phi1-def} and~\eqref{v(u)-def},
we define the operator $\Phi_{n+1} = (\Phi_{n+1,1}, \Phi_{n+1,2})$ on $\X(0,T) \times \Y(0,T)$ by
\begin{align}  \label{Phin+1-def}
  \Phi_{n+1,1}(u,v)(t)&:= \cU_t(u_{0,n} + \bdyop(v_{0,n},u_{0,n}))  -\bdyop(v,u) - \imu  \cI_t(vu)_{R}
         + \imu  \cI_t(b_{\sigma_n} \cdot \nabla u + c_{\sigma_n} u - \mathcal{T}_{\sigma_n+\cdot, \sigma_n}(W_2)u) \notag \\
           &\qquad + \imu \cI_t \bdyop(|\nabla| |u|^2, u)
          - \imu \cI_t \bdyop(v, vu)
          + \imu \cI_t \bdyop(v,  b_{\sigma_n} \cdot \nabla u + c_{\sigma_n} u - \mathcal{T}_{\sigma_n+\cdot, \sigma_n}(W_2)u)
\end{align}
and
\begin{align}
   \Phi_{n+1,2}(u,v)(t) :=  \cV_t v_0 + \imu  \int_0^t e^{\imu (t-s) |\nabla|} (|\nabla| |u|^2)(s) \dd s
\end{align}
for all $t \in [0,T]$, where
$ \mathcal{T}_{\sigma_n+\cdot, \sigma_n}(W_2)$ is given by \eqref{Tsigmat-W2},
the new perturbation coefficients $b_{\sigma_n}$ and  $c_{\sigma_n}$ are defined as in \eqref{b-def} and \eqref{c-def},
respectively,
with $W_1$ replaced by $W_{1,\sigma_n}$ defined in \eqref{W-sigma-rescal}, i.e.,
\begin{align*}
   b_{\sigma_n}(t) = 2 \na W_{1,\sigma_n}(t) , \ \
   c_{\sigma_n}(t) = |\na W_{1,\sigma_n}(t)|^2 + \Delta W_{1,\sigma_n}(t),\ \ t\in [0,T],
\end{align*}
and the initial data is given by
\begin{align*}
   (u_{0,n}, v_{0,n}) := (e^{W_1(\sigma_n)} u_n(\sigma_n), v_n(\sigma_n) + \cT_{\sigma_n}(W_2)).
\end{align*}
Furthermore, we set
\begin{align}  \label{W*sigman-def}
  W^*_{\sigma_n}(t)
  &:= \|W_{1,\sigma_n}(t) \|_{H^3}
     +  \sum\limits_{j=1}^3 \sum\limits_{k=1}^\infty \int  \sup_{y\in \bbr^2}
         |\na \phi^{(1)}_k(r \vece_j+y)| \dd r
         \sup\limits_{s\in[0,t]}  |\beta^{(1)}_k(\sigma_n+s) - \beta^{(1)}_k(\sigma_n)|
   \notag \\
  &\qquad + \|W_{1,\sigma_n}(t) \|_{H^3}^2
  + \|\cT_{\sigma_n+t, \sigma_n}(W_2)\|_{H^1}.
\end{align}
We define the $\{\mathcal{F}_{\sigma_n+t}\}$-stopping time
\begin{align*}
   \tau_{n+1} &:= \inf\{t\in [0,T] \colon t\geq \ve_*(\|u_{0,n}\|_{H^1}, \|v_{0,n}\|_{L^2}),
    W^*_{\sigma_n}(t) \geq \ve_*(\|u_{0,n}\|_{H^1}, \|v_{0,n}\|_{L^2})  \}
   \wedge (T-\sigma_n).
\end{align*}
and
\begin{align*}
   \sigma_{n+1} : = \sigma_n + \tau_{n+1}.
\end{align*}
Then $\sigma_{n+1}$  is an $\{\mathcal{F}_t\}$-stopping time,
see e.g. \cite{BRZ14} for the relevant arguments, and $\sigma_{n+1} \leq T$.

Proceeding as in Step $1$ and Step $2$,
we infer that $\Phi_{n+1}$ is a contractive selfmapping on a closed subset of
$\X(0,\tau_{n+1}) \times \Y(0,\tau_{n+1})$,
and thus there exists $(\wt u_{\sigma_{n+1}}, \wt v_{\sigma_{n+1}})  \in \X(0,\tau_{n+1}) \times \Y(0,\tau_{n+1})$
such that
$\Phi_{n+1} (\wt u_{\sigma_{n+1}}, \wt v_{\sigma_{n+1}}) = (\wt u_{\sigma_{n+1}}, \wt v_{\sigma_{n+1}})$.
In particular,
setting
\begin{align*}
	(u_{\sigma_{n+1}}(t),v_{\sigma_{n+1}}(t)):=(\wt u_{\sigma_{n+1}}(t\wedge\tau_{n+1}),\wt v_{\sigma_{n+1}}(t\wedge\tau_{n+1}))
\end{align*}
for all
$t\in [0,T]$,
we infer that $(u_{\sigma_{n+1}}, v_{\sigma_{n+1}})$ is an $\{\mathscr{F}_{\sigma_n+t}\}$-adapted
continuous process in $H^{1}\times L^2$
and solves \eqref{equa-Zakha-sigma} on $[0,\tau_{n+1}]$,
due to Propositions \ref{Prop-Equiv-Weak-Mild} and \ref{Prop-Equiv-Mild-Norm}.
	
We then use Proposition~\ref{Prop-u1-usigma-glue} to glue together $(u_n,v_n)$ and $(u_{\sigma_{n+1}}, v_{\sigma_{n+1}})$, i.e.,
	 \begin{align*}
		u_{n+1}(t)&:= u_n(t) \chi_{[0,\sigma_n)}(t)
          + e^{-W_1(\sigma_n)}u_{\sigma_{n+1}}( (t-\sigma_n) \wedge \tau_{n+1} ) \chi_{[\sigma_n, T]}(t),   \\
		v_{n+1}(t)&:= v_n(t) \chi_{[0,\sigma_n)}(t)
		+ \(v_{\sigma_{n+1}}( (t-\sigma_n) \wedge \tau_{n+1} ) - e^{\imu ( (t-\sigma_n) \wedge \tau_{n+1} )|\na|}\cT_{\sigma_n}(W_2)\)
		   \chi_{[\sigma_n, T]}(t).
	 \end{align*}
We obtain that
$(u_{n+1}, v_{n+1})$ is an $\{\mathscr{F}_t\}$-adapted continuous process in $H^{1}\times L^2$
(see \cite{BRZ14} for the relevant arguments),
$(u_{n+1}, v_{n+1}) = (u_n,v_n)$ on $[0,\sigma_n)$,
and $(u_{n+1},v_{n+1}) \equiv (u_{n+1}(\sigma_{n+1}), v_{n+1}(\sigma_{n+1}))$
on $[\sigma_{n+1}, T]$.
Moreover, in view of Proposition \ref{Prop-u1-usigma-glue},
$(u_{n+1}, v_{n+1})$ solves \eqref{equa-ranZak} on
the larger time interval $[0, \sigma_{n+1}]$.
	
Proceeding inductively, we thus get an increasing sequence of  $\{\mathscr{F}_t\}$-stopping times $\{\sigma_n\}$ and an $\{\mathscr{F}_t\}$-stopping time  $\tau^*_T (\leq T)$, as well as corresponding
$\{\mathscr{F}_t\}$-adapted continuous processes
	 $(u_n,v_n) \in C([0,T]; H^1 \times L^2)$,
	 $n\geq 1$,
	 such that
	 $\sigma_n \to \tau^*_T$ as $n\to \infty$,
	 $(u_n,v_n)$ solves \eqref{equa-stoZak} on
	 $[0,\sigma_n]$,
	 and $(u_{n+1}, v_{n+1})$ coincides with $(u_n, v_n)$ on $[0,\sigma_n]$ for all $n\geq 1$.
	 In particular,
	 setting $u^T:= \lim_{n \rightarrow \infty} u_n \chi_{[0,\tau_T^*)}$ and $v^T := \lim_{n \rightarrow \infty} v_n \chi_{[0,\tau_T^*)}$,
	 we obtain an analytically weak solution $(u^T,v^T)$ to \eqref{equa-ranZak} on $[0,\tau^*_T)$.
As the uniqueness of solutions is a local property,
it can be derived using similar arguments as in Step $2$.

It is clear that $\tau^*_T$ is increasing with $T$,
and consequently there exists an $\{\mathscr{F}_t\}$-stopping time $\tau^*$
such that $\tau^*_T \to \tau^*$ as $T\to \infty$, $\bbp$-a.s. We set $u := \lim_{T \rightarrow \infty} u^T \chi_{[0,\tau^*)}$ and $v := \lim_{T \rightarrow \infty} v^T \chi_{[0,\tau^*)}$.
By the uniqueness of solutions on $[0,\tau^*_T)$, we thus obtain a unique solution $(u,v)$ to \eqref{equa-ranZak}
on $[0,\tau^*)$.

Using Theorem~\ref{Thm-stocha-weak}, we finally get a unique solution $(X,Y)$ on $[0,\tau^*)$ in the sense of Definition~\ref{Def-weak-sol} by setting $(X,Y) = (e^{W_1} u, v + \cT_{\cdot}(W_2))$.

\medskip
{\bf Step $4$.} It remains to prove the blowup in \eqref{blowup-alter}
if $\tau^*<\infty$. Note that by the boundedness of $\|e^{W_1}\|_{H^2}$ and of $\|\cT_{\cdot}(W_2)\|_{L^2}$ on bounded time intervals, the blow-up alternative~\eqref{blowup-alter} is equivalent to \linebreak
$\lim_{t\to \tau^*} (\|u(t)\|_{H^1} + \|v(t)\|_{L^2}) = \infty$.

Below we consider $\omega\in \Omega$ such that $\tau^*(\omega)<\infty$ and
the convergence \eqref{phikbetak-limit-0} holds.
For simplicity, the dependence on $\omega$ is omitted in the following.

Let us first prove that
if $\tau^*<\infty$ then
\begin{align}  \label{uv-infty-tau*}
	\limsup_{t\to \tau^*} (\|u(t)\|_{H^1} + \|v(t)\|_{L^2}) = \infty.
\end{align}
We argue by contradiction.
Suppose that $\tau^*<\infty$  and
\begin{align*}
	\limsup_{t\to \tau^*} (\|u(t)\|_{H^1} + \|v(t)\|_{L^2}) <\infty.
\end{align*}
Then there exists a constant $C^*$ such that
\begin{align*}
	\|u_{0,n}\|_{H^1} + \|v_{0,n}\|_{L^2} &= \|e^{W_1(\sigma_n)} u_n(\sigma_n)\|_{H^1} + \|v_n(\sigma_n) + \cT_{\sigma_n}(W_2))\|_{L^2} \\
	&\leq (\| W_1(\sigma_n)\|_{H^2} + 1) \|u_n(\sigma_n)\|_{H^1} + \|v_n(\sigma_n)\|_{L^2} + \|\cT_{\sigma_n}(W_2)\|_{L^2} \leq C^* < \infty
\end{align*}
for all $n \in \N$.
Let $T (=T(\omega)) \in (0,\infty)$ be such that $\tau^*<T$
and $\{\sigma_n, \tau_n\}$ be constructed as above.
Then, by construction,
$\sigma_n<T$ for any $n\geq 1$,
and since  $\ve_*(\cdot, \cdot)$ is decreasing with respect to the arguments,
we have either
\begin{align} \label{taun-ve*}
   \tau_{n+1} =  \ve_*(\|u_{0,n}\|_{H^1}, \|v_{0,n}\|_{L^2})
   \geq \ve_*(C^*, C^*),
\end{align}
or
\begin{align} \label{W*sigman-ve*}
    W^*_{\sigma_n}(\tau_{n+1})= \ve_*(\|u_{0,n}\|_{H^1}, \|v_{0,n}\|_{L^2})
    \geq \ve_*(C^*, C^*).
\end{align}

We first discuss the case~\eqref{W*sigman-ve*}.
In view of the convergence \eqref{phikbetak-limit-0},
there exists $n_l \in \N$ such that
\begin{align*}
    \sum\limits_{j=1}^3 \sum\limits_{k=n_l+1}^\infty \int  \sup_{y\in \bbr^2}
    |\na \phi^{(1)}_k(r \vece_j + y)| \dd r \sup\limits_{t\in [0,T]} |\beta^{(1)}_k(t)|
    \leq \frac{1}{20} \ve_*(C^*, C^*).
\end{align*}
Moreover, by the $C^\kappa$-H\"older continuity of Brownian motions,
we infer for the first $n_l$ modes of the noise
\begin{align}
     & \sum\limits_{j=1}^3 \sum\limits_{k=1}^{n_l}
    \int  \sup_{y\in \bbr^2} |\na \phi^{(1)}_k(r \vece_j+y)| \dd r
    \sup\limits_{s \in [0,\tau_{n+1}]}|\beta^{(1)}_k(\sigma_n+s) - \beta^{(1)}_k(\sigma_n)|   \notag \\
   &\leq \sum\limits_{j=1}^3 \sum\limits_{k=1}^{n_l}
    \int  \sup_{y\in \bbr^2} |\na \phi^{(1)}_k(r \vece_j+y)| \dd r\,
          \wt C(k,\kappa,T) \tau_{n+1}^{\kappa}
    =: \wt C(n_l,\kappa, T)  \tau_{n+1}^{\kappa},
\end{align}
where $\wt C(k,\kappa,T) $ is the (random) $C^\kappa$-H\"older norm of $\beta^{(1)}_k$ for $1\leq k\leq n_l$.
Hence, we obtain
\begin{align*}
   &  \sum\limits_{j=1}^3 \sum\limits_{k=1}^\infty \int  \sup_{y\in \bbr^2}
        |\na \phi^{(1)}_k(r \vece_j+y)| \dd r \sup\limits_{s\in[0,\tau_{n+1}]}  |\beta^{(1)}_k(\sigma_n+s) - \beta^{(1)}_k(\sigma_n)|  \\
   &\leq \frac{1}{10} \ve_*(C^*, C^*) + \wt C(n_l, \kappa, T)  \tau_{n+1}^{\kappa}.
\end{align*}
The H\"older continuity of the noise in Lemma \ref{Lem-bc-HolderCont}
also yields that
there exists $\wt C'(\kappa,T)$ such that
\begin{align*}
    & \|W_{1,\sigma_n}(\tau_{n+1}) \|_{H^3}
     + \|W_{1,\sigma_n}(\tau_{n+1}) \|_{H^3}^2
  + \|\cT_{\sigma_n+\tau_{n+1}, \sigma_n}(W_2)\|_{H^1} \notag \\
   &=   \|W_1(\sigma_{n}+\tau_{n+1}) - W_1(\sigma_n)\|_{H^3}
     + \| W_1(\sigma_{n}+\tau_{n+1}) -  W_1(\sigma_n)\|_{H^3}^2 \notag \\
    &\qquad + \Big \|\int_{\sigma_n}^{\sigma_{n}+\tau_{n+1}} e^{\imu(\sigma_n+\tau_{n+1} -s)|\na|} \dd W_2(s) \Big\|_{H^1}  \notag \\
  &\leq \wt C'(\kappa,T) \tau_{n+1}^\kappa.
\end{align*}
Combining the above two estimates and the definition of $W^*_{\sigma_n}$ in \eqref{W*sigman-def},
we infer
\begin{align*}
   W^*_{\sigma_n}(\tau_{n+1})
   \leq  \frac{1}{10} \ve_*(C^*, C^*)
           + ( \wt C^*(n_l, \kappa, T)   +  \wt C'(\kappa,T)) \tau_{n+1}^\kappa.
\end{align*}
If~\eqref{W*sigman-ve*} holds,
we thus get the uniform lower bound
\begin{align} \label{W*sigman-ve*-esti}
   \tau_{n+1} \geq \( \frac{9}{10} \ve_*(C^*, C^*) ( \wt C^*(n_l, \kappa, T)   +  \wt C'(\kappa,T))^{-1}\)^{\frac 1 \kappa}.
\end{align}

Since~\eqref{taun-ve*} also gives such a lower bound for the extension time $\tau_{n+1}$,
we conclude that
$\sigma_{n+1}$ reaches $T$ after finitely many steps,
i.e., $\sigma_N = T$ for some $N$ large enough.
Consequently, we have the contradiction
$\tau^*< T =\sigma_N<\tau^*$,
which proves~\eqref{uv-infty-tau*}, as claimed.

In order to prove \eqref{blowup-alter}, we have to replace the limit superior in~\eqref{uv-infty-tau*} by the limit. To that purpose take any sequence $\wt \sigma_n$ tending to $\tau^*$
and satisfying
\begin{align*}
	\lim_{\wt \sigma_n\to \tau^*} (\|u(\wt \sigma_n)\|_{H^1} + \|v(\wt \sigma_n)\|_{L^2}) <\infty.
\end{align*}
Then, the energy norm of the sequence $(u(\wt \sigma_n),v(\wt \sigma_n))_n$ is uniformly bounded.
Arguing as in Step~3, we can construct extension times $(\wt \tau_{n+1})_n$ such that
\eqref{equa-ranZak} is uniquely solvable on $[0, \wt \sigma_n + \wt \tau_{n+1}]$ for every $n\geq 1$.
As above one sees that there is a lower bound for the extension times $(\wt \tau_{n+1})_n$. Consequently, $\wt \sigma_n + \wt \tau_{n+1} > \tau^*$ for some $n \in \N$. 
In particular,
this implies $\|u\|_{L^\infty([0,\tau^*); H^1)} + \|v\|_{L^\infty([0,\tau^*); L^2)} <\infty$,
which contradicts \eqref{uv-infty-tau*}.

Therefore, we obtain \eqref{blowup-alter} and finish the proof of Theorem \ref{Thm-LWP}.
\hfill $\square$
\medskip

\section{Well-posedness below the ground state}   \label{Sec-WP-Ground}

In this section we prove Theorem \ref{Thm-GWP-Ground}
 using the variational analysis of the ground state $Q$.

Recall that $E_Z$ and $E_S$, introduced in~\eqref{EZ-def} and~\eqref{ES-Hamil},
are the energies of the Zakharov system~\eqref{eq:ZakharovSystemFirstOrder}
and the focusing cubic NLS~\eqref{equa-NLS-cubic}, respectively.
Also recall the action functional $J$ from~\eqref{J-def}
and the corresponding scaling derivative functional $K$ from~\eqref{K-def}.

For any $\lbb>0$, set $Q_\lbb (x):= \lbb Q(\lbb x)$, where $Q$ is the ground state from~\eqref{Ground-def}.
Then $Q_\lbb$ minimizes the action
\begin{align*}
   J_\lbb := E_S + \lbb^2 M,
\end{align*}
that is, $Q$ has the variational characterization
\begin{align} \label{Jlbb-variat}
   \lbb J(Q) = J_\lbb (Q_\lbb)
    = \inf\{J_\lbb(\vf): \vf \not =0,\ K(\vf) =0\},
\end{align}
see~\cite{GNW13}.

In order to prove Theorem~\ref{Thm-GWP-Ground}, we adapt arguments from~\cite{GNW13}, which are based on the variational characterization~\eqref{Jlbb-variat} of
the ground state and the mass conservation of the Schr\"odinger component.

\paragraph{\bf Proof of Theorem \ref{Thm-GWP-Ground}}
In view of the blow-up alternative in Theorem \ref{Thm-LWP},
it is sufficient to prove for $\bbp$-a.e. $\omega\in \Omega$
the boundedness of the $H^1\times L^2$-norm of the solution
$(X(t,\omega),Y(t,\omega))$ on $[0,\sigma_*(\omega) \wedge \tau^*(\omega))$,
i.e.,
\begin{align}  \label{uv-H1L2-bdd}
   \sup\limits_{t \in [0,\sigma_*(\omega) \wedge \tau^*(\omega))}
   (\|X(t,\omega)\|_{H^1} + \|Y(t,\omega)\|_{L^2})<\infty.
\end{align}
For simplicity, we omit the dependence on $\omega$ in the following.

Without loss of generality we assume $X_0 \neq 0$, as otherwise the uniqueness statement in Theorem~\ref{Thm-LWP} leads to $X \equiv 0$ which in turn implies~\eqref{uv-H1L2-bdd}.  We next note that the conservation of mass \eqref{mass-conserv} yields for any $\lbb > 0$ and $t \in [0,\sigma_* \wedge \tau^*)$
\begin{align*}
   E_Z(X(t), Y(t)) + \lbb^2 M(X(t)) - \lbb J(Q)
   = M(X_0) \(\(\lbb - \frac{J(Q)}{2M(X_0)}\)^2 + \frac{4E_Z(X(t),Y(t)) M(X_0) - J^2(Q)}{4 M^2(X_0)}\).
\end{align*}
Using the definition of the stopping time $\sigma_*$ in \eqref{sigma*-def}
and the conservation of mass again, we derive for any $t \in [0, \sigma_* \wedge \tau^*)$
\begin{align*}
   4 E_Z(X(t), Y(t)) M(X_0)
   = 4  E_Z(X(t), Y(t)) M(X(t)) < 4 E_S(Q) M(Q)
   = J^2 (Q),
\end{align*}
where the last identity is a property of the ground state which follows from the variational characterization in~\eqref{Jlbb-variat},
see~\cite[p. 417]{GNW13}.

Hence, we derive that for $\lbb_*:= \frac{J(Q)}{2M(X_0)}$,
independent of $\omega$ and $t$,
\begin{align} \label{Ez-M-J-esti}
   E_Z(X(t), Y(t)) + \lbb_*^2 M(X(t)) < \lbb_* J(Q)  = J_{\lbb_*} (Q_{\lbb_*})
\end{align}
for all $t \in [0, \sigma_* \wedge \tau^*)$, which yields via \eqref{EZ-Hamil}
\begin{align} \label{Jlbbu-JlbbQ}
   J_{\lbb_*} (X(t)) \leq E_Z(X(t), Y(t)) + \lbb_*^2 M(X(t))
   < J_{\lbb_*}(Q_{\lbb_*})
\end{align}
for all $t \in [0, \sigma_* \wedge \tau^*)$. Suppose that $K(X(t_0)) = 0$ for some $t_0\in [0, \sigma_* \wedge \tau^*)$ and $X(t_0)\not = 0$.
Then the variational characterization \eqref{Jlbb-variat} of $Q_{\lbb_*}$ implies that
$J_{\lbb_*}(Q_{\lbb_*}) \leq J_{\lbb_*}(X(t_0))$,
which, however, leads to a contradiction with \eqref{Jlbbu-JlbbQ}.

Consequently, we obtain that for any $t \in [0, \sigma_* \wedge \tau^*)$,
\begin{align}  \label{K-positive}
   K(X(t)) =0 \Leftrightarrow X(t) =0.
\end{align}

If $K(X_0)=0$,
\eqref{K-positive} yields a contradiction to $X_0 \neq 0$.

In the case $K(X_0) > 0$,
we infer
$K(X(t))>0$ for all $t \in [0, \sigma_* \wedge \tau^*)$ from~\eqref{K-positive}.
Combining~\eqref{EZ-Hamil} and~\eqref{K-def} with~\eqref{Ez-M-J-esti}, we thus arrive at
\begin{align*}
   \frac 16 \|\na X(t)\|_{L^2}^2 + \frac{\lbb_*^2}{2}\|X(t)\|_{L^2}
   + \frac 14 \|Y(t) + |X(t)|^2 \|_{L^2}^2
   =& E_Z(X(t), Y(t)) + \lbb_*^2 M(X(t)) - \frac 13 K(X(t))
   < \lbb_* J(Q)
\end{align*}
for all $t \in [0, \sigma_* \wedge \tau^*)$.
In particular, we get the uniform bound
\begin{align} \label{eq:GSH1Bound}
   \|X(t)\|_{H^1} +  \|Y(t) + |X(t)|^2 \|_{L^2}^2  \lesssim 1
\end{align}
on $[0, \sigma_* \wedge \tau^*)$. Employing the Sobolev embedding $H^1 \hookrightarrow L^4$, we further derive
\begin{align} \label{eq:GSL2Bound}
   \|Y(t)\|_{L^2}^2 \lesssim \|Y(t) + |X(t)|^2 \|_{L^2}^2 + \|X(t)\|_{L^4}^4
   \lesssim  \|Y(t) + |X(t)|^2 \|_{L^2}^2 + \|X(t)\|_{H^1}^4
   \lesssim 1,
\end{align}
 for all $t \in [0, \sigma_* \wedge \tau^*)$. The estimates in~\eqref{eq:GSH1Bound} and~\eqref{eq:GSL2Bound} thus imply the $H^1 \times L^2$-boundedness \eqref{uv-H1L2-bdd} of the solution.
\hfill $\square$

\section{Regularization by noise}  \label{Sec-Noise-Reg}

In this section we prove Theorem \ref{Thm-Noise-Reg}
concerning the regularization effect of the {non-conservative} noise
on the well-posedness of the Zakharov system.

As mentioned in Section~\ref{Sec-Intro},
the key point is to prove Strichartz estimates on a bounded time interval for a Schr\"odinger equation with
a nonlocal potential term involving the linear wave $v_1(t) := e^{\imu t|\na|}Y_0$ and the boundary term operator $\bdyop$.

For any interval $I\subseteq [0,T]$,
we define the spaces
\begin{align*}
	S^1(I) := \{f \in C(I; H^1_x) \colon \|f\|_{S^1(I)} < \infty \}, \quad N^1(I) := \{g \in L^1(I; L^2_x) \colon \|g\|_{N^1(I)} < \infty\}
\end{align*}
endowed with the norms
	\begin{align*}
		\|f\|_{S^1(I)} &:= \|f\|_{L^\infty(I;H^1_x)} + \Big(\sum_{N \in 2^\Z} \langle N \rangle^2 \|P_N f\|_{L^2(I;L^6_x)}^2 \Big)^{\frac{1}{2}}, \\
		\|g\|_{N^1(I)} &:= \Big(\sum_{N \in 2^Z} \|P_N g\|_{G_N(I)}^2\Big)^{\frac{1}{2}},
	\end{align*}
where
\begin{align*}
		\|g\|_{G_N(I)} := \inf_{g = g_1 + g_2}\Big( \langle N \rangle \|g_1\|_{L^1(I;L^2_x)}
      + \langle N \rangle \|g_2\|_{L^{\frac{8}{5}}(I;L^{\frac{4}{3}}_x)} \Big), \ \ N \in 2^{\Z}.
\end{align*}
We point out that $S^1(0,T) \subseteq C(I;H^1)$.
Moreover, we still denote $C(I; L^2)$ by $\Y(I)$.

Note that compared with the $\X(I)$- and the $\G(I)$-norm from Section~\ref{Sec-Multi-Esti}, we just dropped the local smoothing component. The local smoothing component was necessary to treat the lower order perturbation terms arising from the conservative noise in~\eqref{equa-ranZak-bc}, which do not appear in~\eqref{equa-ranZak-noncons} as the nonconservative noise is a one-dimensional Brownian motion.

We denote the flow operators
of the homogeneous and inhomogeneous Schrödinger equation with initial time $t_0$ by $\cI_{\cdot; t_0}$ and $\cU_{\cdot; t_0}$, respectively, i.e.,
\begin{align*}
	\cU_{t; t_0}f = e^{\imu (t - t_0) \Delta} f, \qquad \cI_{t; t_0} g = \int_{t_0}^t e^{\imu(t-s)\Delta} g(s) \dd s.
\end{align*}
In the case $t_0 = 0$ we drop $t_0$ from our notation and simply write $\cU$ and $\cI$ for these operators.

Our first observation is that we obtain corresponding multilinear estimates to the ones in Section~\ref{Sec-Multi-Esti} without the local smoothing component, i.e., in the space $S^1(I)$. To be more precise,
arguing as in the proofs of Lemma~\ref{lem:Bdy}, Lemma~\ref{lem:Quadratic}, and Lemma~\ref{lem:Cubic}, we infer the following multilinear estimates.

\begin{lemma} \label{Lem-Stri-cUcT}
Let $I \subseteq \R$ be an interval and $t_0 = \inf I$. We then have 
\begin{align}
		& \|\bdyop(v,u)\|_{S^1(I)} \lesssim (K^{-1} + |I|^{\frac{1}{8}})\|v\|_{\Y(I)} \|u\|_{S^1(I)},    \label{eq:EstBdyTime-S1} \\
        & \Big\|\cI_{\cdot; t_0}(v u)_R \Big\|_{S^1(I)}
               \lesssim  |I|^{\frac{1}{4}} K \log_2 K \|v\|_{\Y(I)} \|u\|_{S^1(I)}, \label{eq:EstQuadraticR-S1}  \\
         &\Big\| \cI_{\cdot, t_0} \bdyop(w_1, w_2 u) \Big\|_{S^1(I)}
          \lesssim |I|^{\frac{1}{8}} \|w_1\|_{\Y(I)} \|w_2\|_{\Y(I)} \|u\|_{S^1(I)}. \label{eq:EstCubicvvu-S1}
\end{align}
Moreover, for any $h\in L^4(I)$ we have
\begin{align}
&\Big\|\cI_{\cdot, t_0} \bdyop(h |\nabla| (u_1 u_2), u_3) \Big\|_{S^1(I)}
       \lesssim \|h\|_{L^4(I)} \|u_1\|_{S^1(I)} \|u_2\|_{S^1(I)} \|u_3\|_{S^1(I)}, \label{eq:EstCubicuuu-S1} \\
        & \Big\|\int_{t_0}^{\cdot} e^{\imu (\cdot - s) |\nabla|} (h(s) |\nabla| (u_1 u_2)(s)) \dd s \Big\|_{\Y(I)}
             \lesssim \|h\|_{L^4(I)} \|u_1\|_{S^1(I)} \|u_2\|_{S^1(I)}.     \label{eq:EstQuadraticWave-S1}         
\end{align}
\end{lemma}

We now provide the crucial estimate for solutions of a type of Schrödinger equation with a nonlocal potential involving the linear wave $v_1$. It is more precisely stated in its mild formulation, see~\eqref{equa-z-Omegav2z-mild}.

\begin{lemma}
	\label{Lem-Est-v1Pot}
	Let $T\in (0,\infty)$, $Y_0 \in L^2$, and $v_1(t) = e^{\imu t |\nabla|} Y_0$ for all $t \in [0,T]$. Let $g \in S^1(0,T)$. Then the equation
	\begin{equation}
		\label{eq:SchrTypeNonlocPot}
		z(t) + \bdyop(v_1, z)(t) + \imu \cI_t(v_1 z)_R + \imu \cI_t \Omega_b(v_1, v_1 z) =  g(t)
	\end{equation}
	has a unique solution $z \in S^1(0,T)$ satisfying~\eqref{eq:SchrTypeNonlocPot} for all $t \in [0,T]$.
Moreover, there is a constant $C = C(T, \|Y_0\|_{L^2}) > 0$, increasing in both its arguments, such that 
\begin{align*}
	\|z\|_{S^1(0,T)} \leq  C \|g\|_{S^1(0,T)}.
\end{align*}
\end{lemma}

\begin{proof}
Let $C_1 \geq 1$ be the maximum of the constants in Lemma~\ref{Lem-Stri-cUcT}, $C_2 \geq 1$ be the constant in Lemma~\ref{lem:LinearEstimates}~\ref{it:LinStrichartz}, and $C_0 = C_1 C_2 \geq 1$. If $C_0 \|Y_0\|_{L^2} \geq 8$, we fix $K \in 2^{\N}$ such that 
\begin{align*}
	4 C_0 \|Y_0\|_{L^2} \leq K \leq 8 C_0 \|Y_0\|_{L^2}.
\end{align*}
Otherwise we take $K = 2^5$. Next we define the constant $c = c(T,\|Y_0\|_{L^2})$ by
\begin{equation}
	\label{eq:DefcTY0}
	c(T,\|Y_0\|_{L^2}) = C_0 \|Y_0\|_{L^2} + C_0 \|Y_0\|_{L^2}^2 + 64 C_0 \|Y_0\|_{L^2} T^{\frac{1}{8}} \max\{C_0 \|Y_0\|_{L^2}, 8\}^2,
\end{equation}
which is increasing in both its arguments. For later reference we already note that
\begin{equation}
	\label{eq:EstKlogK}
	K \log_2 K \leq 64 \max\{C_0 \|Y_0\|_{L^2}, 8\}^2
\end{equation}
by the definition of $K$.

We define a partition $0 = t_0 < t_1 < \ldots < t_L < t_{L+1} = T$ of $[0,T]$ by setting $t_0 := 0$ and
\begin{align} \label{tn-partition-def}
   t_{n+1} := \inf\left\{t\in(t_n, T]: 4 c(T, \|Y_0\|_{L^2})(t - t_{n})^\frac 18 \geq 1 \right\} \wedge T.
\end{align}
Note that
\begin{align}   \label{L-partition-bdd}
    L &\leq  (4 c(T, \|Y_0\|_{L^2}))^8 T,
\end{align}
where the right-hand side is increasing in both $T$ and $\|Y_0\|_{L^2}$.

Obviously, $z$ satisfies~\eqref{eq:SchrTypeNonlocPot} if and only if $z$ solves the equation
\begin{equation}
\label{eq:zFixedPoint}
	z(t) =  -\bdyop(v_1, z)(t) - \imu \cI_t(v_1 z)_R - \imu \cI_t \bdyop(v_1, v_1 z) + g(t).
\end{equation}
We define the fixed point operator $z \mapsto \Phi(z;g)$ by the right-hand side of~\eqref{eq:zFixedPoint}.
We will solve~\eqref{eq:zFixedPoint} by successively constructing fixed points on the intervals $[0,t_n]$. To that purpose, we first set
\begin{align*}
	R_n = 4^{n+1} (1+c(T,\|Y_0\|_{L^2}) T^{\frac{1}{8}})^n \|g\|_{S^1(0,T)}
\end{align*}
for all $n \in \{0, 1, \ldots, L\}$. Now fix such an index $n$ and assume that we have already constructed a fixed point $z_n$ of $\Phi(\cdot; g)$ on the interval $[0, t_n]$, which is unique in $S^1([0,t_n])$ and satisfies $\|z_n\|_{S^1(0,t_n)} \leq R_n$. We then define
\begin{align*}
	B_{n+1} = \{w \in S^1([0, t_{n+1}]) \colon w = z_n \text{ on } [0, t_n], \, \, \|z\|_{S^1(0,t_{n+1})} \leq R_{n+1} \}.
\end{align*}
Here and in the following we use the convention that in the case $n = 0$ the assumptions involving $z_n$ drop. Note that $B_{n+1}$ equipped with the metric induced by $\| \cdot \|_{S^1(0,t_{n+1})}$ is a complete metric space.
We will now show that $\Phi$ is a contractive self-mapping on $B_{n+1}$. To that purpose we first estimate
\begin{align}
	\|\Phi(z;g)\|_{S^1(0,t_{n+1})} &\leq \|\Phi(z;g)\|_{S^1(0,t_n)} + \|\Phi(z;g)\|_{S^1(t_n, t_{n+1})} 
		= \|z\|_{S^1(0,t_n)} + \|\Phi(z;g)\|_{S^1(t_n, t_{n+1})} \label{eq:EstSplittingPhig}
\end{align}
for all $z \in S^1([0,t_{n+1}])$ with $z = \Phi(z;g)$ on $[0,t_n]$.
In order to control $\Phi(z;g)$ on $[t_n, t_{n+1}]$, we first decompose the Duhamel integrals as $\cI_t = \cU_{t;t_n}\cI_{t_n} + \cI_{t;t_n}$. Consequently, we get
\begin{align}
	\|\Phi(z;g)\|_{S^1(t_n, t_{n+1})} &\leq \| \bdyop(v_1,z)\|_{S^1(t_n, t_{n+1})} + \|\cU_{\cdot;t_n} \cI_{t_n}(v_1 z)_R\|_{S^1(t_n, t_{n+1})} + \|\cU_{\cdot;t_n} \cI_{t_n} \bdyop(v_1, v_1 z)\|_{S^1(t_n, t_{n+1})} \nonumber \\
	&\qquad + \|\cI_{\cdot;t_n}(v_1 z)_R\|_{S^1(t_n, t_{n+1})} + \|\cI_{\cdot; t_n} \bdyop(v_1, v_1 z)\|_{S^1(t_n, t_{n+1})} + \|g\|_{S^1(t_n, t_{n+1})}. \label{eq:EstPhigSecInterval}
\end{align}

Applying Lemma~\ref{Lem-Stri-cUcT} to the first, fourth, and fifth term on the above right-hand side, we obtain
\begin{align}
	&\| \bdyop(v_1,z)\|_{S^1(t_n, t_{n+1})}  + \|\cI_{\cdot;t_n}(v_1 z)_R\|_{S^1(t_n, t_{n+1})} + \|\cI_{\cdot; t_n} \bdyop(v_1, v_1 z)\|_{S^1(t_n, t_{n+1})} \nonumber \\
	 &\leq C_0 (K^{-1} + (t_{n+1} - t_n)^{\frac{1}{8}}) \|v_1\|_{\Y(t_n, t_{n+1})} \|z\|_{S^1(t_n, t_{n+1})}  + C_0 (t_{n+1} - t_{n})^{\frac{1}{4}} K \log_2 K \|v_1\|_{\Y(t_n, t_{n+1})} \|z\|_{S^1(t_n, t_{n+1})} \nonumber\\
	&\qquad + C_0 (t_{n+1} - t_{n})^{\frac{1}{8}} \|v_1\|_{\Y(t_n, t_{n+1})}^2 \|z\|_{S^1(t_n, t_{n+1})} \nonumber\\
	&\leq \frac{1}{4} \|z\|_{S^1(t_n, t_{n+1})} + (t_{n+1} - t_n)^{\frac{1}{8}}(C_0 \|Y_0\|_{L^2} + C_0 \|Y_0\|_{L^2}^2 + C_0 \|Y_0\|_{L^2} T^{\frac{1}{8}} K \log_2 K )\|z\|_{S^1(t_n, t_{n+1})} \nonumber\\
	&\leq \frac{1}{4} \|z\|_{S^1(t_n, t_{n+1})} + c(T,\|Y_0\|_{L^2}) (t_{n+1} - t_n)^{\frac{1}{8}} \|z\|_{S^1(t_n, t_{n+1})}
		\leq \frac{1}{2} \|z\|_{S^1(t_n, t_{n+1})}  \label{eq:EstPhig1}
\end{align}
for all $z \in S^1([t_n, t_{n+}])$ by the definition of $K$ and $t_{n+1}$, where we also used~\eqref{eq:DefcTY0} and~\eqref{eq:EstKlogK}. For the second and third term on the right-hand side of~\eqref{eq:EstPhigSecInterval}, we combine Lemma~\ref{lem:LinearEstimates}~\ref{it:LinStrichartz} with Lemma~\ref{Lem-Stri-cUcT} to infer
\begin{align}
	\label{eq:EstPhig2}
	&\|\cU_{\cdot;t_n} \cI_{t_n}(v_1 z)_R\|_{S^1(t_n, t_{n+1})} + \|\cU_{\cdot;t_n} \cI_{t_n} \bdyop(v_1, v_1 z)\|_{S^1(t_n, t_{n+1})}  \nonumber \\
	&\leq C_0 t_n^{\frac{1}{4}} K \log_2 K \|v_1\|_{\Y(0,t_n)} \|z\|_{S^1(0,t_n)} + C_0 t_n^{\frac{1}{8}} \|v_1\|_{\Y(0,t_n)}^2 \|z\|_{S^1(0,t_n)} \nonumber \\
	&\leq  (C_0 \|Y_0\|_{L^2} T^{\frac{1}{8}} K \log_2 K + C_0 \|Y_0\|_{L^2}^2) T^{\frac{1}{8}} \|z\|_{S^1(0,t_n)} 
	\leq c(T,\|Y_0\|_{L^2}) T^{\frac{1}{8}} \|z\|_{S^1(0,t_n)}
\end{align}
for all $z \in S^1([t_n, t_{n+}])$, where we employed~\eqref{eq:DefcTY0} and~\eqref{eq:EstKlogK} again. Combining~\eqref{eq:EstSplittingPhig} to~\eqref{eq:EstPhig2}, we arrive at
\begin{align}
\label{eq:EstPhigFinal}
	\|\Phi(z;g)\|_{S^1(0,t_{n+1})} &\leq \|z\|_{S^1(0,t_n)} + \frac{1}{2} \|z\|_{S^1(t_n, t_{n+1})}+ c(T,\|Y_0\|_{L^2}) T^{\frac{1}{8}} \|z\|_{S^1(0,t_n)} + \|g\|_{S^1(t_n, t_{n+1})}\nonumber \\
						&= (1 + c(T,\|Y_0\|_{L^2}) T^{\frac{1}{8}}) \|z\|_{S^1(0,t_n)} + \|g\|_{S^1(t_n, t_{n+1})} + \frac{1}{2} \|z\|_{S^1(t_n, t_{n+1})}
\end{align}
for all $z \in S^1([0,t_{n+1}])$ with $z = \Phi(z;g)$ on $[0,t_n]$.

In particular, for $z \in B_{n+1}$, where by definition we have $z_n = z = \Phi(z;g)$ on $[0,t_n]$, we infer
\begin{align*}
	\|\Phi(z;g)\|_{S^1(0,t_{n+1})} &\leq (1 + c(T,\|Y_0\|_{L^2}) T^{\frac{1}{8}}) R_n + \|g\|_{S^1(t_n, t_{n+1})} + \frac{1}{2} R_{n+1} \\
	&\leq 4^{n+1} (1 + c(T,\|Y_0\|_{L^2})T^{\frac{1}{8}})^{n+1} \|g\|_{S^1(0,T)} + \|g\|_{S^1(0,T)} + \frac{1}{2} R_{n+1} \leq R_{n+1},
\end{align*}
where we also used that $\|z_n\|_{S^1(0,t_n)} \leq R_n$ by assumption. Since 
\begin{align*}
	\Phi(z; g) - \Phi(\tilde{z};g) = \Phi(z - \tilde{z};0)
\end{align*}
and $\Phi(0; 0) = 0$, estimate~\eqref{eq:EstPhigFinal} also yields
\begin{align}
	\label{eq:EstPhizgContraction}
	\|\Phi(z;g) - \Phi(\tilde{z};g)\|_{S^1(0,t_{n+1})} &= \|\Phi(z - \tilde{z};0)\|_{S^1(0,t_{n+1})} \leq \frac{1}{2} \|z - \tilde{z} \|_{S^1(t_n,t_{n+1})} = \frac{1}{2}  \|z - \tilde{z} \|_{S^1(0,t_{n+1})}
\end{align}
for all $z, \tilde{z} \in S^1([t_n, t_{n+1}])$ with $z = \tilde{z}$ on $[0,t_n]$.

 Consequently, $\Phi(\cdot;g)$ is a contractive self-mapping on the complete metric space $B_{n+1}$ and thus has a unique fixed point $z_{n+1}$ in $B_{n+1}$. By~\eqref{eq:EstPhizgContraction}, this fixed point is unique among the elements of $S^1([t_n, t_{n+1}])$ which satisfy $z = z_n$ on $[0,t_n]$. By the uniqueness property of $z_n$, the fixed point $z_{n+1}$ is thus unique in $S^1([0,t_{n+1}])$. 

Proceeding inductively now, we finally obtain the fixed point $z_{L+1}$, which is the unique solution of~\eqref{eq:SchrTypeNonlocPot} in $S^1([0,T])$ by construction. Moreover, we have $z_{L+1} \in B_{L+1}$ and hence
\begin{align*}
	\|z_{L+1}\|_{S^1(0,T)} &\leq R_{L+1} = 4^{L+2} (1 + c(T, \|Y_0\|_{L^2}) T^{\frac{1}{8}})^{L+1} \|g\|_{S^1(0,T)} \leq C(T, \|Y_0\|_{L^2})  \|g\|_{S^1(0,T)},
\end{align*}
where  $C(T,\|Y_0\|_{L^2})$ is increasing in both its arguments by~\eqref{L-partition-bdd}.
\end{proof}

We are now ready to prove Theorem~\ref{Thm-Noise-Reg}.

{\bf Proof of Theorem \ref{Thm-Noise-Reg}.}
Fix $T \in (0,\infty)$ and set $v_1(t) = e^{\imu t |\nabla|} Y_0$. We fix $K$ as in the proof of Lemma~\ref{Lem-Est-v1Pot}. Lemma~\ref{Lem-Est-v1Pot} then shows that
\begin{equation}
	\label{eq:DefcL}
	\cL \colon S^1(0,T) \rightarrow S^1(0,T), \quad \cL g = z,
\end{equation}
where $z$ denotes the unique solution of~\eqref{eq:SchrTypeNonlocPot} with right-hand side $g$, is well-defined. Since~\eqref{eq:SchrTypeNonlocPot} is linear in $z$, the operator $\cL$ is moreover linear.

We next define
\begin{equation}
	\label{v2z-def}
	 v_2(z) = \imu \int_0^{\cdot} e^{\imu (\cdot -s) |\nabla|} (h|\na||z|^2)(s) \dd s,
\end{equation}
for every $z \in S^1(0,T)$ with $h$ defined in~\eqref{h-W1-def}. Note that $(z, v_2)$ is a solution of system~\eqref{equa-z-Omegav2z-mild} if and only if $z$ satisfies
\begin{align}
	\hspace{-0.5 em} z(t) + \bdyop(v_1, z)(t) + \imu \cI_t(v_1 z)_R + \imu \cI_t \bdyop(v_1, v_1 z) &= \cU_t f - \bdyop(v_2(z), z)(t) - \imu \cI_t(v_2(z) z)_R + \imu \cI_t \bdyop(h |\nabla| |z|^2, z) \nonumber\\
	&\quad  - \imu \cI_t \bdyop(v_2(z), (v_1 + v_2(z)) z) - \imu \cI_t \bdyop(v_1, v_2(z) z), \label{eq:RegNoiseFixedPointz}
\end{align}
with $f = X_0 + \bdyop(Y_0, X_0)$ and $v_2 = v_2(z)$. It is thus sufficient to solve~\eqref{eq:RegNoiseFixedPointz}, the right-hand side of which we denote by $\Psi(z)$, i.e.,
\begin{align}
	\Psi(z) &=  \cU f - \bdyop(v_2(z), z) - \imu \cI(v_2(z) z)_{R}
        + \imu \cI \bdyop(h |\na| |z|^2, z)   - \imu \cI \bdyop(v_2(z), (v_1 + v_2(z)) z) - \imu \cI\bdyop(v_1, v_2(z) z) \notag \\
    &=:  \cU u_0 + \sum\limits_{j=1}^5 \Psi_{j}(z).
\end{align}
Recalling the definition of the operator $\cL$, we see that~\eqref{eq:RegNoiseFixedPointz} is equivalent to the fixed point equation
\begin{equation}
	\label{eq:RegNoiseFixedPointzOperators}
	z = \cL \Psi(z) =: \Phi(z).
\end{equation}

We set $M := 2 C_1 \|X_0\|_{H^1}$, where $C_1 = C_1(T,\|Y_0\|_{L^2})(\geq 1)$ is the constant appearing in~\eqref{eq:EstPhiSelfMap} below.
We will prove that $\Phi$ is a contraction on the ball
\begin{align*}
	B_{S^1(0,\tau)}(M):=\{z\in S^1(0,\tau): \|z\|_{S^1(0,\tau)} \leq  M\}
\end{align*}
for some stopping time $\tau$,
which satisfies $\tau=T$ with high probability if $\phi_1$ is large enough. 
The proof consists of the following three steps.

{\bf Step $1$:} We first show that $\Phi(B_{S^1(0,\tau)}(M)) \subseteq B_{S^1(0,\tau)}(M)$.
For this purpose, we note that
\begin{align}  \label{v2-YT-esti}
   \|v_2(z)\|_{\Y(0,\tau)} \leq C   \|h\|_{L^4(0,\tau)} \|z\|_{S^1(0,\tau)}^2,
\end{align}
by~\eqref{eq:EstQuadraticWave-S1} and that 
\begin{align}  \label{cUX0-esti}
   \|\cU f\|_{S^1(0,\tau)} \leq \|\cU X_0\|_{S^1(0,\tau)} + \|\cU \bdyop(Y_0, X_0)\|_{S^1(0,\tau)} \leq  C \|X_0\|_{H^1} + C  \|Y_0\|_{L^2} \|X_0\|_{H^1}
\end{align}
by Lemma~\ref{lem:LinearEstimates}~\ref{it:LinStrichartz} and Lemma~\ref{lem:Bdy}.
We next estimate each $\Psi_j(z)$ for every $1 \leq j \leq 5$.

By the estimate for the boundary term in~\eqref{eq:EstBdyTime-S1} and~\eqref{v2-YT-esti}, we infer
\begin{align}  \label{Psi1z-esti}
   \|\Psi_1(z)\|_{S^1(0,\tau)}
   &\lesssim (K^{-1} + \tau^\frac 18) \|v_2(z)\|_{\Y(0,\tau)} \|z\|_{S^1(0,\tau)} 
   \lesssim (K^{-1} + \tau^\frac 18) \|h\|_{L^4(0,\tau)}  \|z\|_{S^1(0,\tau)}^3.
\end{align}
Using \eqref{eq:EstQuadraticR-S1},
we also get
\begin{align} \label{Psi2z-esti}
   \| \Psi_2(z)\|_{S^1(0,\tau)}
   \lesssim \tau^\frac 14 K \log_2K \|v_2(z)\|_{\Y(0,\tau)} \|z\|_{S^1(0,\tau)}
   \lesssim \tau^\frac 14 K \log_2K \|h\|_{L^4(0,\tau)} \|z\|_{S^1(0,\tau)}^3,
\end{align}
while~\eqref{eq:EstCubicuuu-S1} gives
\begin{align}  \label{Psi3z-esti}
   \|\Psi_3(z)\|_{S^1(0,\tau)}
   \lesssim  \|h\|_{L^4(0,\tau)} \|z\|_{S^1(0,\tau)}^3.
\end{align}
Moreover, combining~\eqref{eq:EstCubicvvu-S1} with~\eqref{v2-YT-esti} and using $\|v_1\|_{\Y(0,\tau)} = \|Y_0\|_{L^2} $, we derive that
\begin{align}  \label{Psi4z-esti}
   \| \Psi_4(z) \|_{S^1(0,\tau)}
   &\lesssim \tau^\frac 18 \|v_2(z)\|_{\Y(0,\tau)} \|v_1+v_2(z)\|_{\Y(0,\tau)} \|z\|_{S^1(0,\tau)}   \notag \\
   &\lesssim \tau^\frac 18  (\|Y_0\|_{L^2} + \|h\|_{L^4(0,\tau)} \|z\|_{S^1(0,\tau)}^2) \|h\|_{L^4(0,\tau)} \|z\|_{S^1(0,\tau)}^3.
\end{align}
For the last term  $\Psi_5(z)$ we obtain analogously
\begin{align}
   \| \Psi_5(z) \|_{S^1(0,\tau)}
   &\lesssim \tau^\frac 18 \|v_1\|_{\Y(0,\tau)} \|v_2(z)\|_{\Y(0,\tau)} \|z\|_{S^1(0,\tau)}   
   \lesssim \tau^\frac 18 \|Y_0\|_{L^2} \|h\|_{L^4(0,\tau)} \|z\|_{S^1(0,\tau)}^3. \label{Psi5z-esti}
\end{align}
Applying Lemma~\ref{Lem-Est-v1Pot}, combining estimates~\eqref{cUX0-esti} to~\eqref{Psi5z-esti}, recalling the bound~\eqref{eq:EstKlogK} for our choice of $K$,
and using $\|z\|_{S^1(0,\tau)} \leq M$,
we thus conclude
\begin{align}
   \|\Phi(z)\|_{S^1(0,\tau)} &\leq C(\tau, \|Y_0\|_{L^2}) \|\Psi(z)\|_{S^1(0,\tau)} \leq C_1 \|X_0\|_{H^1}
         + C_1 (1+\|h\|_{L^4(0,\tau)}\|z\|_{S^1(0,\tau)}^2) \|h\|_{L^4(0,\tau)}\|z\|_{S^1(0,\tau)}^3 \notag \\
   &\leq C_1 \|X_0\|_{H^1}
         + C_1 M^2 \|h\|_{L^4(0,\tau)} \|z\|_{S^1(0,\tau)}   + C_1 M^4   \|h\|^2_{L^4(0,\tau)} \|z\|_{S^1(0,\tau)}, \label{eq:EstPhiSelfMap}
\end{align}
where $C = C(\cdot,\|Y_0\|_{L^2})$ is the constant from Lemma~\ref{Lem-Est-v1Pot} and $C_1 = C_1(T,\|Y_0\|_{L^2}) \geq 1$ is a constant which is increasing in both its arguments.
Taking the stopping time
\begin{align} \label{tau-def1}
   \tau:=& \inf\bigg\{t\in [0,T]:    C_1(T,\|Y_0\|_{L^2}) M^2
      \|h\|_{L^4(0,t)}> \frac 14 \bigg\} \wedge T,
\end{align}
we then obtain
\begin{align}
   \|\Phi(z)\|_{S^1(0,\tau)}
    \leq   C_1 \|X_0\|_{H^1} + \frac 12 \|z\|_{S^1(0,\tau)}
    \leq \frac{M}{2} + \frac{M}{2} = M.
\end{align}
Consequently, $\Phi$ maps $B_{S^1(0,\tau)}(M)$ into itself.
\medskip

{\bf Step $2$:}
Next we show that $\Phi\colon B_{S^1(0,\tau)}(M) \to B_{S^1(0,\tau)}(M)$ is contractive.
The proof is again
based on the multilinear estimates.
The implicit constants in this step may depend on $T$, $\|X_0\|_{H^1}$ and $\|Y_0\|_{L^2}$.

We take any $\wt z \in B_{S^1(0,\tau)}(M)$ and recall that $v_2(\wt z)$ was defined in~\eqref{v2z-def}. As in \eqref{v2-YT-esti} we have
\begin{align}  \label{wtv2-YT-esti}
   \| v_2(\wt z)\|_{\Y(0,\tau)} \lesssim  \|h\|_{L^4(0,\tau)} \|\wt z\|_{S^1(0,\tau)}^2.
\end{align}
Moreover, \eqref{eq:EstQuadraticWave-S1} and the definition of $B_{S^1(0,\tau)}(M)$ imply
\begin{align} \label{v2z-wtv2z-diff}
   \|v_2(z) - v_2(\wt z)\|_{\Y(0,\tau)}
   &\lesssim \|h\|_{L^4(0,\tau)} (\|z\|_{S^1(0,\tau)} +\|\wt z\|_{S^1(0,\tau)}) \|z-\wt z\|_{S^1(0,\tau)} \lesssim \|h\|_{L^4(0,\tau)}  \|z - \wt z\|_{S^1(0,\tau)}.
\end{align}

In order to estimate the difference between $\Psi(z)$ and $\Psi(\wt z)$ in $S^1(0,\tau)$,
we first note that
by \eqref{eq:EstBdyTime-S1}, \eqref{wtv2-YT-esti} and \eqref{v2z-wtv2z-diff}, we infer
\begin{align}  \label{Psi1z-Pai1wtz}
   \|\Psi_1(z) - \Psi_1(\wt z)\|_{S^1(0,\tau)}
   &= \|\bdyop(v_2(z) -v_2(\wt z), z) + \bdyop(v_2(\wt z), z-\wt z)\|_{S^1(0,\tau)} \notag \\
   &\lesssim (K^{-1}+ \tau^\frac 18) (\|v_2(z)- v_2(\wt z)\|_{\Y(0,\tau)}\|z\|_{S^1(0,\tau)} + \| v_2(\wt z)\|_{\Y(0,\tau)}\|z-\wt z\|_{S^1(0,\tau)}) \notag \\
   &\lesssim \|h\|_{L^4(0,\tau)}\|z-\wt z\|_{S^1(0,\tau)}.
\end{align}
Similarly, using  \eqref{eq:EstQuadraticR-S1} we get
\begin{align} \label{Psi2z-Pai2wtz}
   \|\Psi_2(z) - \Psi_2(\wt z)\|_{S^1(0,\tau)}
   &= \| \cI((v_2(z) - v_2(\wt z)) z)_R + \cI( v_2(\wt z) (z-\wt z))_R\|_{S^1(0,\tau)} \notag \\
   &\lesssim \tau^\frac 14 K \log_2K   (\|v_2(z) - v_2(\wt z)\|_{\Y(0,\tau)}\|z\|_{S^1(0,\tau)} + \|v_2(\wt z)\|_{\Y(0,\tau)}\|z-\wt z\|_{S^1(0,\tau)}) \notag \\
   &\lesssim \|h\|_{L^4(0,\tau)}\|z-\wt z\|_{S^1(0,\tau)},
\end{align}
where we again used the bound~\eqref{eq:EstKlogK} for $K$.
Applying~\eqref{eq:EstCubicuuu-S1}, we also derive
\begin{align}  \label{Psi3z-Pai3wtz}
  \|\Psi_3(z) - \Psi_3(\wt z)\|_{S^1(0,\tau)}
  &\lesssim \|h\|_{L^4(0,\tau)} (\|z\|_{S^1(0,\tau)}^2+ \|\wt z\|_{S^1(0,\tau)}^2) \|z-\wt z\|_{S^1(0,\tau)}  \notag \\
  &\lesssim \|h\|_{L^4(0,\tau)}\|z-\wt z\|_{S^1(0,\tau)}.
\end{align}
Regarding the difference $\Psi_4(z)-\Psi_4(\wt z)$,
we use~\eqref{eq:EstCubicvvu-S1}
and that $\|h\|_{L^4(0,\tau)}\lesssim 1$
to infer
\begin{align}  \label{Psi4z-Pai4wtz}
   \|\Psi_4(z) - \Psi_4(\wt z)\|_{S^1(0,\tau)}
   &= \|\cI \bdyop(v_2(z)- v_2(\wt z), (v_1+v_2(z))z)
         + \cI \bdyop( v_2(\wt z), (v_2(z) - v_2(\wt z))z)  \notag \\
    &\qquad  \qquad  + \cI \bdyop(v_2(\wt z), (v_1+ v_2(\wt z))(z-\wt z))\|_{S^1(0,\tau)}  \notag \\
   &\lesssim \tau^\frac 18 (\|v_2(z)- v_2(\wt z)\|_{\Y(0,\tau)} \|v_1+v_2(z)\|_{\Y(0,\tau)}\|z\|_{S^1(0,\tau)} \notag \\
           &\qquad  \qquad  + \| v_2(\wt z)\|_{\Y(0,\tau)} \|v_2(z) - v_2(\wt z)\|_{\Y(0,\tau)}\|z\|_{S^1(0,\tau)}  \notag \\
     &\qquad   \qquad      + \|v_2(\wt z)\|_{\Y(0,\tau)} \|v_1+  v_2(\wt z)\|_{\Y(0,\tau)}\|z-\wt z\|_{S^1(0,\tau)} )  \notag \\
   &\lesssim \|h\|_{L^4(0,\tau)}\|z-\wt z\|_{S^1(0,\tau)}.
\end{align}
Proceeding analogously, we derive
\begin{align}  \label{Psi5z-Pai5wtz}
     \|\Psi_5(z) - \Psi_5(\wt z)\|_{S^1(0,\tau)}
     &\lesssim \|\cI\bdyop(v_1, (v_2(z) - v_2(\wt z)) z) +\cI\bdyop(v_1, v_2(\wt z)(z - \wt z))\|_{S^1(0,\tau)} \notag \\
     &\lesssim \tau^\frac 18 \|v_1\|_{\Y(0,\tau)}
               (\|v_2(z)-v_2(\wt z)\|_{\Y(0,\tau)}\|z\|_{S^1(0,\tau)} + \|v_2(\wt z)\|_{\Y(0,\tau)}\|z-\wt z\|_{S^1(0,\tau)}) \notag \\
     &\lesssim \|h\|_{L^4(0,\tau)}\|z-\wt z\|_{S^1(0,\tau)}.
\end{align}

Since the operator $\cL$ is linear, Lemma~\ref{Lem-Est-v1Pot} and estimates \eqref{Psi1z-Pai1wtz} to \eqref{Psi5z-Pai5wtz} yield
\begin{align}
    \|\Phi(z) - \Phi(\wt z)\|_{S^1(0,\tau)} &= \|\cL \Psi(z) - \cL \Psi(\wt z)\|_{S^1(0,\tau)} \leq C(\tau, \|Y_0\|_{L^2})  \|\Psi(z) - \Psi(\wt z)\|_{S^1(0,\tau)}
      \notag \\
    &\leq C(T, \|Y_0\|_{L^2}) \sum\limits_{j=1}^5  \|\Psi_j(z) - \Psi_j(\wt z)\|_{S^1(0,\tau)} \leq  C_2 \|h\|_{L^4(0,\tau)}\|z-\wt z\|_{S^1(0,\tau)},
\end{align}
where $C = C(\cdot, \|Y_0\|_{L^2})$ is the constant from Lemma~\ref{Lem-Est-v1Pot} and $C_2 = C_2(T, \|X_0\|_{H^1}, \|Y_0\|_{L^2}) \geq 1$.
We thus define the stopping time
\begin{align} \label{tau-def2}
   \wt \tau:=\inf\left\{t\in [0,T]: C_2(T, \|X_0\|_{H^1}, \|Y_0\|_{L^2}) \|h\|_{L^4(0,t)} \geq \frac 12 \right\} \wedge \tau,
\end{align}
which yields
\begin{align}
   \|\Phi(z) - \Phi(\wt z)\|_{S^1(0, \wt \tau)} &\leq \frac 12 \|z-\wt z\|_{S^1(0,\wt \tau)}.
\end{align}

Consequently, the above two steps show that
$\Phi \colon B_{S^1(0,\wt \tau)}(M) \to B_{S^1(0,\wt \tau)}(M) $ is a contractive self-map.
The Banach fixed point theorem thus gives a unique $z\in B_{S^1(0,\wt \tau)}(M)$ such that
$z=\Phi(z)$.
In view of~\eqref{v2z-def},
we obtain that $(z,v_2(z))$ solves \eqref{equa-z-Omegav2z-mild} on $[0, \wt \tau]$.
The uniqueness of this solution can be proved using standard arguments and the estimates from Step 2.
\medskip

{\bf Step $3$:} In this step we prove that
\begin{align}  \label{tau-T-highprob}
   \bbp (\wt \tau = T) \longrightarrow 1\ \ \text{as} \ \phi_1 \to \infty.
\end{align}
In view of \eqref{tau-def1} and \eqref{tau-def2},
it suffices to prove that
\begin{align}  \label{h-L4-highprob}
    \bbp \(\|h\|_{L^4(0,\infty)}^4 \geq C_3^{-4}(T,\|X_0\|_{H^1}, \|Y_0\|_{L^2})  \)
    \longrightarrow 0,\ \ \text{as}\ \phi_1 \to \infty,
\end{align}
where $C_3(T,\|X_0\|_{H^1}, \|Y_0\|_{L^2})$ is defined by $C_3 = \max\{4 C_1 M^2, 2 C_2\}$ with $C_1 = C_1(T,\|X_0\|_{H^1}, \|Y_0\|_{L^2})$ and $C_2 = C_2(T,\|X_0\|_{H^1}, \|Y_0\|_{L^2})$ from
the construction of $\tau$ and $\wt \tau$ in \eqref{tau-def1} and \eqref{tau-def2}, respectively.

In order to prove \eqref{h-L4-highprob},
we note that \eqref{h-W1-def} yields the identity
\begin{align*}
   \|h\|_{L^4(0,\infty)}^4
   = \int_0^\infty e^{-8\Im \phi_1^{(1)} \beta^{(1)}_1(t) - 8 (\Im \phi_1^{(1)})^2 t} \dd t.
\end{align*}
By virtue of the scaling property of Brownian motion,
i.e.,
$\bbp \circ ( \Im \phi_1^{(1)} \beta^{(1)}_1(\cdot))^{-1}
= \bbp \circ (\beta^{(1)}_1(( \Im \phi_1^{(1)} )^2\cdot))^{-1}$,
we then derive
\begin{align}  \label{PhL4-esti}
   & \bbp (\|h\|_{L^4(0,\infty)}^{4} \geq  C_3(T,\|X_0\|_{H^1}, \|Y_0\|_{L^2})^{-4}   ) \notag \\
   &= \bbp\( \int_0^\infty e^{-8 \beta^{(1)}_1((\Im \phi_1^{(1)} )^2 t) - 8 (\Im \phi_1^{(1)})^2 t} \dd t
            \geq C_3(T,\|X_0\|_{H^1}, \|Y_0\|_{L^2})^{-4}   \) \notag \\
   &= \bbp\( \int_0^\infty e^{-8 \beta^{(1)}_1(t) - 8 t}  \dd t
               \geq (\Im \phi_1^{(1)})^2 C_3(T,\|X_0\|_{H^1}, \|Y_0\|_{L^2})^{-4}  \),
\end{align}
where we also used a change of variables in the last step.

By the law of iterated logarithm for Brownian motion~\eqref{log-BM},
we infer that there exists $t_0=t_0(\omega)$ large enough such that
$- \beta^{(1)}_1(t) - t \leq -\frac 12 t$ for $t \geq t_0$, implying that
 $\int_0^\infty e^{-8 \beta^{(1)}_1(t) - 8 t} dt <\infty $, $\bbp$-a.s.
Since $ C_3(T,\|X_0\|_{H^1}, \|Y_0\|_{L^2})$ is a fixed deterministic constant,
independent of $\phi_1^{(1)}$,
we obtain that the right-hand side of \eqref{PhL4-esti} converges to zero
as $\phi_1^{(1)} \to \infty$, showing~\eqref{h-L4-highprob} and completing the proof.
\hfill $\square$

\section*{Appendix}

\appendix

In this appendix we show that the different notions of solution for the Zakharov system which we use in this paper are all equivalent. More precisely, we prove that the analytically weak, the mild, and the normal form solutions all coincide.

\section{Equivalence between weak and mild solutions}  \label{App-Weak-Mild}

We begin with the equivalence between the weak and the mild solutions.
The proof presented below follows the strategy by Lions and Masmoudi~\cite{LM01}
in the context of the Navier-Stokes equations.

	 \begin{proposition}  \label{Prop-Equiv-Weak-Mild}
	 Let $(u,v)\in C([0,\tau]; H^1) \times C([0,\tau]; L^2)$. If $(u,v)$ is a mild solution of~\eqref{equa-ranZak-bc}, i.e. it satisfies~\eqref{eq:RZakharovDuhamel}, then $(u,v)$ is also an analytically weak solution of \eqref{equa-ranZak-bc}
	 in $H^{-1} \times H^{-1}$,
	 and vice versa.
	 \end{proposition}
	
	 \begin{proof}
	 We only prove the statement for the Schr\"odinger component $u$,
	 as the case of the wave component $v$ can be proved analogously.
	
	 (I) We first assume that $u$ is a mild solution of~\eqref{equa-ranZak-bc}, i.e. $u$ satisfies the Schr\"odinger equation in \eqref{eq:RZakharovDuhamel}.
	 We write $u_\ve:= u \ast \phi_\ve$
	 for the standard spatial mollification and use the same notation for the other appearing functions.
	 Since $(u,v)\in C([0,\tau]; H^1) \times C([0,\tau]; L^2)$,
	 we deduce that for any $1\leq q<\infty$,
	 \begin{align}  \label{uve-u-vve-v}
		 u_\ve \longrightarrow u \ \ {\rm in}\ L^q(0,\tau; H^1), \ \
		v_\ve \longrightarrow v \ \ {\rm in}\ L^q(0,\tau; L^2),  \ \
		{\rm as}\ \ve \to 0.
	 \end{align}
As $b \in L^\infty(0,\tau;H^1)$ and $c, \cT_{\cdot}(W_2) \in L^\infty(0,\tau;L^2)$, we obtain analogous convergence results for $b_\epsilon, c_\epsilon$, and $(\cT_{\cdot}(W_2))_\epsilon$. Taking into account the Sobolev embedding $H^1\hookrightarrow L^4$,
we thus infer
	 \begin{align} \label{uve-vve-asmp-H-1}
		  v_\ve u_\ve  \longrightarrow v u,\ \
		  b_\ve \cdot \na u_\ve \longrightarrow b\cdot \na v, \ \
		  c_\ve u_\ve \longrightarrow cu, \ \
         (\mathcal{T}_\cdot(W_2))_\ve u_\ve \longrightarrow \mathcal{T}_\cdot(W_2) u\ \  {\rm in}\ L^q(0,\tau; H^{-1}).
	 \end{align}
	
	 Set
	 \begin{align}
		g_\ve(t)
		:= e^{\imu t\Delta} u_{0,\ve}
		  - \imu \int_0^t e^{\imu (t-s)\Delta} \big(v_\ve u_\ve - b_\ve \cdot \na u_\ve - c_\ve u_\ve + (\mathcal{T}_s(W_2))_\ve u_\ve \big) \dd s
	 \end{align}
	 for all $t\in [0,\tau]$.
	 Then, $g_\ve$ satisfies the equation
	 \begin{align}   \label{equa-hve}
		\imu \partial_t g_\ve + \Delta g_\ve  = v_\ve u_\ve - b_\ve \cdot \na u_\ve - c_\ve u_\ve + (\mathcal{T}_s(W_2))_\ve u_\ve
	 \end{align}
     with $g_\ve(0) = u_{0,\ve}$
	 and~\eqref{uve-vve-asmp-H-1} implies that
	 for any $q\in [1,\infty)$,
	 \begin{align}   \label{hve-u}
		g_\ve \longrightarrow u\  \ {\rm in}\ L^q(0,\tau; H^{-1}), \ \ {\rm as}\ \ve \to 0.
	 \end{align}
	
	Testing~\eqref{equa-hve} with any Schwartz function $\vf \in \mathcal{S}$ and integrating in time, we obtain
	 \begin{align*}
		\<g_\ve(t), \vf\>
		= \<u_{0,\ve}, \vf\>
		  + \int_0^t \<g_\ve, -\imu \Delta \vf\> \dd s
		   + \int_0^t \<- \imu v_\ve u_\ve + \imu b_\ve\cdot \na u_\ve + \imu c_\ve u_\ve - \imu (\mathcal{T}_s(W_2))_\ve u_\ve  , \vf\> \dd s.
	 \end{align*}
	 Using \eqref{uve-vve-asmp-H-1} and \eqref{hve-u} to pass to the limit $\ve\to 0$,
	 we get
	 for $\dd t$-a.e. $t\in[0,\tau]$
	 \begin{align*}
		 {}_{H^{-1}}\<u(t), \vf\>_{H^1}
		&= {}_{H^{-1}}\<u_0, \vf\>_{H^1}
		  + \int_0^t {}_{H^{-1}}\<u, -\imu \Delta \vf\>_{H^1} \dd s  \\
		 &\qquad + \int_0^t {}_{H^{-1}}\<  -  \imu v u  + \imu b\cdot \na u + \imu c u - \imu \mathcal{T}_s(W_2) u, \vf\>_{H^1} \dd s.
	 \end{align*}
	 Taking into account the continuity of both sides in $t$,
	 $ {}_{H^{-1}}\<u, -\imu \Delta \vf\>_{H^1} =  {}_{H^{-1}}\<\imu \Delta u, \vf\>_{H^1}$,
	 and the density of $\mathcal{S}$ in $H^{1}$,
	 we conclude that
	 $u$ is an analytically weak solution of the Schr\"odinger equation in \eqref{equa-ranZak-bc}
     in $H^{-1}$.
	
	 (II)
	 Now we assume that $u\in C([0,\tau]; H^1)$ is an analytically weak solution of
the Schr\"odinger equation in~\eqref{equa-ranZak-bc}.
	 Using this property and mollifying in space again,
	 we get
	 \begin{align}
     \partial_t u_\ve = \imu \Delta u_\ve - \imu \big(v u - b\cdot \na u - c u + \mathcal{T}_{\cdot}(W_2) u  \big)_\ve
	 \end{align}
with $u_\ve(0) = u_{0,\ve}$.
	Consequently,
	 for any $\phi \in C^\infty([0,\tau)\times \mathbb{R}^3)$ we get
	 \begin{align}  \label{uve-rvephi}
		\int_0^\tau \< \partial_t u_\ve - \imu \Delta u_\ve
          + \imu (v_\ve u_\ve - b_\ve\cdot \na u_\ve - c_\ve u_\ve + (\mathcal{T}_t(W_2))_\ve u_\ve  , \phi\> \dd t
		= r_\ve (\phi),
	 \end{align}
	 where
	 \begin{align*}
		r_\ve (\phi)
		&:=  -\int_0^\tau
			\big< \imu ( (vu)_\ve - v_\ve u_\ve) - \imu ((b\cdot \na u)_\ve - b_\ve \cdot \na u_\ve)
		   - \imu ((cu)_\ve - c_\ve u_\ve)  \\
          &\qquad \qquad + \imu ((\mathcal{T}_s(W_2)  u)_\ve - (\mathcal{T}_s(W_2))_\ve u_\ve ), \phi \big>  \dd t.
	 \end{align*}
	 We infer from~\eqref{uve-vve-asmp-H-1} that
	 \begin{align}  \label{rvephi-0}
		r_\ve (\phi) \longrightarrow 0, \ \ \text{as}\ \ve \to 0.
	 \end{align}
	
	 Next, we set
	 \begin{align} \label{hve-def}
		\wt g_\ve (t) :=
		u_\ve (t) - e^{\imu t \Delta} u_{0,\ve}
		 + \int_0^t e^{\imu (t-s)\Delta} ( \imu v_\ve u_\ve - \imu b_\ve \cdot \na u_\ve - \imu c_\ve u_\ve
                       + \imu (\mathcal{T}_s(W_2))_\ve u_\ve  ) \dd s
	 \end{align}
	 for all $t\in [0,\tau]$. Then $\wt  g_\ve$ satisfies the equation
	 \begin{align}   \label{equa-hve*}
		(\partial_t - \imu \Delta) \wt g_\ve
		&= (\partial_t - \imu \Delta) u_\ve
			+ \imu v_\ve u_\ve - \imu b_\ve \cdot \na u_\ve - \imu c_\ve u_\ve
            + \imu (\mathcal{T}_{\cdot}(W_2))_\ve u_\ve
	 \end{align}
with $\wt g_\ve (0) = 0$.
	Combining~\eqref{uve-rvephi} and~\eqref{equa-hve*}, we thus obtain
	 \begin{align}
		\int_0^\tau  \< \partial_t \wt g_\ve - \imu \Delta \wt g_\ve, \phi\> \dd t
		&= \int_0^\tau \<\partial_t u_\ve - \imu \Delta u_\ve  + \imu (v_\ve u_\ve - b_\ve\cdot \na u_\ve
              - c_\ve u_\ve  + (\mathcal{T}_t(W_2))_\ve u_\ve ), \phi\> \dd t \notag \\
		&= r_\ve (\phi),
	 \end{align}
	 so that an integration by parts gives
	 \begin{align}  \label{hve-phi-rve}
		\int_0^\tau \< \wt g_\ve, -\partial_t \phi + \imu \Delta \phi\> \dd t
		= r_\ve(\phi).
	 \end{align}
	
	 For any $F \in  C_c^\infty([0,\tau]\times \bbr^d)$,
	 let $\phi \in C^\infty([0,\tau)\times \bbr^3)$ be the solution to the inhomogeneous Schr\"odinger equation
	 \begin{align*}
	   - \partial_t \phi + \imu \Delta \phi = F \ \ \text{with}\  \phi(\tau) = 0.
	 \end{align*}
	 Hence, we deduce from \eqref{hve-phi-rve} that
	 \begin{align*}
		\int_0^\tau  \< \wt  g_\ve,  F \> \dd t
		= r_\ve(\phi).
	 \end{align*}
	 In view of \eqref{rvephi-0},
	 we infer that the righ-hand side above converges to zero as $\ve \to 0$,
	 and thus $\wt g_\ve$ converges weakly to zero as $\epsilon \to 0$.
	 On the other hand,
	 using~\eqref{uve-u-vve-v} and~\eqref{uve-vve-asmp-H-1} to pass to the limit $\ve\to 0$ in \eqref{hve-def},
	 we also get
	 \begin{align} \label{hve-limit}
		\wt  g_\ve(t)
		\longrightarrow u(t) - e^{\imu t\Delta} u_0
			+ \int_0^t e^{\imu (t-s)\Delta} \big(\imu v u - \imu b \cdot \na u - \imu c u + \imu \mathcal{T}_s(W_2) u \big) \dd s
	 \end{align}
	in $L^q(0,\tau; H^{-1})$ as $\ve\to 0$.
	 Combining \eqref{hve-limit} with the weak convergence of $\wt g_{\epsilon}$ to zero,
	 we conclude that
	 \begin{align}
	   u(t) - e^{\imu t\Delta} u_0
			+ \int_0^t e^{\imu (t-s)\Delta} \big(\imu v u -  \imu b \cdot \na u - \imu  cu + \imu \mathcal{T}_s(W_2) u \big) \dd s
			=0
	 \end{align}
for any $t\in [0,\tau]$,
which shows that $u$ is a mild solution of
the Schr\"odinger equation in~\eqref{eq:RZakharovDuhamel}.
	 \end{proof}

\section{Equivalence between mild and normal form solutions}  \label{App-Mild-Normal}
 	
In this appendix we show that mild solutions of~\eqref{eq:RZakharovDuhamel}
are also solutions of the normal form formulation~\eqref{eq:RZakharovNormalForm} and vice versa.
The precise formulation is contained in Proposition \ref{Prop-Equiv-Mild-Norm} below.
	
	\begin{proposition}  \label{Prop-Equiv-Mild-Norm}
		Let $(u_0, v_0) \in H^1 \times L^2$. Then, the following holds:
		\begin{enumerate}
			\item \label{it:MildNormalForm} Let $(u,v) \in \X(0,\tau) \times \Y(0,\tau)$ be a mild solution of \eqref{eq:RZakharovDuhamel} on $[0,\tau]$ for some $\tau> 0$. Then $(u,v)$ is also a solution of the normal form formulation~\eqref{eq:RZakharovNormalForm} on $[0,\tau]$ for every $K \in 2^\N$ with $K \geq 2^5$.
			\item \label{it:NormalFormMild} There exist $K' \in 2^\N$ and $\tau' > 0$ with the following property: If $K \in 2^\N$ with $K \geq K'$, $\tau \in (0,\tau']$, and $(u,v) \in \X(0,\tau) \times \Y(0,\tau)$ is a solution of~\eqref{eq:RZakharovNormalForm} with parameter $K$ on $[0,\tau]$, then $(u,v)$ is also a mild solution of~\eqref{eq:RZakharovDuhamel} on $[0,\tau]$.
			\item \label{it:NormalFormIndK} Let $K'$ and $\tau'$ be the parameters from $(ii)$. If $(u,v) \in \X(0,\tau) \times \Y(0,\tau)$ satisfies~\eqref{eq:RZakharovNormalForm} with a dyadic number $K_1 \in 2^\N$ with $K_1 \geq K'$ on a time interval $[0,\tau]$ with $\tau \leq \tau'$, then $(u,v)$ also satisfies~\eqref{eq:RZakharovNormalForm} for every $K_2 \in 2^\N$ with $K_2 \geq K'$ on $[0,\tau]$.
		\end{enumerate}
	\end{proposition}
	
	\begin{proof}
		\ref{it:MildNormalForm} Fix $K \in 2^\N$ with $K \geq 2^5$. Let $(u,v)$ be a mild solution of \eqref{eq:RZakharovDuhamel} on $[0,\tau]$,
where the equations are understood in $H^{-1} \times H^{-1}$.
Let $u_{0,\ve}, v_{0,\ve}, u_\epsilon, v_\epsilon, b_\epsilon, c_\epsilon, (\cT_\cdot(W_2))_\ve$ denote the mollification of the corresponding functions as in the proof of Proposition \ref{Prop-Equiv-Weak-Mild}. We note that $u_\epsilon, v_\epsilon, b_\epsilon, c_\epsilon, (\cT_\cdot(W_2))_\ve  \in C([0,\tau], H^k)$ for all $k \in \N$ and $\epsilon > 0$.
Then the convergence results in~\eqref{uve-u-vve-v} and \eqref{uve-vve-asmp-H-1} are still valid.

We now define $U_\epsilon$ and $V_\epsilon$ by
		\begin{align}
			U_\epsilon(t) &:= e^{\imu t \Delta} u_{0,\epsilon}
           - \imu  \int_0^t e^{\imu (t-s) \Delta}(v_\epsilon u_\epsilon
             - b_\epsilon \cdot \nabla u_\epsilon - c_\epsilon u_\epsilon + (\cT_\cdot(W_2))_\ve u_\ve)(s) \dd s, \label{eq:DefUeps}\\
			V_\epsilon(t) &:= e^{\imu t |\nabla|} v_{0,\epsilon} + \imu \int_0^t e^{\imu (t-s) |\nabla|} |\nabla| |u_\epsilon|^2(s) \dd s \label{eq:DefVeps}
		\end{align}
		for all $t \in [0,\tau]$.
Then, by \eqref{uve-u-vve-v} and \eqref{uve-vve-asmp-H-1},
for any $q\in [1,2]$,
		\begin{align}
			U_\epsilon &\longrightarrow e^{\imu (\cdot) \Delta} u_{0}
            - \imu  \int_0^\cdot e^{\imu (\cdot-s) \Delta}(v u - b \cdot \nabla u - c u + \cT_\cdot(W_2)u)(s) \dd s = u \quad \text{in } L^q(0,\tau; H^{-1}),
            \label{Uve-limit}  \\
			V_\epsilon  &\longrightarrow e^{\imu (\cdot) |\nabla|} v_{0}
                + \imu \int_0^\cdot e^{\imu (\cdot-s) |\nabla|} |\nabla| |u|^2(s) \dd s = v \quad \text{in } L^q(0,\tau; L^2).   \label{Vve-limit}
		\end{align}
		as $\ve \to 0$, where we also used that $u \in \langle \nabla \rangle^{-1} L^2(0,\tau;\dot{B}^0_{6,2})$ since $u \in \X(0,\tau)$.
		
		Note that $U_\epsilon, V_\epsilon \in C^1([0,\tau],H^k)$ for all $k \in \N$. Moreover, we have as in~\eqref{pt-wtu-wtv},
\begin{equation}  \label{equa-Uve-Vve}
	\left\{\aligned
	  \partial_t(e^{\imu t |\xi|^2} \hat{U}_\epsilon(t,\xi)) &= -\imu e^{\imu t |\xi|^2}\cF(v_\epsilon u_\epsilon)(t,\xi) +
           \imu e^{\imu t |\xi|^2}\cF(b_\epsilon \cdot \nabla u_\epsilon + c_\epsilon u_\epsilon - (\cT_s(W_2))_\ve u_\ve)(t,\xi), \\
    	\partial_t(e^{-\imu t |\xi|} \hat{V}_\epsilon(t,\xi)) &= \imu e^{-\imu t |\xi|} |\xi| \cF(|u_\epsilon|^2)(t,\xi).
	\endaligned
	\right.
\end{equation}
        We write
        \begin{align}
        	\hat{U}_\epsilon(t)
        &= e^{-\imu t |\xi|^2} \hat{u}_{0,\epsilon} - \imu \int_0^t e^{-\imu (t-s)|\xi|^2} \cF(v_\epsilon u_\epsilon)_{XL}(s) \dd s - \imu \int_0^t e^{-\imu (t-s)|\xi|^2} \cF(v_\epsilon u_\epsilon)_{R}(s) \dd s \nonumber\\
 		&\qquad  + \imu \int_0^t e^{-\imu (t-s)|\xi|^2}\cF(b_\epsilon \cdot \nabla u_\epsilon + c_\epsilon u_\epsilon - (\cT_s(W_2))_\ve u_\ve)(s) \dd s \nonumber\\
 			&= e^{-\imu t |\xi|^2} \hat{u}_{0,\epsilon}
              - \imu \int_0^t e^{-\imu (t-s)|\xi|^2} \cF(V_\epsilon U_\epsilon)_{XL}(s) \dd s
 		   + \imu \int_0^t e^{-\imu (t-s)|\xi|^2} \cF(V_\epsilon U_\epsilon - v_\epsilon u_\epsilon)_{XL}(s) \dd s \notag \\
           &\qquad - \imu \int_0^t e^{-\imu (t-s)|\xi|^2} (v_\epsilon u_\epsilon)_{R}(s) \dd s
              + \imu \int_0^t e^{-\imu (t-s)|\xi|^2}\cF(b_\epsilon \cdot \nabla u_\epsilon + c_\epsilon u_\epsilon -(\cT_s(W_2))_\ve u_\ve )(s) \dd s \label{eq:FTUeps}
        \end{align}
        Because of the regularity of $U_\epsilon$ and $V_\epsilon$,
        the integration-by-parts argument in~\eqref{II-normform}
        is rigorously justified here and we obtain
        \begin{align}  \label{intUV-ve}
        	 & - \imu \int_0^t e^{-\imu (t-s)|\xi|^2} \cF(V_\epsilon U_\epsilon)_{XL}(s,\xi) \dd s  \notag \\
        	&=  - \cF \bdyop(V_\epsilon,U_\epsilon)(t,\xi) + e^{-\imu t |\xi|^2} \cF\bdyop(v_{0,\epsilon},u_{0,\epsilon})(\xi)
 				  + \imu \cF \Big(\int_0^t e^{\imu (t-s)\Delta} \bdyop(|\nabla| |u_\epsilon|^2, U_\epsilon)(s) \dd s\Big)(\xi)  \\
 			&\qquad + \imu \cF\Big(\int_0^t e^{\imu(t-s)\Delta} \bdyop(V_\epsilon, - v_\epsilon u_\epsilon + b_\epsilon \cdot \nabla u_\epsilon
              + c_\epsilon u_\epsilon - (\cT_\cdot(W_2))_\ve u_\ve )(s) \dd s \Big)(\xi). \notag
        \end{align}
        Inserting this formula back into~\eqref{eq:FTUeps}, we thus arrive at
        \begin{align}  \label{Uve-normal}
        	U_\epsilon(t) &= e^{\imu t \Delta} u_{0,\epsilon} -\bdyop(V_\epsilon,U_\epsilon)(t) + e^{\imu t \Delta} \bdyop(v_{0,\epsilon},u_{0,\epsilon}) + \imu \int_0^t e^{\imu (t-s)\Delta} (V_\epsilon U_\epsilon - v_\epsilon u_\epsilon)_{XL}(s) \dd s  \notag  \\
         &\qquad - \imu \int_0^t e^{\imu(t-s)\Delta}(v_\epsilon u_\epsilon)_R(s) \dd s
          + \imu \int_0^t e^{\imu (t-s)\Delta}(b_\epsilon \cdot \nabla u_\epsilon + c_\epsilon u_\epsilon - (\cT_s(W_2))_\ve u_\ve )(s) \dd s  \notag  \\
         &\qquad + \imu \int_0^t e^{\imu (t-s)\Delta} \bdyop(|\nabla| |u_\epsilon|^2, U_\epsilon)(s) \dd s  \notag \\
         &\qquad + \imu \int_0^t e^{\imu(t-s)\Delta} \bdyop(V_\epsilon, - v_\ve u_\ve + b_\epsilon \cdot \nabla u_\epsilon
                + c_\epsilon u_\epsilon - (\cT_\cdot(W_2))_\ve u_\ve )(s) \dd s.
        \end{align}

In order to pass to the limit on the right-hand side above,
arguing as in~\eqref{eq:EstBdyOpInitial} we derive
the boundary estimate in $H^{-1}$
        \begin{align}
	&\| \bdyop(w_1, w_2)\|_{H^{-1}} \lesssim \Big(\sum_{N \in 2^\Z} \|\langle \nabla \rangle^{-1} \bdyop(P_N w_1, P_{\leq K^{-1} N} w_2)\|_{L^2}^2 \Big)^{\frac{1}{2}} \nonumber\\
	&\lesssim \Big(\sum_{N \in 2^\Z} \Big(\sum_{N_1 \leq K^{-1} N} \langle N \rangle^{-2} N^{-1}
         \||\nabla| \langle \nabla \rangle \bdyop(P_N w_1, P_{N_1}w_2)\|_{L^2}\Big)^2 \Big)^{\frac{1}{2}} \nonumber\\
	&\lesssim \Big(\sum_{N \in 2^\Z} \Big(\sum_{N_1 \leq K^{-1} N} \langle N \rangle^{-2} N^{-1} \| P_N w_1 \|_{L^2} \|P_{N_1} w_2\|_{L^\infty}\Big)^2 \Big)^{\frac{1}{2}} \nonumber\\
	&\lesssim \|w_1\|_{H^{-\frac{1}{2}}} \Big(\sum_{N \in 2^\Z} \Big(\sum_{N_1 \leq K^{-1} N} \langle N \rangle^{-\frac{3}{2}} N^{-1}N_1^{\frac{3}{2}} \langle N_1 \rangle  \|\langle N_1 \rangle^{-1} P_{N_1} w_2 \|_{L^2}\Big)^2 \Big)^{\frac{1}{2}} \nonumber\\
	& \lesssim \|w_1\|_{H^{-\frac{1}{2}}} \Big(\sum_{N \in 2^\Z} \Big(\sum_{N_1 \leq K^{-1} N} N^{-1} N_1 \|\langle N_1 \rangle^{-1} P_{N_1} w_2\|_{L^2} \Big)^2 \Big)^{\frac{1}{2}} \lesssim K^{-1} \|w_1\|_{H^{-\frac{1}{2}}} \|w_2\|_{H^{-1}}. \nonumber
\end{align}

Consequently, employing \eqref{uve-u-vve-v}, \eqref{uve-vve-asmp-H-1}, \eqref{Uve-limit}, and \eqref{Vve-limit},
we first obtain $(V_\ve U_\ve - v_\ve u_\ve)_{XL} \to 0$ in $L^q(0,\tau; H^{-3})$ for any $q \in [1,2]$ and then
\begin{align*}
	U_\epsilon(t)
         &\longrightarrow e^{\imu t \Delta} u_{0} -\bdyop(v,u)(t)
             + e^{\imu t \Delta} \bdyop(v_{0},u_{0})
             - \imu \int_0^t e^{\imu(t-s)\Delta}(v u)_R(s) \dd s  \\
         &\qquad  + \imu \int_0^t e^{\imu (t-s)\Delta}(b \cdot \nabla u + c u - \cT_\cdot(W_2)u)(s) \dd s
                  + \imu \int_0^t e^{\imu (t-s)\Delta} \bdyop(|\nabla| |u|^2, u)(s) \dd s  \\
         &\qquad + \imu \int_0^t e^{\imu(t-s)\Delta} \bdyop(v, - vu+ b \cdot \nabla u + c u - \cT_\cdot(W_2)u )(s) \dd s
\end{align*}
in $L^q(0,\tau; H^{-3})$ for any $q\in [1,2]$
as $\epsilon \rightarrow 0$.
As $U_\epsilon$ also converges to $u$ in $L^q(0,\tau;H^{-1})$ by~\eqref{Uve-limit}, we conclude
\begin{align*}
	u(t) &=  e^{\imu t \Delta} u_{0} -\bdyop(v,u)(t) + e^{\imu t \Delta} \bdyop(v_{0},u_{0}) - \imu \int_0^t e^{\imu(t-s)\Delta}(v u)_R(s) \dd s \\
         &\qquad  + \imu \int_0^t e^{\imu (t-s)\Delta}(b \cdot \nabla u + c u - \cT_\cdot(W_2) u)(s) \dd s
           + \imu \int_0^t e^{\imu (t-s)\Delta} \bdyop(|\nabla| |u|^2, u)(s) \dd s   \\
         &\qquad   + \imu \int_0^t e^{\imu(t-s)\Delta} \bdyop(v, - vu+ b \cdot \nabla u + c u - \cT_\cdot(W_2) u )(s) \dd s,
\end{align*}
i.e. $(u,v)$ solves the normal form formulation~\eqref{eq:RZakharovNormalForm}.

\ref{it:NormalFormMild}
Let $(u,v) \in \X(0,\tau) \times \Y(0,\tau)$ be a solution of the normal form~\eqref{eq:RZakharovNormalForm}.
As in part~\ref{it:MildNormalForm},
we apply a mollifier in space to obtain smooth functions
$u_\epsilon, v_\epsilon, b_\epsilon, c_\epsilon, (\cT_\cdot(W_2))_\ve \in C([0,\tau] \times H^k(\R^3))$ for all $k \in \N$,
and the convergence results in~\eqref{uve-u-vve-v} and \eqref{uve-vve-asmp-H-1} are still valid.
Let $u_{0,\epsilon}, v_{0,\epsilon} \in H^k$,  $k \in \N$,
 be the mollification of $u_0$ and $v_0$.
		
We keep using the notations $U_\ve$ and $V_\ve$ as in part~\ref{it:MildNormalForm} for the approximations of the mild solutions, i.e. $U_\epsilon$ and $V_\epsilon$ are defined by~\eqref{eq:DefUeps} and~\eqref{eq:DefVeps}.
Moreover, we denote by $\wt U_\ve$ the approximation of the normal form solution
		\begin{align*}
			\wt U_\epsilon(t) &:= e^{\imu t \Delta} u_{0,\epsilon} -\bdyop(v_\epsilon,u_\epsilon)(t) + e^{\imu t \Delta} \bdyop(v_{0,\epsilon},u_{0,\epsilon}) - \imu \int_0^t e^{\imu (t-s)\Delta} (v_\epsilon u_\epsilon)_{R}(s) \dd s \\
         &\qquad + \imu \int_0^t e^{\imu (t-s)\Delta}(b_\epsilon \cdot \nabla u_\epsilon + c_\epsilon u_\epsilon - (\cT_\cdot(W_2))_\ve u_\ve )(s) \dd s
             + \imu \int_0^t e^{\imu (t-s)\Delta} \bdyop(|\nabla| |u_\epsilon|^2, u_\epsilon)(s) \dd s\\
         &\qquad  + \imu \int_0^t e^{\imu(t-s)\Delta} \bdyop(v_\epsilon, - v_\ve u_\ve
                  + b_\epsilon \cdot \nabla u_\epsilon + c_\epsilon u_\epsilon
                  - (\cT_\cdot(W_2))_\ve u_\ve )(s) \dd s.
		\end{align*}
		
		We write
		\begin{align}
		\label{eq:NormalToMildSplitting}
			&\imu \int_0^t e^{\imu (t-s)\Delta} \bdyop(|\nabla| |u_\epsilon|^2, u_\epsilon)(s) \dd s
              + \imu \int_0^t e^{\imu(t-s)\Delta} \bdyop(v_\epsilon, -v_\epsilon u_\epsilon
                + b_\epsilon \cdot \nabla u_\epsilon + c_\epsilon u_\epsilon - (\cT_\cdot(W_2))_\ve u_\ve )(s)   \dd s \nonumber\\
			&= \imu \int_0^t e^{\imu (t-s)\Delta} \bdyop(|\nabla| |u_\epsilon|^2, {U}_\epsilon)(s) \dd s
             + \imu \int_0^t e^{\imu(t-s)\Delta} \bdyop(V_\epsilon, - v_\epsilon u_\epsilon
               + b_\epsilon \cdot \nabla u_\epsilon + c_\epsilon u_\epsilon - (\cT_\cdot(W_2))_\ve u_\ve )(s) \dd s \nonumber\\
			&\quad + \imu \int_0^t e^{\imu (t-s)\Delta} \bdyop(|\nabla| |u_\epsilon|^2, u_\epsilon -  {U}_\epsilon)(s) \dd s \nonumber\\
			&\quad + \imu \int_0^t e^{\imu(t-s)\Delta} \bdyop(v_\epsilon - V_\epsilon, - v_\epsilon u_\epsilon
            + b_\epsilon \cdot \nabla u_\epsilon + c_\epsilon u_\epsilon - (\cT_\cdot(W_2))_\ve u_\ve)(s) \dd s \notag \\
            &=: I + II + III + IV.
		\end{align}
        Note that, as in part~\ref{it:MildNormalForm}, equation~\eqref{equa-Uve-Vve} is still valid.
		Hence, by \eqref{intUV-ve} we have
		\begin{align*}
			\cF(I + II)
             = \cF \bdyop(V_\epsilon, {U}_\epsilon)(t)
                 - \cF(e^{\imu t \Delta} \bdyop(v_{0,\epsilon}, u_{0,\epsilon}))
                 - \imu \cF \int_0^t e^{\imu (t-s) \Delta} (V_\epsilon {U}_\epsilon)_{XL}(s) \dd s.
		\end{align*}
		Inserting this identity together with~\eqref{eq:NormalToMildSplitting} into the definition of $\wt U_\epsilon$, we obtain
		\begin{align*}
			\wt U_\epsilon(t)
          &= e^{\imu t \Delta} u_{0,\epsilon} - (\bdyop(v_\epsilon,u_\epsilon)(t)
         - \bdyop(V_\epsilon, {U}_\epsilon)(t)) - \imu \int_0^t e^{\imu (t-s)\Delta} (v_\epsilon u_\epsilon)(s) \dd s \\
         &\qquad + \imu \int_0^t e^{\imu (t-s) \Delta} (v_\epsilon u_\epsilon - V_\epsilon {U}_\epsilon)_{XL} \dd s
         + \imu \int_0^t e^{\imu (t-s)\Delta}(b_\epsilon \cdot \nabla u_\epsilon + c_\epsilon u_\epsilon - (\cT_\cdot(W_2))_\ve u_\ve )(s) \dd s \\
         &\qquad + \imu \int_0^t e^{\imu (t-s)\Delta} \bdyop(|\nabla| |u_\epsilon|^2, u_\epsilon - {U}_\epsilon)(s) \dd s\\
         &\qquad + \imu \int_0^t e^{\imu(t-s)\Delta} \bdyop(v_\epsilon - V_\epsilon, - v_\epsilon u_\epsilon
            + b_\epsilon \cdot \nabla u_\epsilon + c_\epsilon u_\epsilon - (\cT_\cdot(W_2))_\ve u_\ve)(s) \dd s.
		\end{align*}
		We thus get
		\begin{align}
		\label{eq:DifferenceUeps}
			\wt U_\epsilon(t) - {U}_\epsilon(t)
            &=  - (\bdyop(v_\epsilon,u_\epsilon)(t) - \bdyop(V_\epsilon, {U}_\epsilon)(t))
             + \imu \int_0^t e^{\imu (t-s) \Delta} (v_\epsilon u_\epsilon - V_\epsilon {U}_\epsilon)_{XL} \dd s  \nonumber\\
         &\qquad + \imu \int_0^t e^{\imu (t-s)\Delta} \bdyop(|\nabla| |u_\epsilon|^2, u_\epsilon - {U}_\epsilon)(s) \dd s \nonumber\\
         &\qquad + \imu \int_0^t e^{\imu(t-s)\Delta} \bdyop(v_\epsilon - V_\epsilon, - v_\epsilon u_\epsilon
          + b_\epsilon \cdot \nabla u_\epsilon + c_\epsilon u_\epsilon - (\cT_\cdot(W_2))_\ve u_\ve )(s) \dd s.
		\end{align}

		 Below we will fix $K'$ and $\tau'$ in such a way that $K' \geq \tilde{K}$ and $\tau' \leq \tilde{\tau} \leq 1$, where $\tilde{K}$
         and $\tilde{\tau}$ are the dyadic integer and the time step constructed in Step~1 and Step~2 of the proof of Theorem~\ref{Thm-LWP}.
         Consequently, $(u,v)$ is the unique fixed point of the fixed point operator $\Phi$ from the proof of
         Theorem~\ref{Thm-LWP} and is an element of the ball $B_{\X(0,\tau) \times \Y(0,\tau)}(M)$ constructed there. In particular,
         there is a radius $M$ (depending only on $\|u_0\|_{H^1}$ and $\|v_0\|_{L^2}$) such that $\|u\|_{\X(0,\tau)} + \|v\|_{\Y(0,\tau)} \leq M$
         on $[0,\tau']$. We thus also have
         $\|u_\epsilon\|_{\X(0,\tau)} + \|v_\epsilon\|_{\Y(0,\tau)} \leq M$
         and by Lemma~\ref{lem:QuadraticWave} $\|V_\epsilon\|_{\Y(0,\tau)} \lesssim_{M} 1$ on $[0,\tau']$.
		
		We want to estimate the difference in~\eqref{eq:DifferenceUeps} in the space
		\begin{align*}
		 S(0,\tau) := L^\infty(0,\tau; L^2) \cap L^2(0,\tau; \dot{B}^0_{6,2}).
		\end{align*}
		To that purpose, we first note that adapting the proofs of Lemma~\ref{lem:Bdy},
estimates~\eqref{eq:EstQuadraticR} and \eqref{eq:EstCubicuuu}, we obtain in the same way the estimates
		\begin{align*}
			&\|\bdyop(w_1,w_2)\|_{S(0,\tau)} \lesssim (K^{-1} + \tau^{\frac{1}{8}}) \|w_1\|_{\Y(0,\tau)} \|w_2\|_{S(0,\tau)}, \\
			&\left\|\int_{0}^\cdot e^{\imu (\cdot-s) \Delta} (w_1 w_2)_{XL}(s) \dd s \right\|_{S(0,\tau)}
               \lesssim \tau^{\frac{1}{4}} \|w_1\|_{\Y(0,\tau)} \|w_2\|_{S(0,\tau)}, \\
			&\left\|\int_{0}^\cdot e^{\imu (\cdot-s) \Delta} \bdyop(|\nabla||w_1|^2, w_2)(s) \dd s \right\|_{S(0,\tau)}
             \lesssim \tau^\frac 58 \|w_1\|^2_{L^\infty(I; H^1_x)} \|w_2\|_{S(0,\tau)}.
		\end{align*}
		We thus infer
		\begin{align*}
			&\|\bdyop(v_\epsilon,u_\epsilon) - \bdyop(V_\epsilon, {U}_\epsilon) \|_{S(0,\tau)} \\
			&\lesssim \|\bdyop(v_\epsilon - V_\epsilon, u_\epsilon)\|_{S(0,\tau)}
               + \| \bdyop(V_\epsilon, u_\epsilon - \wt U_\epsilon) \|_{S(0,\tau)}
               + \|\bdyop(V_\epsilon, \wt U_\epsilon - {U}_\epsilon)\|_{S(0,\tau)} \\
			&\lesssim_{M} \|v_\epsilon - V_\epsilon\|_{\Y(0,\tau)}
                + \|u_\epsilon - \wt U_\epsilon\|_{S(0,\tau)}
                + (K^{-1} + \tau^{\frac{1}{8}}) \|\wt U_\epsilon - {U}_\epsilon\|_{S(0,\tau)},
		\end{align*}
		exploiting that $\|u_\epsilon\|_{\X(0,\tau)} + \|V_\epsilon\|_{\Y(0,\tau)} \lesssim_{M} 1$. Analogously, we deduce
		\begin{align*}
			&\left\|  \int_0^\cdot e^{\imu (\cdot-s) \Delta} (v_\epsilon u_\epsilon - V_\epsilon {U}_\epsilon)_{XL} \dd s \right\|_{S(0,\tau)} \\
			&\lesssim \tau^{\frac{1}{4}} (\|v_\epsilon - V_\epsilon\|_{\Y(0,\tau)} \| u_\epsilon\|_{S(0,\tau)}
                + \|V_\epsilon\|_{\Y(0,\tau)} \|u_\epsilon - \wt U_\epsilon\|_{S(0,\tau)}
                + \|V_\epsilon\|_{\Y(0,\tau)} \|\wt U_\epsilon - {U}_\epsilon\|_{S(0,\tau)}) \\
			&\lesssim_{M} \|v_\epsilon - V_\epsilon\|_{\Y(0,\tau)}
               + \|u_\epsilon - \wt U_\epsilon\|_{S(0,\tau)}
               + \tau^{\frac{1}{4}} \|\wt U_\epsilon - {U}_\epsilon\|_{S(0,\tau)},
		\end{align*}
		and
		\begin{align*}
			 \left\| \int_0^\cdot e^{\imu (\cdot-s)\Delta}
          \bdyop(|\nabla| |u_\epsilon|^2, u_\epsilon - {U}_\epsilon)(s) \dd s \right\|_{S(0,\tau)}
          &\lesssim \tau^\frac 58 \|u_\epsilon\|^2_{L^\infty_t H^1_x} \|u_\epsilon - {U}_\epsilon\|_{S(0,\tau)} \\
			&\lesssim_{M} \|u_\epsilon - \wt U_\epsilon\|_{S(0,\tau)} + \tau^{\frac{5}{8}} \|\wt U_\epsilon - {U}_\epsilon\|_{S(0,\tau)}.
		\end{align*}
		For the last term in~\eqref{eq:DifferenceUeps} we apply Lemma~\ref{lem:Cubic} to obtain
		\begin{align*}
			&\left\| \int_0^\cdot e^{\imu(\cdot-s)\Delta} \bdyop(v_\epsilon - V_\epsilon, v_\epsilon u_\epsilon
             - b_\epsilon \cdot \nabla u_\epsilon - c_\epsilon u_\epsilon
              - (\cT_\cdot(W_2))_\ve u_\ve )(s) \dd s \right\|_{S(0,\tau)} \\
			&\lesssim \tau^{\frac{1}{8}} \|v_\epsilon - V_\epsilon\|_{\Y(0,\tau)}
              (\|v_\epsilon\|_{\Y(0,\tau)}  + \|b_\epsilon\|_{L^\infty(0,\tau; H^1_x)}
                 + \|c_\epsilon\|_{\Y(0,\tau)}
                 +  \|(\cT_\cdot(W_2))_\ve\|_{\Y(0,\tau)})\|u_\epsilon\|_{\X(0,\tau)}   \\
			&\lesssim_{M} \|v_\epsilon - V_\epsilon\|_{\Y(0,\tau)},
		\end{align*}
		where we also used
		\begin{align*}
			\|b_\epsilon\|_{L^\infty(0,\tau; H^1_x)}
           + \|c_\epsilon\|_{\Y(0,\tau)}
           + \|(\cT(W_2))_\ve\|_{\Y(0,\tau)}
           \lesssim \|b\|_{L^\infty(0,\tau; H^1_x)} + \|c\|_{\Y(0,\tau)} + \|\cT(W_2)\|_{\Y(0,\tau)} \lesssim 1.
		\end{align*}
		Combining these estimates with~\eqref{eq:DifferenceUeps}, we arrive at
		\begin{align*}
			\|\wt U_\epsilon - {U}_\epsilon\|_{S(0,\tau)}
        \leq C_0 (K^{-1} + \tau^{\frac{1}{8}}) \|\wt U_\epsilon - {U}_\epsilon\|_{S(0,\tau)}
        + C_0 (\|u_\epsilon - \wt U_\epsilon\|_{S(0,\tau)} + \|v_\epsilon - V_\epsilon\|_{\Y(0,\tau)}).
		\end{align*}
		Now we fix $K' \in 2^\N$
        and $\tau'  > 0$ such that $K' \geq \tilde{K}$, $\tau' \leq \tilde{\tau}$, and
		\begin{align*}
			C_0 (K'^{-1} + \tau'^{\frac{1}{8}}) \leq \frac{1}{2}.
		\end{align*}
		Consequently, we arrive at
		\begin{equation}
		\label{eq:DifferenceUepsEst}
			\|\wt U_\epsilon - {U}_\epsilon\|_{S(0,\tau)} \leq 2 C_2 (\|u_\epsilon - \wt U_\epsilon\|_{S(0,\tau)}
             + \|v_\epsilon - V_\epsilon\|_{\Y(0,\tau)})
		\end{equation}
		if $K \in 2^\N$ with $K \geq K'$ and $\tau \in (0,\tau']$.
	
	Since $v_\epsilon \rightarrow v$ in $\Y(0,\tau)$ and
$V_\epsilon \rightarrow v$ in $\Y(0,\tau)$ by Lemma~\ref{lem:QuadraticWave},
we have
\begin{align} \label{vve-Vve-limit}
     \|v_\epsilon - V_\epsilon\|_{\Y(0,\tau)} \longrightarrow 0,\ \  \text{as}\ \epsilon \rightarrow 0.
\end{align}
Using that $b_\epsilon \rightarrow b$ and $c_\epsilon \rightarrow c$ in $L^\infty(0,\tau;H^1)$ as $b,c \in C([0,\tau];H^1)$ and that
 $u_\epsilon \rightarrow u$ in $\X(0,\tau)$, we thus obtain
	\begin{align} \label{wtUve-u-limit}
		\wt U_\epsilon \longrightarrow u\ \  \text{in } S(0,\tau), \ \ \text{as}\ \ve\to 0,
	\end{align}
which implies
\begin{align}  \label{uve-wtUve-limit}
      u_\ve - \wt U_\ve \longrightarrow 0\ \ \text{in}\ S(0,\tau),\ \text{as}\ \ve\to 0.
\end{align}
 Hence, we obtain from \eqref{eq:DifferenceUepsEst}, \eqref{vve-Vve-limit} and \eqref{uve-wtUve-limit} that
	\begin{align} \label{Uve-wtUve-limit}
		 \wt U_\epsilon - {U}_\epsilon  \longrightarrow 0\ \ \text{in}\ S(0,\tau),\ \text{as}\ \ve\to 0.
	\end{align}

Using~\eqref{Uve-limit}, \eqref{wtUve-u-limit}, and \eqref{Uve-wtUve-limit}, we thus conclude
	\begin{align*}
		u(t) = e^{\imu t \Delta} u_{0} - \imu  \int_0^t e^{\imu (t-s) \Delta}( v u - b \cdot \nabla u - c u + \cT_\cdot(W_2)u)(s) \dd s
	\end{align*}
	in $H^{-1}$ for all $t \in [0,T]$, i.e., $(u,v)$ is a mild solution of~\eqref{eq:RZakharovDuhamel}.
	
	\ref{it:NormalFormIndK} Part~\ref{it:NormalFormIndK} is now an immediate consequence of parts~\ref{it:MildNormalForm} and~\ref{it:NormalFormMild}.
Therefore, the proof is complete.
	\end{proof}

\section*{Acknowledgements}

We thank the anonymous referees for their helpful remarks, which, in particular, led to an improvement of the regularity assumption in Theorem~\ref{Thm-Noise-Reg}.

Funded by the Deutsche Forschungsgemeinschaft (DFG, German Research Foundation) – SFB 1283/2 2021 – 317210226.
D. Zhang  is also grateful for the support by NSFC (No. 12271352, 12322108)
and Shanghai Rising-Star Program 21QA1404500.

\end{document}